\documentclass[11pt,reqno,a4paper]{amsart}

\usepackage{amsmath,amssymb,amsthm,mathtools}
\usepackage{tikz-cd}

\usepackage{mathrsfs}
\usepackage{booktabs,array}

\usepackage[margin=2.6cm]{geometry}
\usepackage{enumitem}
\usepackage{needspace}
\usepackage{xcolor}
\usepackage{todonotes}
\usepackage[colorlinks=true,linkcolor=blue!55!black,citecolor=blue!55!black,urlcolor=blue!55!black,bookmarks=true,bookmarksnumbered=true]{hyperref}
\makeatletter
\@namedef{subjclassname@2020}{%
  \textup{2020} Mathematics Subject Classification}
\makeatother

\newif\ifshowchanges
\showchangesfalse
\newcommand{\revblue}[1]{\ifshowchanges{\color{blue}#1}\else#1\fi}
\newcommand{\revsection}[2]{\section{\texorpdfstring{\revblue{#1}}{#2}}}
\newcommand{\revsubsection}[2]{\subsection{\texorpdfstring{\revblue{#1}}{#2}}}

\newcommand{\Fq}{\mathbb{F}_q}

\newcommand{\TT}{\mathbb{T}}
\newcommand{\OO}{\mathcal{O}}

\newcommand{\ii}{\mathfrak{i}}
\newcommand{\cM}{\mathcal{M}}
\newcommand{\GM}{\operatorname{GM}}
\newcommand{\cB}{\mathcal{B}}

\newcommand{\Om}{\Omega^1}

\DeclareMathOperator{\supp}{supp}
\DeclareMathOperator{\divv}{div}

\DeclareMathOperator{\id}{id}
\newcommand{\Exp}{\exp_\varphi}

\newcommand{\Pic}{\operatorname{Pic}}

\newcommand{\tor}{\tau}

\newcommand{\Lie}{\operatorname{Lie}}
\newcommand{\cont}{\operatorname{cont}}

\newcommand{\sfm}{\operatorname{sf}}

\newcommand{\LL}{\mathcal{L}}

\newcommand{\MM}{\mathcal{M}}

\newcommand{\two}{ \mathbf{I\!I}}
\newcommand{\one}{  \mathbf{I}}

\newcommand{\Cinf}{\mathbb{C}_\infty}

\newcommand{\res}{\operatorname{res}} 
\newcommand{\rat}{\operatorname{rat}}

\newcommand{\mult}{\mathrm{mult}}

\newcommand{\Ninf}{\mathrm{N}}
\newcommand{\rd}{\mathrm{d}}
\newcommand{\rdt}{\rd t}
\newcommand{\rdu}{\rd u}

\theoremstyle{plain}
\newtheorem{theorem}{Theorem}[section]
\newtheorem{proposition}[theorem]{Proposition}
\newtheorem{lemma}[theorem]{Lemma}
\newtheorem{notation}[theorem]{Notation}
\newtheorem{corollary}[theorem]{Corollary}

\newtheorem{lettertheorem}{Theorem}

\theoremstyle{definition}
\newtheorem{definition}[theorem]{Definition}
\newtheorem{example}[theorem]{Example}

\theoremstyle{remark}
\newtheorem{remark}[theorem]{Remark}

\begin{document}

\title[Pellarin's identity, Carlitz period, and Anderson generating functions]{Pellarin's
Identity, Carlitz Period, and Anderson Generating Functions over Arbitrary Curves}

\author{Chuangqiang Hu}
\address{School of Mathematics, Sun Yat-Sen University, Guangzhou, 510275, P.~R.~China}
\email{\href{huchq5@mail.sysu.edu.cn}{huchq5@mail.sysu.edu.cn}}

\author{Stephen S.-T. Yau$^*$}
\address{Department of Mathematical sciences, Tsinghua University, Beijing, 100084, P.~R.~China}
\email{\href{yau@uic.edu}{yau@uic.edu}}

\author{Lishan Yu}
\address{Beijing Institute of Mathematical Sciences and Applications, Beijing, 101408, P.~R.~China}
\email{\href{lishany@bimsa.cn}{lishany@bimsa.cn}}

\begin{abstract}
\revblue{We establish three results for a rank-one Drinfeld module over the coefficient ring of a smooth, projective, geometrically connected curve over \(\mathbb F_q\), with a closed point at infinity of arbitrary degree. First, a Tate-algebra unit exists in the module of special functions precisely when this module is free of rank one, that is, when the period lattice is isomorphic to the module of regular differentials; this proves the Gazda--Maurischat conjecture. Second, under the unit-anchor normalization for an isogeny to a principal-period target, a strictly normalized shtuka product and one characteristic residue give the fundamental period; the same defect equation gives the logarithm residue formula. Third, finite \'{e}tale torsion traces over the field of definition identify the first Frobenius twist of the normalized shtuka differential and express Pellarin's \(\mathcal L(1)\)-series at a fundamental period. The source period lattice need not be free.}
\end{abstract}


\maketitle
\enlargethispage{7pt}

\setcounter{tocdepth}{2}
\tableofcontents
\clearpage

\section{Introduction}\label{sec:intro}

\subsection{Background and aim}\label{sec:background}

 Let $X/\Fq$ be a smooth, projective, geometrically connected curve
of genus $g$ over the finite field $\Fq$ of characteristic $p$, let
$\infty$ be a closed point of degree $\Ninf$, and put
$U=X\setminus\{\infty\}$, $A=\Gamma(U,\OO_X)$, $K=\Fq(X)$.  Write
$K_\infty$ for the completion of $K$ at $\infty$ and
$\Cinf=\widehat{\overline{K_\infty}}$ for its completed algebraic
closure. Fix an isometric $\Fq$-embedding
$\iota\colon K_\infty\hookrightarrow\Cinf$, whose restriction to $A$
defines the characteristic point $\Xi\in X(\Cinf)$.  Throughout,
$\varphi$ denotes a rank-one Drinfeld $A$-module over $\Cinf$,
$\exp_\varphi$ its exponential, and
$\Lambda_\varphi:=\ker(\exp_\varphi)\subset\Cinf$ its period lattice, a
discrete projective rank-one $A$-module. 

For a rational infinite point and a fixed sign function,
Goss' theory relates arithmetic special values to powers of
fundamental periods: for positive $n$ divisible by $q-1$,
\[
\zeta_A(n)=\alpha_n\,\tilde\pi^n,\qquad  \alpha_n\in VH,
\]
by Goss' theorem~\cite{Gos96}, where $VH$ is the
composition field of \cite[Theorem~4.6]{Pel26} (for $A=\Fq[t]$ one has
$VH=K$).  Here $\zeta_A(n)$ is the $n$-th
special value of Goss' zeta function of $A$. Notably, $\tilde\pi$ is a
fundamental period, i.e., a chosen generator of the principal
period lattice of the sign-normalized rank-one Drinfeld module used
in this formula, and is well defined up to $\Fq^\times$. For general
rank-one Drinfeld modules the period lattice is only projective, and
this paper presupposes no such generator.  For $A=\Fq[t]$ such evaluations go back to Carlitz
\cite{Car35}, who expressed $\zeta_{\Fq[t]}(q-1)$ as a rational
multiple of $\tilde\pi^{q-1}$, where
\[ 
\tilde\pi=\theta(-\theta)^{1/(q-1)}
 \prod_{i>0}\bigl(1-\theta/\theta^{q^i}\bigr)^{-1}, \qquad  \theta:=\iota(t),
 \]
is the Carlitz period (already transcendental over $\Fq(\theta)$ by
Wade \cite{Wad41}), and to Anderson--Thakur \cite{AT90}, whose tensor powers of the
Carlitz module give explicit formulas for every $n$. Explicit formulas
for elliptic curves are proved in \cite{GP18}, and for rational
function fields with a higher-degree point at infinity in
\cite{HHY26}.  In full
generality, however, it is hard to determine explicit formulas for the
fundamental period $\tilde\pi$ \cite[\S4]{Pel26}.

 A second viewpoint is geometric interpolation.  Pellarin's zeta
functions over curves \cite{Pel26} are functions of a characteristic point
$\chi\in X(\Cinf)\setminus\{\infty\}$, which corresponds to an
$\Fq$-algebra homomorphism $A\to\Cinf$.  The function $Z_{X,A}$ is entire away from
$\infty$ and has an essential singularity there.  Its values at suitable points
recover special values of the Goss--Angl\`es functions \cite[\S5.2.1]{Pel26}.    

At a rational infinite point, Ferraro's
theorem~\cite[Theorem~7.27]{Fer24} and its curve-zeta form in
\cite[Theorem~5.3]{Pel26} states that 
\[
\omega^{(1)}Z_{X,A}=\tilde\pi\,\delta^{(1)}.
\] The
function $f$ is the unique shtuka function with divisor
$V^{(1)}-V+[\Xi]-[\infty]$, where $V$ is a degree-$g$ divisor with $h^0(V)=1$.
The special function $\omega\in\Cinf\widehat\otimes_{\Fq}A$ is nonzero and
satisfies $\omega^{(1)}=f\omega$. Finally, the unique rational factor $\delta$ belongs
to $A\otimes_{\Fq}H$, where $H$ is the Hilbert class field of $A$.   

 Both factors $\omega^{(1)}$ and $Z_{X,A}$ are transcendental and have an essential
singularity at $\infty$. Their product, however, is rational: 
\[
\divv(\delta)=V+V^*-2g[\infty],
\]
where $V^*$ is the dual shtuka divisor.  Its twisted support plays the
role of the non-trivial zeros of $Z_{X,A}$.  The twisted characteristic points
$\Xi^{(i)}$, $i\geqslant1$, are the trivial zeros.  Thus
$\tilde\pi\,\delta^{(1)}$ is a function-field analogue of Riemann's
$\xi$-function.  The algebraic factor $\delta$ is again difficult to
compute in general; for $X=\mathbb P^1$ one has $\delta=-1$, and the
identity reads
$\omega^{(1)}Z_{\mathbb P^1,\Fq[\theta]}=-\tilde\pi$.

 The aim of this paper is to prove the Gazda--Maurischat unit
conjecture (Theorem~\ref{cor:GM-unit-equivalence}), to derive the corresponding
strictly normalized product and residue formula for the fundamental period
(Theorem~\ref{thm:explicit-period-lettered}), and to compute the rational
differential factor in the Pellarin $L(1)$-identity by finite torsion traces
(Theorem~\ref{thm:finite-trace-lettered}). 
The construction uses Anderson generating functions, which connect
exponential uniformization with Frobenius difference equations and
arithmetic special values \cite{And86,AT90,Pel12,AP15,GM21}.
 On a general curve, the denominator of the classical generating
series is replaced by a two-variable Cauchy differential whose diagonal
residue is prescribed.  Evaluation in the tautological motive frame gives
the period differential.  The Frobenius defect then produces the first
twist of the normalized shtuka differential, which the finite traces identify. 

\subsection{Tautological coordinates and intrinsic evaluation}
\label{sec:levels-intro}

Let $\varphi$ be a rank-one Drinfeld module with period lattice
$\Lambda_\varphi$ and shtuka function $f_\varphi$.  Its Drinfeld divisor
$V_\varphi$ is effective of degree $g$ and satisfies $h^0(V_\varphi)=1$.

For an acyclic divisor line $\LL=\OO(D)$, meaning
$h^0(\LL)=h^1(\LL)=0$, write the normalized Szeg\H o
section in its tautological rational frame as
\[
S_{\LL}=K_{\LL}(\one,\two)e_{\LL}(\one)\otimes e_{\LL}(\two)^{-1}.
\]
Here $K_{\LL}$ is a differential in the second variable; the
placeholders $\one$ and $\two$ mark the two factors of $X\times X$, and
$\Delta$ denotes the diagonal. It has diagonal
residue one and divisor bound
\[ 
\divv(K_{\LL})+D_{\one}-D_{\two}+\Delta\geqslant0,\]
where
$D_\one$ and $D_\two$ are the pullbacks of $D$ to the two factors.
 Motive evaluation acts on $K_{\LL}s_\varphi$ in the chosen
rational frame.  The construction uses the divisor, residue, and
Frobenius defect of this coordinate.
For the shtuka divisor line $\LL_\psi$, respectively
$\LL_\varphi$, the notation
$K_\psi:=K_{\LL_\psi}$, respectively $K_\varphi$, is used in
the finite formulas below.

Section~\ref{sec:evaluation} constructs rational, completed-tensor, and
characteristic-orbit evaluation. The latter is defined by a normally
convergent series of rational motive values on suitable affinoids. These
descriptions agree on their common domains.

\subsection{The main construction of Anderson generating functions}

The Tate domain and the characteristic point have different roles.
A general Tate function cannot be
evaluated at $\Xi$, while the kernel construction gives the meromorphic
continuation needed to take a characteristic residue.

We use four forms of the same period information:
\begin{enumerate}[label=\textup{(\roman*)},leftmargin=2.2em]
\item The canonical differential $\mathcal H_\varphi(z)$ is the evaluation
of $K_\varphi$. At a period it agrees with the differential-valued
Gazda--Maurischat transform $\GM_u^\Omega$ (\cite{GM24}) and the evaluation  of every
$V_\varphi$-bounded Cauchy kernel $K\in\mathscr C_\varphi^\Omega$ from Theorems~\ref{thm:Vphi-Cauchy-fiber-invariance} and \ref{thm:canonical-assembly}.
Precisely, for $\lambda\in\Lambda_\varphi$, there is a common value
\[
\mathcal H_\varphi(\lambda)= \mathcal{E}_{\varphi,\lambda}(K s_\varphi)=\GM_u^\Omega(\lambda),
\]
where $\mathcal{E}_{\varphi}$ denotes motive evaluation with $s_\varphi$ the
tautological motive frame (Section~\ref{sec:evaluation}).
Applied after meromorphic continuation, Theorem~\ref{thm:identity-dual-factorization} implies the residue identification\revblue{:}
 \[
 \operatorname{res}_\Xi\mathcal H_\varphi(\lambda)=\lambda.
 \]


\item For a separable isogeny datum
$J=[\eta_J;h_J;\ii_J,\mu_J]:\varphi\to\psi$ as in
Notation~\ref{not:isogeny-datum-main}, the multiplier $h_J$ is defined by
the motive frames, and the ideal $\ii_J$ is the annihilator ideal of the
isogeny. The isogeny-adapted differential $\mathcal G_J(z)$ satisfies the
inhomogeneous Frobenius equation\revblue{:}
\begin{equation}\label{eq:intro-onestep}
 \mathcal G_J(z)^{(1)}-f_\varphi\mathcal G_J(z)
 =-\frac{\widetilde\lambda_\psi^{(1)}}{h_J^{(1)}}
       \exp_\psi(h_J(\Xi)z),
 \qquad \operatorname{res}_\Xi\mathcal G_J(z)=z.
\end{equation}
(See
Theorems~\ref{thm:frob} and \ref{thm:primitive-G}, and Corollary~\ref{cor:H-source-lattice-recovery}.)
Here $\exp_\psi$ is the exponential of the target module $\psi$,
and $\widetilde\lambda_\psi^{(1)}$ is the first Frobenius twist of its
normalized shtuka differential (Definition~\ref{def:shtuka-diff}).

\item For an $A$-homomorphism $\jmath:\Omega_A\to\Lambda_\varphi$,
residue duality constructs a scalar special function $\omega_\jmath$.  This realizes the Gazda--Maurischat correspondence
\cite{GM21,GM24}, with inverse given by the stabilized local logarithm:
\[
 \operatorname{Hom}_A(\Omega_A,\Lambda_\varphi)
 \simeq\operatorname{sf}(\varphi),\qquad
 \jmath\longmapsto\omega_\jmath .
\]
(Theorem~\ref{thm:direct-logarithmic-special-line}.)
For every $\eta\in\Omega_A$,
$\mathcal H_\varphi(\jmath(\eta))=\omega_\jmath\otimes\eta$ (Theorem~\ref{thm:identity-dual-factorization}); it is the
common scalar coordinate of the period-valued families, and it is a
Tate-algebra unit whenever $\jmath$ is an isomorphism.
For every unit anchor, this unit is explicitly realized as
$\omega_\jmath=\delta^{\mathrm{GM}}_u(\jmath(\rdu))$
with $u\in A$ any nonconstant separating element
(Proposition~\ref{prop:GM-residue-dual-coordinate}).
Concretely, if
$\Lambda_\psi=\iota(A)\tilde{\pi}$ and $\Omega_A=\ii_J\nu_J$, the isogeny
anchor
\[
 \jmath_J(a\nu_J)=\iota(a)\frac{\tilde{\pi}}{h_J(\Xi)}
 \quad(a\in\ii_J)
\]
supplies the period isomorphism used in practice.
\item 
Let $\widetilde\lambda_\psi=\alpha\nu$ with
$0\ne\nu\in\Omega_A$. Denote by
$f_\psi^{(i)}$ the $i$-th coefficient Frobenius twist of
$f_\psi$, and
$F_\psi^{[i]}:=f_\psi f_\psi^{(1)}\cdots f_\psi^{(i-1)}$.
The deformed logarithm
\[
\operatorname{Log}_{\alpha,f_\psi}(\two;Z)
 :=Z+\sum_{i\geqslant1}
 \frac{\alpha^{(i+1)}(\two)}
 {\alpha^{(1)}(\two)f_\psi^{(1)}(\two)\cdots f_\psi^{(i)}(\two)}Z^{q^i}
\]
and the Papanikolas-type generating series
\[
 \mathcal P_{\alpha,f_\psi}(\two;Z)
 :=\sum_{i\geqslant0}\frac{\alpha^{(i+1)}(\two)}{F_\psi^{[i+1]}(\two)}Z^{q^i}
 =\frac{\alpha^{(1)}(\two)}{f_\psi(\two)}
  \operatorname{Log}_{\alpha,f_\psi}(\two;Z)
\]
of Section~\ref{sec:large-logarithms} give a further series expression for
these differentials after multiplication by $\nu$ on their common domain of normal convergence. The deformed logarithm specializes as
$\operatorname{Log}_{\alpha,f_\psi}(\Xi;Z)=\log_\psi(Z)$,
whereas the generating differential satisfies\revblue{:}
\[ \operatorname{res}_\Xi(\mathcal P_{\alpha,f_\psi}(\two;Z)\nu)
=\log_\psi(Z).\]

The isogeny transport
$\mathcal P_{\alpha,f_\psi}=h_J\,\mathcal P_{\alpha/h_J,f_\varphi}$
expresses the isogeny-adapted family in these coordinates.  The residue formula\revblue{:}
\[ \log_\psi(Z)=\sum_{i\geqslant0}\operatorname{res}_\Xi
\bigl(\widetilde\lambda_\psi^{(i+1)}/F_\psi^{[i+1]}\bigr)Z^{q^i},\]
proved in arbitrary degree in \cite[Theorem~II.20]{Chu21},
follows here by formal composition from the kernel defect equation
(Theorem~\ref{prop:log-is-log}).
Thus the logarithm coefficients arise from the same kernel identity
that governs isogeny transport.
\end{enumerate}

Figure~\ref{fig:constructions} displays the relations between the
different constructions.
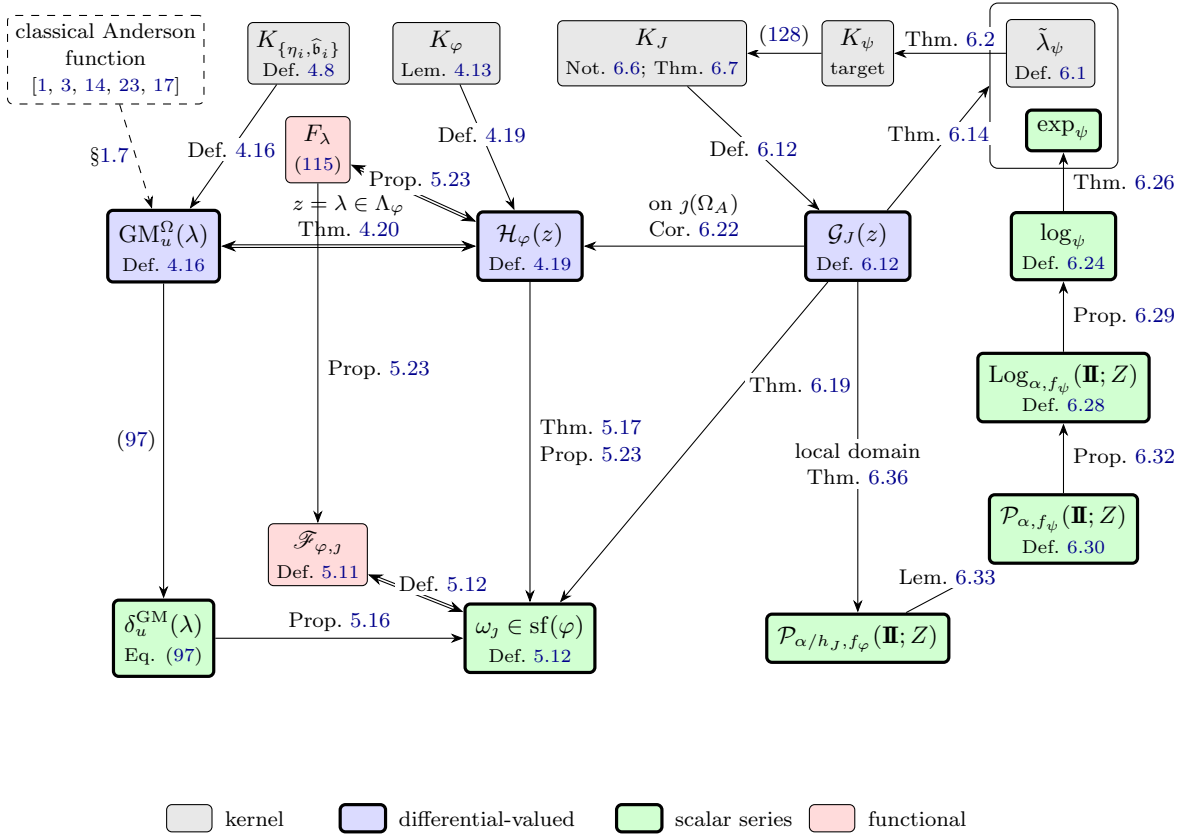
\begin{figure}[ht]
\centering
\usetikzlibrary{calc,fit}

\begin{tikzpicture}[
  x=.85cm,y=.85cm,>=Stealth,
  kern/.style={draw,fill=gray!18,rounded corners=2pt,font=\footnotesize,align=center,inner sep=3pt},
  main/.style={draw,very thick,fill=blue!14,rounded corners=2pt,font=\footnotesize,align=center,inner sep=4pt},
  scal/.style={draw,very thick,fill=green!18,rounded corners=2pt,font=\footnotesize,align=center,inner sep=4pt},
  peri/.style={draw,fill=red!14,rounded corners=2pt,font=\footnotesize,align=center,inner sep=3pt},
  prop/.style={draw,dashed,fill=yellow!15,rounded corners=2pt,font=\tiny,align=left,inner sep=5pt},
  lab/.style={font=\scriptsize,align=center,fill=white,inner sep=1.5pt},
  ext/.style={draw,dashed,fill=white,rounded corners=2pt,font=\scriptsize,align=center,inner sep=3pt},
]
\begin{scope}[xshift=-0.7cm]
\node[font=\small\bfseries] at (4.8,12.8)
  {Anderson generating functions / differentials};

\node[ext] (classical) at (-2.9,10.9) {classical Anderson\\  function\\ {\tiny \cite{And86,AT90,EGP14,HHY26,GM21}}};
\node[kern] (kgm) at (0.1,11.0) { {$K_{\{\eta_i,\widehat{\mathfrak b}_i\}}$}\\ {\tiny Def.~\ref{def:GM-ordered-Cauchy-lift}}};
\node[kern] (kphi) at (2.4,11.0) {$K_\varphi$\\ {\tiny Lem.~\ref{lem:SzegoC}}};
\node[kern] (kJ) at (5.6,11.0) {$K_J$\\ {\tiny Not.~\ref{not:KJ}; Thm.~\ref{thm:isogeny-adapted-szego-kernel}}};
\node[kern] (kpsi) at (8.8,11.0) {$K_\psi$\\ {\tiny target}};
\node[kern] (lam) at (11.8,11.0) {$\tilde\lambda_\psi$\\ {\tiny Def.~\ref{def:shtuka-diff}}};

\node[main] (GM) at (-2.0,8.0) {$\mathrm{GM}^{\Omega}_{u}(\lambda)$\\ {\tiny Def.~\ref{def:canonical-assembly}}};
\node[main] (H) at (3.7,8.0) {$\mathcal H_{\varphi}(z)$\\ {\tiny Def.~\ref{def:kernel}}};
\node[main] (GJ) at (8.8,8.0) {$\mathcal G_{J}(z)$\\ {\tiny Def.~\ref{def:isogeny-generating-differential}}};

\node[scal] (om) at (3.7,1.9) {$\omega_{\jmath}\in\sfm(\varphi)$\\ {\tiny Def.~\ref{def:omega-jmath}}};
\node[scal] (dGM) at (-2.0,1.9) {$\delta^{\mathrm{GM}}_u (\lambda) $ \\ {\tiny Eq. \eqref{eq:GM-delta-identification}}};

\node[scal] (exp) at (12.0,9.8) {$\exp_{\psi}$};
\node[scal] (log) at (12.0,8.0) {$\log_{\psi}$\\ {\tiny Def.~\ref{def:log}}};
\node[scal] (Log) at (12.0,5.8) {$\mathrm{Log}_{\alpha,f_\psi}(\two;Z)$\\ {\tiny Def.~\ref{def:large-log}}};
\node[scal] (Hs) at (12.0,3.6) {$\mathcal P_{\alpha,f_\psi}(\two;Z)$\\ {\tiny Def.~\ref{def:one-step-H}}};
\node[scal] (Ht) at (8.8,1.9) {$\mathcal P_{\alpha/h_J,f_\varphi}(\two;Z)$};

\node[peri] (W) at (0.4,9.5) {$F_{\lambda}$\\
 {\tiny \eqref{eq:hg-direct-special-function-map} }};
\node[peri] (Fj) at (0.4,3.2) {$\mathscr{F}_{\varphi,\jmath}$\\ {\tiny Def.~\ref{def:Fjmath}}};

\node[draw,rounded corners=3pt,inner sep=6pt,fit=(lam)(exp)] (lamexp) {};



\draw[->] (kgm) -- (GM) node[lab,pos=.52]{Def.~\ref{def:canonical-assembly}};
\draw[->] (kphi) -- (H) node[lab,pos=.4]{Def.~\ref{def:kernel}};
\draw[->] (kJ) -- (GJ) node[lab,pos=.5]{Def.~\ref{def:isogeny-generating-differential}};
\draw[->] (kpsi) -- (kJ) node[lab,pos=.5,above=1pt]{\eqref{eq:isogeny-kernel-coordinate-change}};
\draw[->] (lam) -- (kpsi) node[lab,pos=.5,above=1pt]{Thm.~\ref{prop:exist-shtuka}};
\draw[->] (GJ) -- (lamexp.west) node[lab,pos=.5,above=1pt]{Thm.~\ref{thm:frob}};
\draw[->,dashed] (classical) -- (GM) node[lab,pos=.5,left=1pt]{\S\ref{sec:relation}};

\draw[<->,double] (GM) -- (H) node[lab,pos=.5,above=2pt]{ {$z=\lambda\in\Lambda_\varphi$}\\Thm.~\ref{thm:canonical-assembly}};
\draw[<->,double] (H) -- (W) node[lab,pos=.42,above=1pt]{Prop.~\ref{prop:dualelement} };
\draw[->] (GJ) -- (H) node[lab,pos=.5,above=2pt]{ {on $\jmath(\Omega_A)$}\\ {Cor.~\ref{cor:primitive-G-GH}}};
\draw[->] (H) -- (om) node[lab,pos=.5,right=2pt]{Thm.~\ref{thm:identity-dual-factorization}\\Prop.~\ref{prop:dualelement}};
\draw[->] (GM) -- (dGM) node[lab,pos=.5,left=1pt]{\eqref{eq:GM-delta-identification}};
\draw[->] (dGM) -- (om) node[lab,pos=.5,above=1pt]{Prop.~\ref{prop:GM-residue-dual-coordinate}};
\draw[->] (GJ) -- (om) node[lab,pos=.32,right=1pt]{ {Thm.~\ref{thm:primitive-G}}};

\draw[->] (GJ) -- (Ht) node[lab,pos=.55]{ {local domain}\\Thm.~\ref{thm:HJ-equal}};
\draw[->] (Ht) -- (Hs) node[lab,pos=.5,above=0 pt]{Lem.~\ref{lem:isogeny-H} };
\draw[->] (Hs) -- (Log) node[lab,pos=.5,right=2pt]{Prop.~\ref{prop:H-large-log}};

\draw[->] (Log) -- (log) node[lab,pos=.5,right=2pt]{Prop.~\ref{prop:large-log}};
\draw[->] (log) -- (exp) node[lab,pos=.5,right=2pt]{Thm.~\ref{prop:log-is-log}};

\draw[->]  (W) -- (Fj) node[lab,pos= .55, right=2pt]{Prop.~\ref{prop:dualelement}};

\draw[<->,double] (Fj) -- (om) node[lab,pos=.25, right=1pt]{Def.~\ref{def:omega-jmath}};

\node[kern,minimum width=6mm,minimum height=3.6mm] at (-1.6,-0.9) {};
\node[font=\scriptsize,anchor=west] at (-1.2,-0.9) {kernel};
\node[main,minimum width=6mm,minimum height=3.6mm] at (1.1,-0.9) {};
\node[font=\scriptsize,anchor=west] at (1.5,-0.9) {differential-valued};
\node[scal,minimum width=6mm,minimum height=3.6mm] at (5.4,-0.9) {};
\node[font=\scriptsize,anchor=west] at (5.8,-0.9) {scalar series};
\node[peri,minimum width=6mm,minimum height=3.6mm] at (8.4,-0.9) {};
\node[font=\scriptsize,anchor=west] at (8.8,-0.9) {functional};
\end{scope}
\end{tikzpicture}
\caption{Relations between the main constructions.}
\label{fig:constructions}
\end{figure}

\subsection{Unit special functions}

An isogeny to a principal-period target computes an anchor map (i.e. $ \jmath_J : \Omega_A \to \Lambda_\varphi $) before
$\omega_\jmath$ is known. Suppose
$\Lambda_\psi=\iota(A) \tilde{\pi}$ and $\Omega_A=\ii_J\nu_J$ with
$\nu_J\in \Omega_A^{\rat} $. Then
\begin{equation}\label{eq:intro-direct-period-map}
\jmath_J(a\nu_J)=\iota(a)\frac{ \tilde{\pi}}{h_J(\Xi)}
\quad(a\in\ii_J)
\end{equation}
defines an isomorphism.
Theorem~\ref{thm:primitive-G} extends the factorization to the larger
differential lattice:
\[
\mathcal G_J(\iota(a)z_J)=a\omega_{\jmath_J}\otimes\nu_J
\quad(a\in A),\qquad z_J= \tilde{\pi}/h_J(\Xi).
\]
Its value is integral exactly when $a\in\ii_J$.

Every rank-one source admits an isogeny to a principal-period target.
 The condition $\Omega_A=\ii_J\nu_J$ holds precisely when
$[\Lambda_\varphi]=[\Omega_A]$ in $\operatorname{Pic}(A)$.
In particular, neither module needs to be free. 
 The following theorem settles the Gazda--Maurischat conjecture
\cite[\S3.4]{GM21} at every infinity degree.  Freeness of
$\operatorname{sf}(\varphi)$ forces it to contain a Tate-algebra unit. 

\begin{lettertheorem}[Unit locus]\label{cor:GM-unit-equivalence}
For every $\Ninf\geqslant1$, the following are equivalent:
\begin{enumerate}[label=\textup{(\roman*)},leftmargin=2.2em]
\item\label{item:intro-period-class}
$\Lambda_\varphi\simeq\Omega_A$.
\item\label{item:intro-free-special}
$\operatorname{sf}(\varphi)$ is free of rank one.
\item\label{item:intro-unit-special}
$\operatorname{sf}(\varphi)$ contains a Tate-algebra unit.
\item\label{item:intro-strict-coefficient}
In an $\Fq$-basis $(a_i)$ of $A$ with $a_0=1$, the function $f_\varphi$
belongs to $\mathbb T_A$ and its coefficients, defined by
$f_\varphi=\sum_i\varsigma_i\otimes a_i$, satisfy
$\varsigma_0\ne0$ and $\sup_{i>0}|\varsigma_i/\varsigma_0|<1$.
\end{enumerate}
The equivalence of \ref{item:intro-period-class} and
\ref{item:intro-free-special} follows from the Gazda--Maurischat
transform.
The correspondence restricts to an $\Fq^\times$-equivariant bijection
\begin{equation}\label{eq:unit-locus-period-isomorphism-torsor}
\operatorname{sf}(\varphi)\cap\mathbb T_A^\times
\simeq\operatorname{Isom}_A(\Omega_A,\Lambda_\varphi).
\end{equation}
\end{lettertheorem}

 Part~\ref{item:intro-strict-coefficient} characterizes Tate-algebra
units by strict normalization.  
 The proof of \cite[Theorem~4.8]{Fer25} requires a strict decrease in
$|\exp_\varphi(a_i^*)|$.  This argument need not apply when $\Ninf>1$.
Several basis vectors can then have the same pole order and the same dual norm.  

\subsection{Explicit fundamental period}
Thanks to Theorem~\ref{cor:GM-unit-equivalence}, we can extract the
period from the strict shtuka product with the chosen unit-anchor
normalization.  A unit-normalizing scalar can be chosen from the shtuka
function: in the basis of Theorem~\ref{cor:GM-unit-equivalence}, take
$\beta^{q-1}=\varsigma_0$.
The differential transport formula below applies to arbitrary source
lattices. 
\begin{lettertheorem} \label{thm:explicit-period-lettered} \revblue{Let $J$ be an isogeny datum to a principal-period target $\psi$, with the unit-anchor normalization $\Omega_A=\ii_J\nu_J$ and $\jmath(\nu_J)=\tilde{\pi}/h_J(\Xi)$ after localization. Suppose $\Xi^{(-1)}\notin\operatorname{supp}(V_\varphi)$, and put $\beta_0=\exp_\varphi(c_0)$ for the minimum-norm $c_0$. Then $\beta_0$ is unit-normalizing in the sense of Definition~\ref{def:unit-normalizing-scalar}. Every unit-normalizing $\beta$ has a unique decomposition $\beta=\zeta u\beta_0$, with $\zeta\in\Fq^\times$ and $u\in1+\mathfrak m_\infty$, where $\mathfrak m_\infty=\{z\in\Cinf:|z|<1\}$. The period is given by:} 
\[
 \tilde{\pi}
 = h_J(\Xi)\,\zeta^{-1}\,\beta\,{\Upsilon_\beta(\Xi)}\, \operatorname{res}_\Xi(\nu_J/f_{\varphi,\beta}),
 \qquad \Upsilon_\beta:=\prod_{n\geqslant1}\bigl(f_{\varphi,\beta}^{(n)}\bigr)^{-1},
\]
where $\Upsilon_\beta(\Xi)$ is the value of the meromorphic tail
product after cancellation; individual early factors need not be
regular at $\Xi$.
\end{lettertheorem}

\subsection{Finite special values}\label{sec:finite-special-values}

Let $\Lambda_\psi=\iota(A) \tilde{\pi}$, write $\chi(a)=1\otimes a$, and set
\[
\mathcal L_A=\sum_{0\ne a\in A}\frac{\chi(a)}{\iota(a)}.
\]

\begin{lettertheorem}[Finite-trace identification]\label{thm:finite-trace-lettered}
For every $\Ninf\geqslant1$, put
$\mathcal R_\psi^{[k]}=- \tilde{\pi}^{-1}\mathcal L_A
\mathcal{H}_\psi(\tilde{\pi})^{(k)}$ \revblue{for $k\geqslant0$} (the factor
$\tilde{\pi}^{-1}\mathcal L_A$ remains untwisted)\revblue{, and set
$\mathcal R_\psi=\mathcal R_\psi^{[0]}$}.  Then
$\widetilde\lambda_\psi^{(1)}=\mathcal R_\psi^{[1]}$ and
$K_\psi(\Xi,\cdot)=\mathcal R_\psi$.
\end{lettertheorem}

 The coefficients of $\widetilde\lambda_\psi^{(1)}$ have the
following finite description.  Let $H\subset\Cinf$ be a field
containing $\iota(K)$ and the coefficients of $\psi$,
put $m=\lceil2/\Ninf\rceil$, and choose an $\Fq$-basis
$\mu_1,\ldots,\mu_e$ of $H^0(X,\Omega_X^1(m\infty))$.
Choose $c_j\in A$ and a common denominator
$b\in A\setminus\Fq$ such that
$\langle c_j/b,\mu_i\rangle_\infty=\delta_{ij}$, where
$\langle u,\mu\rangle_\infty
=\operatorname{Tr}_{\mathbb F_\infty/\Fq}
\operatorname{res}_\infty(u\mu)$ and
$\mathbb F_\infty$ is the residue field at $\infty$.
Put
$\mathscr T_b=H[Z]/(\psi_b(Z)/Z)$.  Then
\[
 \widetilde\lambda_\psi^{(1)}
 =-\frac1{\iota(b)}
  \sum_{j=1}^e
  \operatorname{Tr}_{\mathscr T_b/H}
   \left(\frac{\psi_{c_j}(Z)^q}{Z}\right)\mu_j .
\]
\revblue{By} Theorem~\ref{thm:hg-finite-torsion-trace-value}\revblue{,} the traces are computed in the finite \'{e}tale algebra $\mathscr T_b$
from the Drinfeld action polynomials.

\revblue{Theorem~\ref{thm:pellarin-period-identity}\ref{item:pellarin-first} gives the contraction formula for every $\hbar\in\operatorname{Hom}_A(\Omega_A,A)$:}
\[
\mathcal L_A\omega_{\jmath_\hbar}^{(1)}
=- \tilde{\pi}\,\hbar(\widetilde\lambda_\psi^{(1)}),
\qquad \jmath_\hbar= \tilde{\pi}\,\iota\circ\hbar.
\]
For an arbitrary source lattice, a free-target isogeny gives instead\revblue{:}
\[
\frac{\mathcal L_A}{ \tilde{\pi}}\mathcal G_J(\iota(a)z_J)
=-\frac{\chi(a)}{h_J}\mathcal R_\psi\qquad(a\in A)\]
by Proposition~\ref{prop:finite-torsion-isogeny-kernels}.
For the unit special function
$\omega_{\jmath_*}=\mathcal G_J(z_J)/\nu_J$ attached to the unit
anchor, Theorem~\ref{thm:finite-value-isogeny-coordinate} evaluates the
same pair:
\[
\mathcal L_A\frac{\omega_{\jmath_*}}{\tilde{\pi}}
=-\frac{K_\psi(\Xi,\cdot)}{h_J\nu_J},
\qquad
\mathcal L_A\frac{\omega_{\jmath_*}^{(1)}}{\tilde{\pi}}
=-\frac{\tilde{\lambda}_\psi^{(1)}}{h_J^{(1)}\nu_J}.
\]
\revblue{These identities give three forms of the finite-trace identification:
through $\mathcal G_J$, through contractions by $\hbar$, and through the unit
coordinate $\omega_{\jmath_*}$.}
At rational infinity, $\mathcal L_A$ is the negative of the monic
curve zeta series. The final comparison records this sign and the residue
normalization needed to match the algebraic factor in
\cite[Theorem~5.3]{Pel26}.

\subsection{Relation to earlier constructions}\label{sec:relation}

\revblue{The classical generating series starts with the exponential
\cite{And86,AT90,EGP14}.  The special-function and characteristic-residue
formalisms were developed in \cite{GM21,GM24}.}
The algebraic relations among special values of the geometric
$\Gamma$-function over $\Fq[t]$ were determined in
\cite{ABP04}, and Tannakian and difference-Galois frameworks in \cite{Pap08}.
\revblue{Our Cauchy-coordinate construction produces the Gazda--Maurischat
transform and the isogeny-adapted family.  Both are rational Anderson
differentials on their respective homogeneous parameter lattices.} It also separates
the source period lattice from the larger homogeneous lattice attached
to an isogeny.

Ferraro's differential pairings already provide rationality and affine
regularity results in a general Drinfeld-module setting
\cite[Theorem~7.3]{Fer25}.
Ferraro also expresses their residue-dual evaluations through
the Drinfeld action and its adjoint
\cite[Theorem~7.18 and Remark~7.31]{Fer25}.
\revblue{In the present setting, the coefficients of these pairings are finite
\'{e}tale traces.  They are computed from the $b$-torsion of $\psi$ and
descend to the field of definition of the Drinfeld action.  No individual
torsion points need to be adjoined.  The trace is taken over
$\mathscr T_b=H[Z]/(\psi_b(Z)/Z)$, as in
\eqref{eq:hg-nonzero-torsion-algebra}; the evaluation and pole bounds are given
in Theorem~\ref{thm:hg-finite-torsion-trace-value}.} \revblue{In rank one, the poles must be located within the full infinity orbit.
This identifies the first pairing with the first Frobenius twist of the
normalized shtuka differential.} The elliptic formula of
Green--Papanikolas \cite{GP18} and the curve zeta identity reviewed in
\cite{Pel26} provide comparisons for the scalar special
values. At rational infinity, formulas for exponential and logarithm
coefficients also enter tensor-power special-zeta evaluations
\cite{Gre19} and Eulerian multiple-zeta identities \cite{CNP23}.
The analytic theory of Pellarin's $L$-series was
developed in \cite{Gos13}, and its convolution generalizations in
\cite{HP22}.

 Compared with \cite{Fer25} and \cite{Pel26}, the main new
contributions are the following.  They hold for a closed point of arbitrary
degree, including degree divisible by the characteristic, and do not assume that
the source lattice is free:  
\begin{enumerate}[label=\textup{\arabic*.},leftmargin=2.2em]
\revblue{\item Cauchy-fiber invariance
(Theorem~\ref{thm:Vphi-Cauchy-fiber-invariance}).  Cauchy representatives with
first-variable poles bounded by $V_\varphi$ have the same value at every period.
The whole fiber has a common value exactly on $\Lambda_\varphi$.}
\revblue{\item The isogeny-adapted generating differential $\mathcal G_J$
(Definition~\ref{def:isogeny-generating-differential};
Theorem~\ref{thm:frob}).  It satisfies a one-step Frobenius defect equation.
Its homogeneous lattice is exactly
$h_J(\Xi)^{-1}\Lambda_\psi=\iota(\ii_J^{-1})\Lambda_\varphi$
(Corollary~\ref{cor:H-source-lattice-recovery}).}
\item Finite torsion-trace formulas for the algebraic factor:
Theorem~\ref{thm:hg-shtuka-trace-identification} identifies
$\widetilde\lambda_\psi^{(1)}=\mathcal R_\psi^{[1]}$ by its divisor and
residue.  The coefficients are finite torsion traces over the field of
definition (Theorem~\ref{thm:hg-finite-torsion-trace-value}).
Contraction and isogeny transport give the special-value identities of
Theorems~\ref{thm:pellarin-period-identity}
and~\ref{thm:finite-value-isogeny-coordinate}.  The curve-zeta
normalization is described in Remark~\ref{rem:curve-zeta-comparison}.
\revblue{\item The explicit fundamental period
(Theorem~\ref{thm:explicit-period-lettered}).  A strictly normalized infinite
product of Frobenius twists of the shtuka function, multiplied by one
characteristic residue, gives $\tilde\pi$.  This holds at every infinity degree
under the stated unit-anchor condition, which also fixes the normalization.}
\revblue{\item The unit theorem for arbitrary infinity degree
(Theorem~\ref{thm:unit-special-arbitrary-degree-public}).  The proof uses a
residue-orthogonal basis and the radial norm formula, extending the argument of
\cite[\S4]{Fer25} beyond $\deg\infty=1$.  If $\jmath$ is a period
isomorphism, then $\omega_\jmath$ is a Tate-algebra unit for every $\Ninf$.
This proves the general Gazda--Maurischat conjecture \cite{GM21}
(Theorem~\ref{cor:GM-unit-equivalence}).}
\end{enumerate}

\subsection{Proof strategy and organization}

Section~\ref{sec:Preliminary} fixes the geometric realization and
isogeny conventions and develops the characteristic orbit
with its automatic weighted estimates. Section~\ref{sec:evaluation}
constructs motive evaluation over the rational motive and,
after completion, over the separated tensor domain and its local orbital
extension.
 Section~\ref{sec:cauchy-gm-period-class} constructs the
 Szeg\H o coordinate, compares 
Cauchy representatives, and identifies their period evaluation with the
Gazda--Maurischat transform. 

Section~\ref{sec:anchor} constructs residue-dual coordinates and the
logarithmic inverse. \revblue{Section~\ref{sec:onestep} constructs the normalized shtuka
differential.  It proves the isogeny defect, determines the exact homogeneous
parameter lattice, and establishes the primitive $G$-identity.  The section also
studies the scalar forcing divisor, the logarithm residue formula, and the
Papanikolas-type series with their isogeny transport.}
Section~\ref{sec:unit-special} proves invertibility and the product and
period formulas. \revblue{Finally, Section~\ref{sec:higher-genus-single-infinity} computes the
torsion pairings by finite traces.  Divisor bounds, the characteristic residue,
and Corollary~\ref{cor:shtuka-residue-uniqueness} identify the first pairing
with the first Frobenius twist of the normalized shtuka differential.} 

\revsection{Preliminaries}{Preliminaries}
\label{sec:Preliminary}

We fix the geometric realization and the coefficient conventions used in
motive evaluation.  The identity additive polynomial determines the rational
frame, while the Drinfeld divisor records the integral motive.  This
distinction will control both period annihilation and isogeny transport.
\revblue{The final two subsections develop the characteristic orbit and its
weighted specialization.  They also prove the automatic estimates used by the
motive evaluations of Section~\ref{sec:evaluation}.}

\revsubsection{Geometric setting}{Geometric setting}\label{sec:basedata}

Let $X/\Fq$ be a smooth, projective, geometrically connected curve with a fixed
closed point $\infty$ of degree $\Ninf=\deg(\infty)$; put
$U=X\setminus\{\infty\}$, and write $g=g(X)$,
$A=\Gamma\bigl(U,\OO_X\bigr)$, and $K=\Fq(X)$.
Let $K_\infty$ be the completion of $K$ at $\infty$, and let
$\Cinf=\widehat{\overline{K_\infty}}$.  Fix an isometric $\Fq$-embedding
$\iota:K_\infty\hookrightarrow\Cinf$, with
$|\iota(a)|=q^{-\Ninf v_\infty(a)}$.  Its restriction to $A$ determines the
\emph{characteristic point}
$\Xi\in X(\Cinf)$, characterized by $c(\Xi)=\iota(c)$ for $c\in K$. When
$A=\Fq[t]$ we write $\theta:=\iota(t)$, so that $\Xi=\theta$ in the coordinate $t$.
Write $\Omega_A:=\Omega^1_{A/\Fq}$ for the $A$-module of regular differentials on $U$, $\Omega_A^{\rat}:=K\otimes_A\Omega_A$ for the rational differentials, $\widehat\Omega_A:=K_\infty\otimes_A\Omega_A$, and $\overline\Omega:=\widehat\Omega_A/\Omega_A$.

For every Drinfeld module used below, its period lattice means the embedded
lattice
\begin{equation}\label{eq:embedded-period-lattice-convention}
 \Lambda_\varphi:=\ker(\exp_\varphi:\Cinf\to\Cinf)
 \subset\Cinf.
\end{equation}
It is a discrete projective rank-one $A$-submodule of $\Cinf$; its abstract
projective-module type is denoted by $[\Lambda_\varphi]\in\Pic(A)$.  If
$\mathfrak a\subset K$ is a fractional ideal and
$z\in\Cinf$, then
\begin{equation}\label{eq:embedded-ideal-scalar-convention}
 \iota(\mathfrak a)z:=\{\iota(a)z:a\in\mathfrak a\}\subset\Cinf.
\end{equation}
We use the conventional shorthand $Az:=\iota(A)z$ for a principal embedded
lattice, while retaining $\iota(\mathfrak a)$ for a nonprincipal fractional
ideal.  Thus $Az$, $A\pi$, and $A\lambda$ denote submodules of $\Cinf$, whereas
$[\Lambda_\varphi]$ denotes an ideal class.  The natural localization map
\begin{equation}\label{eq:localized-lattice-natural-injection}
 K\otimes_A\Lambda_\varphi\hookrightarrow\Cinf,
 \qquad a\otimes\lambda\longmapsto\iota(a)\lambda,
\end{equation}
is injective, with image $\iota(K)\Lambda_\varphi$.  We use it to identify
$K\otimes_A\Lambda_\varphi$ with the localized $K$-line of the embedded
lattice.

Write $K\otimes_{\Fq}\Cinf$ for the algebraic constant extension inside the
$\Cinf$-function field
$\Cinf(X)=\operatorname{Frac}(A\otimes_{\Fq}\Cinf)$.  
We use the general Tate algebra
\begin{equation}\label{eq:general-Tate-algebra-main}
 \mathbb T_A:=\Cinf\widehat\otimes_{\Fq}A.
\end{equation}
For any $\Fq$-basis $(a_i)_{i\in I}$ of $A$, every element has a unique
convergent expansion $F=\sum_i c_i\otimes a_i$, with $c_i\to0$, and
\begin{equation}\label{eq:general-Tate-basis-norm-main}
 \lVert F\rVert_{\mathrm G}=\sup_i|c_i|.
\end{equation}
Changes of $\Fq$-basis act isometrically for the supremum norm, so this
description is basis independent; multiplication is submultiplicative.
Coefficient Frobenius is denoted by
$F\mapsto F^{(1)}=\sum_i c_i^q\otimes a_i$.  

\subsection{Rank-one realization}\label{sec:realization}

Let $\varphi$ be a rank-one Drinfeld $A$-module \cite{Dri74} over $\Cinf$.
No sign normalization is imposed. We fix a \emph{geometric shtuka realization}
of $\varphi$, consisting of:
\begin{enumerate}[label=(\alph*),leftmargin=2.2em]
\item a geometric point $q_\varphi\in X(\Cinf)$ above $\infty$ whose Frobenius orbit
is closed of period $ \Ninf $, i.e.\ $q_\varphi^{(\Ninf)}=q_\varphi$;
\item an effective divisor $V_\varphi$ on $X_{\Cinf}$ of degree $g$ with
$h^0\bigl(X_{\Cinf},\OO(V_\varphi)\bigr)=1$, such that
$\Xi,q_\varphi^{(-1)}\notin\supp(V_\varphi)$.  No disjointness of the full Frobenius orbit of
$q_\varphi$ from $V_\varphi$ is imposed;
\item the \emph{shtuka function} $f_\varphi\in\Cinf(X)^\times$, with
\begin{equation}\label{eq:divf}
\divv(f_\varphi)\;=\;V_\varphi^{(1)}-V_\varphi+[\Xi]-[q_\varphi] ;
\end{equation}
\item a $\tor$-equivariant isomorphism
$\kappa_\varphi\colon \cM^{\mathrm{geo}}_\varphi\xrightarrow{\ \sim\ }\cM_\varphi$
from the \emph{geometric motive}
$\cM^{\mathrm{geo}}_\varphi:=\Gamma\bigl(U_{\Cinf},\OO(V_\varphi)\bigr)$, equipped
with $\tor^{\mathrm{geo}}(g)=f_{\varphi}\,g^{(1)}$, onto the Anderson motive
$\cM_\varphi$ of $\varphi$.  In the standard additive-polynomial
realization it sends $1$ to the identity additive polynomial; we denote this
distinguished element by $s_\varphi$ (so $s_\varphi(x)=x$).  We use
$s_\varphi$ as a generic rational frame; it is not
assumed to generate the integral motive as an
$A\otimes_{\Fq}\Cinf$-module.
\end{enumerate}
Writing $F_\varphi^{[n]}:=f_\varphi f_\varphi^{(1)}\cdots f_\varphi^{(n-1)}$, the map
$\kappa_\varphi$ satisfies $\kappa_\varphi\bigl(F_\varphi^{[n]}\bigr)=\tor^n
s_\varphi$.  Thus, if the additive polynomial attached to
$m\in\cM_\varphi$ is
\begin{equation}\label{eq:additive-polynomial-coordinate}
 m=\sum_{i=0}^N\alpha_i\tor^i,
 \qquad \alpha_i\in\Cinf,
\end{equation}
then its geometric coordinate in the generic frame $s_\varphi$ is
\begin{equation}\label{eq:geometric-motive-coordinate}
 g_m:=\sum_{i=0}^N\alpha_iF_\varphi^{[i]},
 \qquad m=g_ms_\varphi.
\end{equation}
In particular, if
$\varphi_a=\sum_i\varphi_{a,i}\tor^i$, $A$-linearity gives
\begin{equation}\label{eq:a-shtuka-coordinate}
 a(\one)=\sum_i\varphi_{a,i}F_\varphi^{[i]}(\one),
 \qquad a\in A.
\end{equation}
 Indeed, the action is $A$-linear and $1\cdot s_\varphi=s_\varphi$, so the geometric coordinate of $a\cdot s_\varphi$ is the Frobenius-basis expansion of the scalar function $a$ itself.
These coordinate identities and the rank-one projective-module structure are
the standard shtuka realization; see \cite[\S2.3]{ANDTR17a}.
In the identity rational frame ($A=\Fq[t]$) the motive is the free
$\Cinf[t]$-module $\Cinf[t]\,s_\varphi$ with $\tor$-action
$\tor\,s_\varphi=f_\varphi(t)\,s_\varphi$; e.g.\ for the Carlitz module
$\varphi_t=\theta+\tau$ one has $f_\varphi(t)=t-\theta$, and for
$\varphi_t=\theta+\Delta\tau$ one has $f_\varphi(t)=(t-\theta)/\Delta$.

\begin{remark}\label{rem:scalar-conjugation-realization}
For $c\in\Cinf^\times$, put $\varphi^c_a=c\varphi_a c^{-1}$.
Its geometric realization has the same $V_\varphi$ and $q_\varphi$, with
\begin{equation}\label{eq:scalar-conjugation-realization}
 f_{\varphi^c}=c^{q-1}f_\varphi,\qquad
 \exp_{\varphi^c}(z)=c\exp_\varphi(c^{-1}z),\qquad
 \Lambda_{\varphi^c}=c\Lambda_\varphi.
\end{equation}
Indeed, $\varphi^c_{a,i}=c^{1-q^i}\varphi_{a,i}$ and
$F_{\varphi^c}^{[i]}=c^{q^i-1}F_\varphi^{[i]}$, so
\eqref{eq:a-shtuka-coordinate} is unchanged.  More explicitly,
$T_c(m):=cmc^{-1}$ sends the identity polynomial to itself and satisfies
$\tau T_c(m)=c^{q-1}T_c(\tau m)$.  Thus
$T_c\circ\kappa_\varphi$ gives the required geometric motive isomorphism.
The exponential formula follows from its normalized functional equation,
and taking its kernel gives the lattice formula.
\end{remark}

We shall repeatedly use the \emph{membership criterion}: for
$c\in\Cinf(X)$,
\begin{equation}\label{eq:membership}
\begin{aligned}
c\,s_\varphi\in\cM_\varphi
&\text{ if and only if }c\in\Gamma\bigl(U_{\Cinf},\OO(V_\varphi)\bigr),\\
c\,s_\varphi\in\cM_{\varphi,K}
&\text{ if and only if }c\in (A\smallsetminus\{0\})^{-1}\Gamma\bigl(U_{\Cinf},\OO(V_\varphi)\bigr),
\end{aligned}
\end{equation}
where $\cM_{\varphi,K}:=K\otimes_A\cM_\varphi$ is the localized motive and
$c\,s_\varphi$ denotes the generic-frame coordinate transported through
$\kappa_\varphi$.  Every $m\in\cM_\varphi$ is an additive polynomial as in
\eqref{eq:additive-polynomial-coordinate}, acting by
$m(x)=\sum_i\alpha_i x^{q^i}$.  For the action of $a\in A$,
exponential equivariance gives 
\begin{equation}\label{eq:functional}
\varphi_a\bigl(\Exp(w)\bigr)=\Exp\bigl(\iota(a)\,w\bigr),
\qquad a\in A .
\end{equation}

\begin{lemma}
\label{lem:direct-geometric-shtuka-Frobenius-basis}
In the fixed geometric realization, with $q_\varphi$ of period
$\Ninf$, $h^0\bigl(X_{\Cinf},\OO(V_\varphi)\bigr)=1$,
$\Xi,q_\varphi^{(-1)}\notin\operatorname{supp}(V_\varphi)$, and
\eqref{eq:divf}, put
\[
 F_\varphi^{[0]}:=1,
 \qquad
 F_\varphi^{[n]}:=f_\varphi f_\varphi^{(1)}\cdots f_\varphi^{(n-1)}\quad(n\geqslant1).
\]
\begin{enumerate}[label=\textup{(\roman*)},leftmargin=2.2em]
\item\label{item:Frobenius-affine-basis} The Frobenius products give the direct basis
\begin{equation}\label{eq:direct-geometric-shtuka-Frobenius-basis}
 \Gamma\bigl(U_{\Cinf},\OO(V_\varphi)\bigr)
 =\bigoplus_{n\geqslant0}\Cinf F_\varphi^{[n]},
 \qquad
 f_\varphi\bigl(F_\varphi^{[n]}\bigr)^{(1)}=F_\varphi^{[n+1]}.
\end{equation}
\item\label{item:Frobenius-cyclic-motive} This geometric motive is free of rank one over
$\Cinf\{\tau\}$, with generator $1$.
\end{enumerate}
\end{lemma}

\begin{proof}
 The fixed realization $\kappa_\varphi$ sends $1$ to the identity
additive polynomial and intertwines $g\mapsto f_\varphi g^{(1)}$
with left multiplication by $\tau$.  Hence it sends
$F_\varphi^{[n]}$ to $\tau^n$.  Since the additive-polynomial
realization of $\mathcal M_\varphi$ is $\Cinf\{\tau\}$, the inverse
images of its $\Cinf$-basis $(\tau^n)_{n\geqslant0}$ give
\eqref{eq:direct-geometric-shtuka-Frobenius-basis}, and $1$ generates
the geometric motive freely over $\Cinf\{\tau\}$. 
\end{proof}

\subsection{Isogeny data}\label{sec:dual}

\begin{notation}\label{not:isogeny-datum-main}
Let $\eta_J\colon\varphi\to\psi$ be a separable isogeny between rank-one
Drinfeld modules.  We write
\[
 J=[\eta_J;h_J;\ii_J,\mu_J]\colon\varphi\longrightarrow\psi
\]
for the following compatible data.
\begin{enumerate}
\item Under the motive--additive-polynomial dictionary,
\begin{equation}\label{eq:intertwiner}
 \eta_{J}^{!} (s_\psi)=h_Js_\varphi;
\end{equation}
equivalently, the motive element $h_Js_\varphi$ represents the additive
polynomial $\eta_J$.  In particular it is an integral element of
$\cM_\varphi$ and
\begin{equation}\label{eq:intertwiner-difference}
 h_J^{(1)}f_\varphi=h_Jf_\psi.
\end{equation}

\item The ideal $\ii_J$ is the annihilator ideal
\begin{equation}\label{eq:isogeny-annihilator-ideal-main}
 \ii_J=\operatorname{Ann}_A(\ker\eta_J).
\end{equation}
Identifying a nonzero integral ideal with its effective divisor on
$U=\operatorname{Spec}A$, and writing
$\pi:U_{\Cinf}\to U$ for the projection, its exact affine divisor is
\begin{equation}\label{eq:isogeny-intertwiner-affine-divisor-main}
 \pi^*\divv_U(\ii_J)
 =\left(\divv(h_J)-V_\psi+V_\varphi\right)\big|_{U_{\Cinf}}.
\end{equation}
\item Let $\varphi_{\ii_J}$ be the monic ideal isogeny with kernel
$\varphi[\ii_J]$.  Then
\begin{equation}\label{eq:isogeny-leading-coefficient-main}
 \eta_J=\mu_J\varphi_{\ii_J},
 \qquad \mu_J\in\Cinf^\times,
\end{equation}
where $\mu_J$ is the coefficient of the highest power of $\tau$ in
$\eta_J$.  In generic characteristic the ideal isogeny
$\varphi_{\ii_J}$ is separable.
\end{enumerate}
Finally, the tangent coefficient is the derived scalar
\begin{equation}\label{eq:isogeny-tangent-from-intertwiner-main}
 \partial\eta_J=h_J(\Xi)\ne0.
\end{equation}
 The equality $\partial\eta_J=h_J(\Xi)$ uses $F_\varphi^{[i]}(\Xi)=0$ for all $i\geqslant1$ (indeed $\operatorname{ord}_\Xi(F_\varphi^{[i]})=m_{-i}-m_0+1\geqslant1$ by \eqref{eq:divf} and $m_0=0$), so only the constant term of $\eta_J$ survives evaluation at $\Xi$.
\end{notation}

Analytic uniformization gives
\begin{equation}\label{eq:isogeny-exponential-compatibility-main}
 \eta_J\bigl(\exp_\varphi(w)\bigr)=\exp_\psi(h_J(\Xi)w).
\end{equation}
See, for example, \cite[Corollary~4.9.5]{Gos96}. 

\begin{proposition}\label{prop:isogeny-lattice-transport}
The ideal-isogeny uniformization formula is
\begin{equation}\label{eq:isogeny-lattice-transport-main}
 \Lambda_\psi
 =h_J(\Xi)\,\iota(\ii_J^{-1})\Lambda_\varphi,
 \qquad
 h_J(\Xi)^{-1}\Lambda_\psi
 =\iota(\ii_J^{-1})\Lambda_\varphi.
\end{equation}
\end{proposition}
 \begin{proof} 
 The identity $\eta_J=\mu_J\varphi_{\ii_J}$ gives
$\ker\eta_J=\varphi[\ii_J]$.  Under $\exp_\varphi$, the kernel of $\eta_J$
corresponds to $h_J(\Xi)^{-1}\Lambda_\psi/\Lambda_\varphi$.  The ideal torsion
$\varphi[\ii_J]$ corresponds to
$\iota(\ii_J^{-1})\Lambda_\varphi/\Lambda_\varphi$.  These two subgroups of
$\Cinf/\Lambda_\varphi$ therefore coincide.  This proves the second identity,
and multiplication by $h_J(\Xi)$ gives the first. 
 \end{proof} 
Thus the parameter
lattice killed by the $J$-defect is generally a fractional enlargement of
the source lattice.  It equals the source lattice only when
$\ii_J=A$, in particular for the identity isogeny.  We retain the
characteristic embedding on ideal multiples when it matters and suppress it
only in uncluttered formulas.
Replacing $\psi$ by an isomorphic scalar conjugate permits the normalization
$h_J(\Xi)=1$, but we state the datum before making that optional choice.
The map $\eta_J^!$ is contravariant on motives, whereas $h_J(\Xi)$ transports
the embedded lattices covariantly.  We keep these two maps distinct.

\subsection{\texorpdfstring{Special functions}{Special functions}}
\begin{definition}\label{def:special-functions}
The special-function module of $\varphi$ is
\[
 \sfm(\varphi):=
 \{\omega\in\mathbb T_A:\varphi_a(\omega)=(1\otimes a)\omega
                       \text{ for every }a\in A\},
\]
where $\varphi_a$ acts on the coefficient factor.  Its $A$-module
structure is multiplication in the geometric factor.
\end{definition}

\begin{lemma}\label{lem:scalar-shtuka-equation}
For $\omega\in\mathbb T_A$, the special-function relations are equivalent to
$\omega^{(1)}=f_\varphi\omega$ in $\operatorname{Frac}(\mathbb T_A)$.
\end{lemma}
\begin{proof}
If $\omega$ is special, evaluation $m\mapsto m(\omega)$ on additive
polynomials is $B$-linear for $B=\Cinf\otimes_{\Fq}A$: indeed
$(m\circ\varphi_a)(\omega)=m((1\otimes a)\omega)
=(1\otimes a)m(\omega)$, since coefficient Frobenius fixes $1\otimes a$.
Write $f_\varphi=P/Q$ with $P,Q\in B$, $Q\ne0$.
The motive identity $Q\tau s_\varphi=Ps_\varphi$ gives
$Q\omega^{(1)}=P\omega$. Dividing by $Q$ in
$\operatorname{Frac}(\mathbb T_A)$ gives the scalar shtuka equation.

Conversely, iteration gives $\omega^{(i)}=F_\varphi^{[i]}\omega$.
Using \eqref{eq:a-shtuka-coordinate} then gives
$\varphi_a(\omega)=\sum_i\varphi_{a,i}F_\varphi^{[i]}\omega
=(1\otimes a)\omega$.
\end{proof}

\subsection{Frobenius constants of the Tate field}
 \label{sec:tate-fixed-field}

A homogeneous quotient of special functions is Frobenius fixed.  The next
lemma identifies such quotients inside the Tate algebra and its fraction
field.

\begin{lemma}
\label{lem:general-Tate-Frobenius-constants}
The general Tate algebra is a domain, and coefficient Frobenius has constants
\begin{equation}\label{eq:general-Tate-Frobenius-constants}
 \mathbb T_A^{(1)=1}=A,
 \qquad
 \operatorname{Frac}(\mathbb T_A)^{(1)=1}=K.
\end{equation}
\end{lemma}

\begin{proof}
Choose an $\Fq$-basis $(a_i)$ of $A$.  If
$F=\sum_i c_i\otimes a_i\in\mathbb T_A$ is Frobenius fixed, uniqueness of
the convergent basis expansion in
\eqref{eq:general-Tate-basis-norm-main} gives
$c_i^q=c_i$, hence $c_i\in\Fq$.
Since $c_i\to0$ and every nonzero element of $\Fq$ has norm one, only
finitely many $c_i$ are nonzero.  Thus $F\in A$, proving the first equality.

We prove the fraction-field assertion without appealing to descent from the
single Frobenius automorphism.  Choose a nonconstant $t\in A$.  Its only pole
on $X$ is supported at $\infty$, so $B:=\Fq[t]\subset A$ is finite.  Since
$A$ is torsion free over the principal ideal domain $B$, it is free; fix a
$B$-basis $e_1,\ldots,e_s$.  Put
\[
 \mathbb T_0:=\Cinf\langle t\rangle,
 \qquad F_0:=\operatorname{Frac}(\mathbb T_0).
\]
Completion commutes with this finite free extension and gives
\begin{equation}\label{eq:Tate-Noether-normalization}
 \mathbb T_A=\bigoplus_{j=1}^s\mathbb T_0e_j.
\end{equation}
The reduction of $\mathbb T_A$ is the scalar extension of the geometrically
integral ring $A$; the Gauss seminorm is therefore a multiplicative norm, so
$\mathbb T_A$ is a domain.  Localizing
\eqref{eq:Tate-Noether-normalization} at
$\mathbb T_0\setminus\{0\}$ produces a finite-dimensional domain over $F_0$,
hence a field, and consequently
\begin{equation}\label{eq:Tate-field-Noether-normalization}
 \operatorname{Frac}(\mathbb T_A)
 =\bigoplus_{j=1}^sF_0e_j.
\end{equation}

We first record the one-variable fixed field
\begin{equation}\label{eq:one-variable-Tate-fixed-field}
 F_0^{(1)=1}=\Fq(t).
\end{equation}
Indeed, let $H\in F_0$ be fixed.  Its zero and pole divisors on the closed
unit disc are finite and Frobenius stable.  By one-variable Weierstrass
preparation  \cite[\S2.2, Corollaries~9--10]{Bos14}, the monic polynomials whose roots, with multiplicity, are these
zeros and poles are uniquely determined.  Frobenius stability and monicity
therefore force both polynomials to lie in $\Fq[t]$.  Dividing $H$ by their
quotient leaves a Frobenius-fixed unit $U\in\mathbb T_0^\times$.  The first
part of the proof, applied to the monomial basis of $B$, gives
$\mathbb T_0^{(1)=1}=\Fq[t]$.  A polynomial in $\Fq[t]$ which is a unit on
the closed unit disc is constant, because every nonconstant polynomial over
$\Fq$ has a root in that disc.  Hence $U\in\Fq^\times$, proving
\eqref{eq:one-variable-Tate-fixed-field}.

Finally write a Frobenius-fixed $H\in\operatorname{Frac}(\mathbb T_A)$
uniquely, using \eqref{eq:Tate-field-Noether-normalization}, as
$H=\sum_jH_je_j$ with $H_j\in F_0$.  Coefficient Frobenius fixes every
$e_j$, so uniqueness gives $H_j^{(1)}=H_j$ for all $j$.  Equation
\eqref{eq:one-variable-Tate-fixed-field} yields $H_j\in\Fq(t)$, and hence
$H\in K$.  The reverse inclusion is fixed pointwise, completing the proof.
\end{proof}

\subsection{The characteristic orbit and weighted specialization}
\label{sec:orbit}

The \emph{orbit} of the characteristic point is the sequence
$\Xi^{(n)}\in X(\Cinf)$, $n\geqslant0$, obtained by twisting the
$\Cinf$-coefficients of $\Xi$ $n$ times.  Write
\begin{equation}\label{eq:exponential-coefficients}
 \Exp(w)=\sum_{n\geqslant0}e_n(\varphi)w^{q^n},
 \qquad e_0(\varphi)=1,
\end{equation}
and recall that
$F_\varphi^{[n]}=f_\varphi f_\varphi^{(1)}\cdots
f_\varphi^{(n-1)}$, with $F_\varphi^{[0]}=1$.

\begin{theorem}
\label{thm:orbital-exponential-expansion}
For every rank-one geometric shtuka realization fixed in
\S\ref{sec:realization},
\begin{equation}\label{eq:standard-shtuka-exp-coefficient}
 e_n(\varphi)
 =\left.(F_\varphi^{[n]})^{-1}\right|_{\Xi^{(n)}}
 \qquad(n\geqslant0),
\end{equation}
where the right side is the value of the rational function
$(F_\varphi^{[n]})^{-1}$ and is allowed to be zero.  Moreover,
\begin{equation}\label{eq:automatic-exponential-coefficient-decay}
 |e_n(\varphi)|^{1/q^n}\longrightarrow0.
\end{equation}
\end{theorem}

\begin{proof}
 Formula~\eqref{eq:standard-shtuka-exp-coefficient} is the standard
exponential-coefficient formula for rank-one shtukas.  See
\cite[Proposition~0.3.6]{Tha93}, \cite[\S8.2]{Tha04}, and
\cite[\S3.1]{ANDTR17a}.  For the realization on elliptic curves, see
\cite[Theorem~3.4]{GP18}.  The product is formed in the geometric
function field $\Cinf(X)$ before specialization and does not require
$f_\varphi$ to belong to the smaller ring $K\otimes_{\Fq}\Cinf$.  Since
$\Exp$ is entire, for every $R>0$ one has
$|e_n(\varphi)|R^{q^n}\to0$.  Taking $q^n$-th roots and then allowing $R$ to
be arbitrarily large proves
\eqref{eq:automatic-exponential-coefficient-decay}.
\end{proof}

\begin{remark}
\label{rem:zero-exponential-coefficients}
It is expected that $e_n(\varphi)\neq0$ for every $n\geqslant0$, but this is not proved in general.  Thakur proved that a vanishing $e_n(\varphi)=0$ is possible only for $n\in\{2,\dots,g-1\}$, where $g$ is the genus of $X$, and never occurs when $K$ has a place of degree one \cite[\S3.1]{ANDTR17a}, \cite[proof of Thm.~3.2]{Tha93}. In particular all $e_n(\varphi)$ are nonzero in the Carlitz and elliptic models below. The weighted specialization of
Definition~\ref{def:weighted-orbital-specialization} also covers the case
$e_n(\varphi)=0$. If $F_\varphi^{[n]}$ is regular and nonzero at
$\Xi^{(n)}$, then one may put
\[
 D_n^\varphi:=F_\varphi^{[n]}(\Xi^{(n)})
 =\prod_{i=1}^{n}
   f_\varphi(\Xi^{(i)})^{q^{n-i}}
\]
 In this situation, write $e_n(\varphi)=(D_n^\varphi)^{-1}$.
Evaluating each factor separately requires $f_\varphi$ to be regular at every
$\Xi^{(i)}$, $1\leqslant i\leqslant n$.  Otherwise, read the product as a formal
rational identity. This reciprocal notation applies
to the Carlitz and elliptic models used below; elsewhere we retain
$e_n(\varphi)$.
\end{remark}

\begin{lemma}\label{lem:orbit-escape}
The characteristic orbit $\{\Xi^{(n)}\}_{n\geqslant0}$ is injective and exits
every affinoid subdomain of $U_{\Cinf}^{\mathrm{an}}$.
\end{lemma}
\begin{proof}
Choose a nonconstant $p\in A$ with $|\iota(p)|>1$.  Since
$p(\Xi^{(n)})=\iota(p)^{q^n}$, these values are distinct and unbounded.  Every
analytic function is bounded on an affinoid, so the orbit is injective and
eventually leaves every affinoid in $U_{\Cinf}^{\mathrm{an}}$.
\end{proof}

\begin{lemma}
\label{lem:weighted-orbit-divisor-cancellation}
Put $m_n=\operatorname{mult}_{\Xi^{(n)}}(V_\varphi)$.  Then
\begin{equation}\label{eq:F-product-orbit-order}
 \operatorname{ord}_{\Xi^{(n)}}(F_\varphi^{[n]})=-m_n.
\end{equation}
Consequently \eqref{eq:weighted-orbital-specialization} is defined whenever
the first-variable pole of $g$ at $\Xi^{(n)}$ is bounded by $V_\varphi$.
\end{lemma}

\begin{proof}
The case $n=0$ follows from $\Xi\notin\supp(V_\varphi)$.  For $n\geqslant1$,
the characteristic and infinity terms in \eqref{eq:divf} do not occur at
$\Xi^{(n-j)}$ for $0\leqslant j<n$.
 (here we use the injectivity of the characteristic orbit and its disjointness from the infinity orbit, proved in Lemma~\ref{lem:orbit-escape})  Simultaneous twisting of the divisor and the
point therefore gives
\[
 \operatorname{ord}_{\Xi^{(n)}}(f_\varphi^{(j)})
 =m_{n-j-1}-m_{n-j}.
\]
Summing over $j=0,\ldots,n-1$ telescopes to
$m_0-m_n=-m_n$, which proves \eqref{eq:F-product-orbit-order}.  Thus
$(F_\varphi^{[n]})^{-1}$ has a zero of order $m_n$, exactly the order needed
to cancel a pole bounded by $V_\varphi$.
\end{proof}


\section{Motive evaluation}
\label{sec:evaluation}
Rational motive evaluation applies an additive polynomial to an exponential
value.  This section develops three evaluations with a common
domain: the rational evaluation of Section~\ref{sec:levelI}, the completed
separated-tensor evaluation $\widehat{\mathcal{E}}^{\mathrm{sep}}_{\varphi,u}$ of
Section~\ref{sec:separated-tensor}, and the orbital evaluation along the
characteristic orbit of Section~\ref{sec:orbital-extension}.  A kernel that
expands as a separated-tensor series on a disk is evaluated by
$\widehat{\mathcal{E}}^{\mathrm{sep}}_{\varphi,u}$; a general diagonal Cauchy
kernel (one that admits no such expansion, or whose first
variable has poles bounded by $V_\varphi$) requires the orbital evaluation
$\mathcal{E}^{\Xi}_{\varphi,u}$.  The three constructions agree on their
common domain.

\subsection{The rational evaluation}\label{sec:levelI}

\begin{definition}\label{def:levelI}
For $u\in\Cinf$, the \emph{rational motive exponential evaluation} is the
$\Cinf$-linear map
\[
\mathcal{E}^{\mathrm{rat}}_{\varphi,u}\colon \cM_{\varphi,K}\longrightarrow\Cinf\]
defined as
\[ 
\frac{1}{a}m\;\longmapsto\;
m\Bigl(\Exp\Bigl[\frac{u}{\iota(a)}\Bigr]\Bigr),
\]
where $m\in\cM_\varphi$, $a\in A\smallsetminus\{0\}$, and $m(\cdot)$ is the additive
polynomial of $m$.
\end{definition}

\begin{lemma}
\label{lem:levelI}
If $x=m/a$ and $m=\sum_{i=0}^N\alpha_i\tor^i$ as an additive polynomial,
then the geometric generic-frame coordinate of $x$ is
\begin{equation}\label{eq:localized-geometric-coordinate}
 g_x(\one)=\frac{\sum_{i=0}^N\alpha_iF_\varphi^{[i]}(\one)}{a(\one)},
 \qquad x=g_xs_\varphi,
\end{equation}
and no global freeness of $\cM_\varphi$ is used.
\end{lemma}

\begin{proof}
Equations~\eqref{eq:geometric-motive-coordinate} and
\eqref{eq:a-shtuka-coordinate} give
\eqref{eq:localized-geometric-coordinate}.
\end{proof}

\begin{lemma}
\label{lem:rat-well-defined}
$\mathcal{E}^{\mathrm{rat}}_{\varphi,u}$ is well defined.  If
$x=m/a$ and $m=\sum_{i=0}^N\alpha_i\tor^i$ as an additive polynomial,
then
\begin{equation}\label{eq:coordinate}
 \mathcal E^{\mathrm{rat}}_{\varphi,u}(x)
 =\sum_{i=0}^N\alpha_i
   \Exp\!\left(\frac{u}{\iota(a)}\right)^{q^i}.
\end{equation}
\end{lemma}

\begin{proof}
If $m/a=m'/a'$, then $a'\cdot m=a\cdot m'$ in the motive, where the
$A$-action on additive polynomials is precomposition.  Evaluating the two
sides at $\Exp(u/(\iota(a)\iota(a')))$ and using exponential equivariance
gives
\[
 m\bigl(\Exp(u/\iota(a))\bigr)
 =m'\bigl(\Exp(u/\iota(a'))\bigr).
\]
Thus Definition~\ref{def:levelI} is independent of the presentation.
Applying the additive polynomial $m$ to $\Exp(u/\iota(a))$ gives
\eqref{eq:coordinate}.
\end{proof}

\begin{example}[Carlitz]\label{ex:carlitz-levelI}
For $A=\Fq[t]$ and the Carlitz module,
$\mathcal{E}^{\mathrm{rat}}_{C,u}(c(t)s_C)
=\exp_C(c(\theta)u)$ for $c\in K$.  For a general
$g=\sum_\nu\alpha_\nu c_\nu\in K\otimes_{\Fq}\Cinf$, one instead has
\[
 \mathcal{E}^{\mathrm{rat}}_{C,u}(g s_C)
 =\sum_\nu\alpha_\nu\exp_C(c_\nu(\theta)u);
\]
the single-exponential formula cannot be extended by treating $\exp_C$ as
$\Cinf$-linear.
\end{example}

\begin{proposition} \label{prop:certificates}
Let $u\in\Cinf$. The rational evaluation
$\mathcal{E}^{\mathrm{rat}}_{\varphi,u}$ satisfies:
\begin{enumerate}[label={\textup{(\roman*)}},leftmargin=2.2em]
\item \emph{(Normalization)} $\mathcal{E}^{\mathrm{rat}}_{\varphi,u}(s_\varphi)=\Exp u$.
\item \emph{(Frobenius interchange)}
\[
\mathcal{E}^{\mathrm{rat}}_{\varphi,u}\bigl(\tor_{\one}(m)\bigr)=\bigl(\mathcal{E}^{\mathrm{rat}}_{\varphi,u}(m)\bigr)^{(1)},
\]
where $\tor_{\one}(c\,s_\varphi)=f_{\varphi}(\one)\,c^{(1)}\,s_\varphi$.
\item \emph{(Isogeny naturality)} for every $m\in\cM_{\psi,K}$,\begin{equation}\label{eq:full-isogeny-naturality}\mathcal{E}^{\mathrm{rat}}_{\varphi,u}\bigl(\eta_J^!\,m\bigr)=\mathcal{E}^{\mathrm{rat}}_{\psi,h_J(\Xi)u}(m),\end{equation} where $\eta_J^!n:=n\circ\eta_J$ is the contravariant motive map.
\end{enumerate}
\end{proposition}

\begin{proof}
(i) is the case $a=1$, $m=1$ of Definition~\ref{def:levelI}.

For (ii),
write $x=m/a$ with $m\in\cM_\varphi$ and $0\ne a\in A$.  The operator
$\tau$ acts on the additive-polynomial realization by postcomposition with
$x\mapsto x^q$ and commutes with the $A$-action by precomposition.  Therefore
\[
 \mathcal E^{\mathrm{rat}}_{\varphi,u}(\tau_Ix)
 =(\tau m)\bigl(\Exp(u/\iota(a))\bigr)
 =m\bigl(\Exp(u/\iota(a))\bigr)^q
 =\mathcal E^{\mathrm{rat}}_{\varphi,u}(x)^q.
\]

For (iii),
write $m=x/a$ with $x\in\cM_\psi$ and
$a\in A\smallsetminus\{0\}$, and let $P_x$ be the additive polynomial
corresponding to $x$.  Under the motive--additive-polynomial dictionary,
$\eta_J^! x$ corresponds to $P_x\circ\eta_J$.  Therefore
\begin{align*}
 \mathcal E^{\mathrm{rat}}_{\varphi,u}
  \bigl(\eta_J^! (x/a)\bigr)
 &=P_x\!\left(
    \eta_J\!\left(\exp_\varphi
      \left(\frac{u}{\iota(a)}\right)\right)\right)\\
 &=P_x\!\left(\exp_\psi
      \left(\frac{h_J(\Xi)u}{\iota(a)}\right)\right)
  =\mathcal E^{\mathrm{rat}}_{\psi,h_J(\Xi)u}(x/a),
\end{align*}
which proves \eqref{eq:full-isogeny-naturality}.
\end{proof}

\begin{remark}\label{rem:classical-identity-isogeny}
For $m=s_\psi$, item (iii) specializes to the classical identity
$\mathcal{E}^{\mathrm{rat}}_{\varphi,u}(h_Js_\varphi)=\eta_J(\Exp u)$: indeed
$\eta_J^!(s_\psi)=h_Js_\varphi$ by \eqref{eq:intertwiner}, and
\[
\mathcal{E}^{\mathrm{rat}}_{\psi,h_J(\Xi)u}(s_\psi)=\exp_\psi(h_J(\Xi)u)=\eta_J(\exp_\varphi u).
\]
\end{remark}

\begin{lemma}\label{lem:period-annihilator}
For $x\in\cM_{\varphi,K}$,
$\mathcal E^{\mathrm{rat}}_{\varphi,\lambda}(x)=0$
for every $\lambda\in\Lambda_\varphi$  precisely when 
$x\in\cM_\varphi$.
\end{lemma}

\begin{proof}
An integral motive element vanishes after evaluation at a period.  Conversely,
write $x=m/a$ with $m\in\cM_\varphi$ and $0\ne a\in A$.  Uniformization
identifies
\[
 \iota(a)^{-1}\Lambda_\varphi/\Lambda_\varphi
 \xrightarrow{\ \sim\ }\varphi[a],\qquad
 \lambda/\iota(a)\longmapsto
 \Exp\bigl(\lambda/\iota(a)\bigr).
\]
Thus the additive polynomial $m$ vanishes on every root of the separable
polynomial $\varphi_a$.  Right division in $\Cinf\{\tau\}$ gives
$m=q\circ\varphi_a+r$ with $\deg_\tau r<\deg_\tau\varphi_a$.  The remainder
$r$ vanishes at all roots of $\varphi_a$, and root counting forces $r=0$.
 The quotient is integral automatically: $\cM_\varphi\simeq\Cinf\{\tau\}$ is free of rank one by Lemma~\ref{lem:direct-geometric-shtuka-Frobenius-basis}\ref{item:Frobenius-cyclic-motive}, so every right-division quotient in $\Cinf\{\tau\}$ lies in $\cM_\varphi$.
Hence $m\in a\cM_\varphi$ for the precomposition action of $A$ on the motive,
so $x\in\cM_\varphi$.
\end{proof}

\begin{lemma}
\label{lem:isogeny-period-annihilator-main}
Let $\eta:\varphi\to\psi$ be a separable isogeny, put
$c=\partial\eta\ne0$, and let
\[
 \eta^!:\cM_\psi\hookrightarrow\cM_\varphi,
 \qquad n\longmapsto n\circ\eta,
\]
be the contravariant motive map.   {For $x\in\cM_{\varphi,K}$, the vanishing $\mathcal E^{\mathrm{rat}}_{\varphi,w}(x)=0$ for every $w\in c^{-1}\Lambda_\psi$ \revblue{holds exactly when} $x\in\eta^!(\cM_\psi)$.}
\end{lemma}

\begin{proof}
If $x=n\circ\eta$, exponential compatibility gives
\[
 \mathcal E^{\mathrm{rat}}_{\varphi,w}(x)
 =n\bigl(\eta(\Exp(w))\bigr)
 =n\bigl(\exp_\psi(cw)\bigr)=0.
\]
Conversely, write $x=m/a$ with $m\in\cM_\varphi$ and $0\ne a\in A$.
As $w$ ranges over $c^{-1}\Lambda_\psi$, the points
$\Exp(w/\iota(a))$ range over
\[
 \ker(\eta\circ\varphi_a)
 \simeq c^{-1}\iota(a)^{-1}\Lambda_\psi/\Lambda_\varphi .
\]
 Indeed, $w\mapsto\Exp(w/\iota(a))$ has kernel $\iota(a)\Lambda_\varphi$, and surjectivity onto $\ker(\eta\circ\varphi_a)$ follows by comparing cardinalities: both sides have order $q^{\deg_\tau\eta+\deg_\tau\varphi_a}$, since $c^{-1}\Lambda_\psi/\Lambda_\varphi$ represents $\ker\eta$.
The polynomial $\eta\circ\varphi_a$ is separable.  Right division and root
counting therefore give $m=q\circ\eta\circ\varphi_a$ for some
$q\in\cM_\psi$  {(the quotient is again integral by the rank-one freeness of Lemma~\ref{lem:direct-geometric-shtuka-Frobenius-basis}\ref{item:Frobenius-cyclic-motive})}.  Localizing by $a$ yields
$x=q\circ\eta=\eta^!(q)$.
\end{proof}

\subsection{The completed separated-tensor domain}\label{sec:separated-tensor}

The second-factor coefficients require a completion adapted to motive
evaluation. Let $B$ be a nonarchimedean Banach
$\Cinf$-vector space; in particular, it may be a Banach algebra or a finite
Banach module over one.  On a finite separated tensor put
\begin{equation}\label{eq:finite-tensor-evaluation-map}
 \mathcal{E}^{\mathrm{sep}}_{\varphi,u}\!\left(\sum_{j=1}^N x_j\otimes b_j\right)
 :=\sum_{j=1}^N
   \mathcal E^{\mathrm{rat}}_{\varphi,u}(x_j)b_j.
\end{equation}
For $R>0$ and $x\in\cM_{\varphi,K}$ put
\begin{equation}\label{eq:evaluation-seminorm}
 p_R(x):=\sup_{|u|\leqslant R}
 \bigl|\mathcal E^{\mathrm{rat}}_{\varphi,u}(x)\bigr|.
\end{equation}
This is finite because the right-hand side is the supremum norm of an entire
function on a closed disk.  On the algebraic tensor product there are two
useful seminorms:
\begin{equation}\label{eq:projective-evaluation-seminorm}
 q_R(F):=\sup_{|u|\leqslant R}\lVert\mathcal{E}^{\mathrm{sep}}_{\varphi,u}(F)\rVert_B,
 \qquad
 \pi_R(F):=
 \inf_{F=\sum_{j=1}^N x_j\otimes b_j}
 \max_{1\leqslant j\leqslant N}p_R(x_j)\lVert b_j\rVert_B .
\end{equation}
The ultrametric inequality gives $q_R\leqslant\pi_R$.  The first seminorm records
exactly the values that motive evaluation can detect; the second is a
convenient sufficient estimate for convergence.

\begin{notation}\label{def:evaluation-completion}
The notation
$
 \cM_{\varphi,K}\widehat\otimes_{\mathrm{ev}}B
$
means the Hausdorff completion of
$\cM_{\varphi,K}\otimes_{\Cinf}B$ for the family $\{q_R\}_{R>0}$.  Thus it
is the closure of the finite separated-tensor evaluations inside the space
of entire $\Fq$-linear $B$-valued functions, endowed with uniform
convergence on bounded disks.  The projective completion for
$\{\pi_R\}_{R>0}$ maps continuously to this evaluation completion.  In
particular, taking
$B=\mathbb T_A$ gives a specified interpretation of the
completed separated-tensor domain occurring in motive evaluation.
\end{notation}

\begin{definition}\label{def:completed-sep-evaluation}
 For $u\in\Cinf$, the completed separated-tensor evaluation is the continuous extension 
\[
 \widehat{\mathcal E}^{\mathrm{sep}}_{\varphi,u}:
 \cM_{\varphi,K}\widehat\otimes_{\mathrm{ev}}B\longrightarrow B
\]
of finite-tensor evaluation.
\end{definition}

\begin{proposition}
\label{prop:continuous-tensor-evaluation}
The completed separated-tensor evaluation is well defined: for
every $u\in\Cinf$, finite-tensor evaluation extends uniquely to
$\widehat{\mathcal E}^{\mathrm{sep}}_{\varphi,u}$, its norm on a finite
tensor is bounded by $q_R$ for $R\geqslant|u|$, and the extended evaluations
jointly separate points.
\end{proposition}
\begin{proof}
The defining inequality
$\|\mathcal{E}^{\mathrm{sep}}_{\varphi,u}(F)\|_B\leqslant q_R(F)$ gives a continuous extension to
the Hausdorff completion.  A Cauchy family for all $q_R$ has a uniform
limit on every bounded parameter disk; these limits agree on overlaps.
They realize the completion as a space of $B$-valued functions.
If every evaluation of a limit is zero, each of its supremum seminorms
is zero, and Hausdorffness makes the limit zero.
\end{proof}

\subsection{Automatic orbital estimates}\label{sec:hypotheses}

\begin{definition}
\label{def:weighted-orbital-specialization}
For a meromorphic first-variable coefficient $g(\one,\two)$ for which the
displayed quotient is regular at $\one=\Xi^{(n)}$, put
\begin{equation}\label{eq:weighted-orbital-specialization}
 \operatorname{orb}_{\varphi,n}(g)
 :=\left.
   \frac{g(\one,\two)}{F_\varphi^{[n]}(\one)}
  \right|_{\one=\Xi^{(n)}}.
\end{equation}
If $g$ is regular at $\Xi^{(n)}$ and
$\Xi^{(n)}\notin\operatorname{supp}(V_\varphi)$, then
$\operatorname{orb}_{\varphi,n}(g)
=e_n(\varphi)g(\Xi^{(n)},\two)$.
\end{definition}

\begin{proposition}
\label{thm:automatic-orbit-growth}
Let $W\subset
U_{\Cinf}^{\mathrm{an}}$ be an affinoid and let $R(\one,\two)$ be
a fixed rational function on $X_{\Cinf}\times X_{\Cinf}$ which is
regular on $\{\Xi^{(n)}:n\notin E\}\times W$ for some finite set $E$. Assume moreover that the poles of $R$ on
$U_{\Cinf}\times W$ are bounded by a fixed effective divisor $D$ on
$X_{\Cinf}$ plus the diagonal: the polar divisor of $R$ on
$U_{\Cinf}\times W$ is contained in $\operatorname{pr}_{\one}^*D+\Delta$.
Write $\|\cdot\|_W$ for the supremum norm on $W$.  The
following hold.
\begin{enumerate}[label=\textup{(\roman*)},leftmargin=2.2em]
\item there is a constant $C_W\geqslant1$ such that
\begin{equation}\label{eq:automatic-algebraic-orbit-growth}
 \bigl\|R(\Xi^{(n)},\cdot)\bigr\|_W\leqslant C_W^{q^n}
 \qquad(n\notin E).
\end{equation}
\item \emph{(Vertical regularity.)}  Suppose that for every
exceptional index $n\in E$ with $\Xi^{(n)}\notin W$,
\[
 {\operatorname{ord}_{\{\Xi^{(n)}\}\times X_{\Cinf}}(R)}
 \geqslant-\operatorname{mult}_{\Xi^{(n)}}(V_\varphi),
\]
{the order being taken along the vertical prime divisor
$\{\Xi^{(n)}\}\times X_{\Cinf}$.}
Then for every $n$ with $\Xi^{(n)}\notin W$ the quotient
$R/F_\varphi^{[n]}$ is regular along $\{\Xi^{(n)}\}\times W$, and
$\operatorname{orb}_{\varphi,n}(R)$ is analytic on $W$;
\item for $n\notin E$ one has
\begin{equation}\label{eq:automatic-weighted-orbit-growth}
 \bigl\|\operatorname{orb}_{\varphi,n}(R)\bigr\|_W
 \leqslant |e_n(\varphi)|C_W^{q^n}.
\end{equation}
\end{enumerate}
\end{proposition}

\begin{proof}
\emph{(i).}  {Choose a nonconstant $p\in A$ with
$|\iota(p)|>1$.  Enlarge $D$ at the points above infinity to bound
the actual first-variable poles of $R$ there; this does not change
the hypothesis on $U_{\Cinf}\times W$.  Write this part of $D$ as
$\sum_j\nu_j[q_j]$.  Around each $q_j$, choose a sufficiently small
closed affinoid disk $U_j$, disjoint from $W$ and from the other
points of $\operatorname{supp}(D)$, with local parameter $s_j$.
Then $s_j^{\nu_j}R$ is analytic and bounded on $U_j\times W$.
Since $p$ has a pole at each $q_j$, there are constants $B\geqslant1$ and
$M\geqslant0$, independent of $j$, such that}
\[
 \|R(P,\cdot)\|_W\leqslant B\max\{1,|p(P)|\}^{M}
\]
for $P\in U_j\setminus\{q_j\}$.  Since
$|p(\Xi^{(n)})|=|\iota(p)|^{q^n}\to\infty$, the points $\Xi^{(n)}$
eventually lie in these disks.  The displayed bound is then at most
$B|\iota(p)|^{Mq^n}\leqslant C_W^{q^n}$ after enlarging $C_W$.
The finitely many remaining indices outside $E$ are absorbed by
the same constant.

\smallskip
\emph{(ii).}  Put $m_n:=\operatorname{mult}_{\Xi^{(n)}}(V_\varphi)$.
By Lemma~\ref{lem:weighted-orbit-divisor-cancellation},
$\operatorname{ord}_{\Xi^{(n)}}F_\varphi^{[n]}=-m_n$, independently of
$\two$, so division by $F_\varphi^{[n]}$ contributes $+m_n$ to the order
along $\one=\Xi^{(n)}$.  For $n\notin E$ the function $R$ is regular on
$\{\Xi^{(n)}\}\times W$ by assumption; for $n\in E$ with
$\Xi^{(n)}\notin W$, the displayed bound gives
$\operatorname{ord}_{\{\Xi^{(n)}\}\times X_{\Cinf}}(R/F_\varphi^{[n]})
\geqslant-m_n+m_n=0$.  Division by $F_\varphi^{[n]}$, which depends
only on the first variable, introduces only vertical polar
components.  After the component centered at $\Xi^{(n)}$ has been
removed by this order bound, all remaining polar components are
vertical fibers with different centers, or lie on $\Delta$.
None meets $\{\Xi^{(n)}\}\times W$, since $\Xi^{(n)}\notin W$.
In both cases
$R/F_\varphi^{[n]}$ is regular along the whole vertical line
$\{\Xi^{(n)}\}\times W$.  \revblue{Since the quotient is rational, its polar locus is closed and misses
this compact line.  It is therefore regular on a tube around the line.
Specialization at $\one=\Xi^{(n)}$ is consequently analytic on $W$.}  By Definition~\ref{def:weighted-orbital-specialization} this
specialization is $\operatorname{orb}_{\varphi,n}(R)$.

\smallskip
\emph{(iii).}  For $n\notin E$, $R$ is regular on
$\{\Xi^{(n)}\}\times W$ by hypothesis, and
$(F_\varphi^{[n]})^{-1}$ is regular at $\Xi^{(n)}$ by
Lemma~\ref{lem:weighted-orbit-divisor-cancellation}.
Thus the quotient is regular along this fiber, whether or not
$\Xi^{(n)}\in W$.  If $\Xi^{(n)}\notin\operatorname{supp}(V_\varphi)$,
Definition~\ref{def:weighted-orbital-specialization} reads
$\operatorname{orb}_{\varphi,n}(R)=e_n(\varphi)\,R(\Xi^{(n)},\cdot)$, and
part (i) gives \eqref{eq:automatic-weighted-orbit-growth}.  If
$\Xi^{(n)}\in\operatorname{supp}(V_\varphi)$, then
$(F_\varphi^{[n]})^{-1}$ vanishes along $\one=\Xi^{(n)}$, so
$\operatorname{orb}_{\varphi,n}(R)=0$ and
\eqref{eq:automatic-weighted-orbit-growth} holds trivially.
\end{proof}

Let $W\subset U_{\Cinf}^{\mathrm{an}}$ be an affinoid disjoint
from the characteristic orbit and from the second-factor poles under
consideration.

\begin{definition}
\label{def:orbital-domain}
The coefficient $g$ is \emph{weighted-orbit-admissible on $W$}
if the following hold.
\begin{enumerate}[label=\textup{(\Alph*)},leftmargin=2.6em]
\item \emph{Specialization is well defined.}  There is a finite set
$E_g\subset\mathbb N$ such that
\begin{itemize}[leftmargin=2.4em]
\item[\textup{(A1)}] for $n\notin E_g$, $g$ is regular at
$\one=\Xi^{(n)}$ and
$\operatorname{orb}_{\varphi,n}(g)=e_n(\varphi)\,g(\Xi^{(n)},\two)$;
\item[\textup{(A2)}] for $n\in E_g$, the first-variable pole of $g$ at
$\one=\Xi^{(n)}$ is bounded by $V_\varphi$
($\operatorname{ord}_{\Xi^{(n)}}(g)\geqslant-m_n$), and the quotient
$g(\one,\two)/F_\varphi^{[n]}(\one)$ is regularized at $\one=\Xi^{(n)}$ by
the zero of $(F_\varphi^{[n]})^{-1}$, which defines $\operatorname{orb}_{\varphi,n}(g)$
(Lemma~\ref{lem:weighted-orbit-divisor-cancellation}).
\end{itemize}
\item \emph{Range condition.}  For every $n\geqslant0$,
$\operatorname{orb}_{\varphi,n}(g)\in\mathcal O(W)$.
\item \emph{Growth condition.}  There is $C_W\geqslant1$ such that
\[
 \lVert g(\Xi^{(n)},\cdot)\rVert_W\leqslant C_W^{q^n}
 \qquad(n\notin E_g).
\]
\end{enumerate}
Here $m_n=\operatorname{mult}_{\Xi^{(n)}}(V_\varphi)$ and the norm is
the sup norm on $W$.
\end{definition}

{By Lemma~\ref{lem:SzegoC}, the poles of $K_\varphi$ on
$U_{\Cinf}\times W$ are bounded by
$\operatorname{pr}_{\one}^*V_\varphi+\Delta$.
Proposition~\ref{thm:automatic-orbit-growth} therefore gives
its weighted-orbit admissibility on $W$.
The isogeny-adapted kernels and their weighted twists are treated
in Lemma~\ref{lem:Admissibility}.}

\subsection{The local orbital extension}\label{sec:orbital-extension}

The rational evaluation and its Pellarin-series refinement evaluate separated first-factor data.  A diagonal Cauchy
kernel need not belong to an unspecified completed tensor product, so we define
its evaluation directly on affinoids $W$ in the second variable, as in
Definition~\ref{def:orbital-domain}.

\begin{definition} \label{def:levelIII}
For a weighted-orbit-admissible coefficient $g$ define
\begin{equation}\label{eq:orbital-evaluation}
 \mathcal E^\Xi_{\varphi,u}(g\,s_\varphi)
 :=\sum_{n\geqslant0}u^{q^n}\operatorname{orb}_{\varphi,n}(g).
\end{equation}
\end{definition}

 \begin{proposition} \label{lem:orbital-evaluation-convergence}
 \revblue{For a weighted-orbit-admissible coefficient $g$, the series
 \eqref{eq:orbital-evaluation} converges normally on $W$ for every $u\in\Cinf$.
 Consequently, $u\mapsto\mathcal E^\Xi_{\varphi,u}(g\,s_\varphi)$ is an
 entire $\mathcal O(W)$-valued function of $u$.}
 \end{proposition} 
\begin{proof}
Let $E_g$ be the finite exceptional set of
Definition~\ref{def:orbital-domain}.  By
\eqref{eq:weighted-orbital-specialization} and
\eqref{eq:standard-shtuka-exp-coefficient},
\begin{equation}\label{eq:orbital-component-core-identity}
 {\operatorname{orb}_{\varphi,n}(g)
 =e_n(\varphi)\,g(\Xi^{(n)},\two)\qquad(n\notin E_g).}
\end{equation}
{Thus the orbital evaluation series reads\revblue{:}
\[
 \mathcal E^\Xi_{\varphi,u}(g\,s_\varphi)
 =\sum_{n\in E_g}u^{q^n}\operatorname{orb}_{\varphi,n}(g)
 +\sum_{n\notin E_g}e_n(\varphi)\,u^{q^n}\,g(\Xi^{(n)},\cdot).
\]}

{{For the same exceptional set $E_g$ and $n\notin E_g$,} $g$ is
regular at $\one=\Xi^{(n)}$ and
\[
 \bigl\|u^{q^n}\operatorname{orb}_{\varphi,n}(g)\bigr\|_W
 =|e_n(\varphi)|\cdot|u|^{q^n}\cdot\bigl\|g(\Xi^{(n)},\cdot)\bigr\|_W
 \leqslant|e_n(\varphi)|\,(|u|\,C_W)^{q^n}.
\]
\revblue{Fix $R>0$.  By \eqref{eq:automatic-exponential-coefficient-decay},
choose $N$ such that $|e_n(\varphi)|\leqslant(2RC_W)^{-q^n}$ for $n\geqslant N$.
If $|u|\leqslant R$, the $n$-th term has norm at most $2^{-q^n}$.  Since
$\sum_{n\geqslant0}2^{-q^n}<\infty$, the series converges normally on
$\mathcal O(W)$, uniformly for $|u|\leqslant R$.}  The finitely many
exceptional indices are regularized by $(F_\varphi^{[n]})^{-1}$
(Definition~\ref{def:orbital-domain}(A2)) and do not affect
convergence.  Each term is an entire $\mathcal O(W)$-valued
$\Fq$-linear function of $u$, and the convergence is uniform on bounded
disks; hence $u\mapsto\mathcal E^\Xi_{\varphi,u}(g\,s_\varphi)$ is an
entire $\mathcal O(W)$-valued $\Fq$-linear function of $u$.}
\end{proof}

\begin{theorem}[Weighted Pellarin identity on the localized motive]
\label{thm:pellarin}
Let $x\in\cM_{\varphi,K}$ and let $g_x$ be its geometric coordinate
\eqref{eq:localized-geometric-coordinate}.  Every regularized specialization
$\operatorname{orb}_{\varphi,n}(g_x)$ is defined, the following series
converges for every $u\in\Cinf$, and
\begin{equation}\label{eq:pellarin}
 \mathcal E^{\mathrm{rat}}_{\varphi,u}(x)
 =\sum_{n\geqslant0}u^{q^n}
   \operatorname{orb}_{\varphi,n}(g_x).
\end{equation}
More generally, if $g$ is weighted-orbit-admissible on $W$ and
$g(\one,\two)=\sum_{j=1}^N g_j(\one)b_j(\two)$ with
$g_js_\varphi\in\cM_{\varphi,K}$ and $b_j\in\mathcal O(W)$, then orbital
evaluation equals the $\mathcal O(W)$-linear extension of rational
evaluation:
\[
 \mathcal E^\Xi_{\varphi,u}(g\,s_\varphi)
 =\sum_{j=1}^N
 \mathcal E^{\mathrm{rat}}_{\varphi,u}(g_j s_\varphi)b_j.
\]
The same equality holds for a convergent separated-tensor
expansion whenever its chosen completed tensor topology makes rational
evaluation continuous and permits termwise evaluation.
\end{theorem}

\begin{proof}
Write $x=m/a$ with
$m=\sum_{i=0}^N\alpha_i\tor^i$.  It suffices to treat the coordinate
$F_\varphi^{[i]}/a(\one)$.  For $n\geqslant i$, the factorization
\[
 F_\varphi^{[n]}
 =F_\varphi^{[i]}\bigl(F_\varphi^{[n-i]}\bigr)^{(i)}
\]
and $a(\Xi^{(n)})=\iota(a)^{q^n}$ give
\begin{equation}\label{eq:orbital-monomial-coordinate}
 \operatorname{orb}_{\varphi,n}
  \left(\frac{F_\varphi^{[i]}}{a(\one)}\right)
 =\frac{e_{n-i}(\varphi)^{q^i}}{\iota(a)^{q^n}}.
\end{equation}
For $n<i$, put
$m_k=\operatorname{mult}_{\Xi^{(k)}}(V_\varphi)$ for $k\in\mathbb Z$.
The same divisor calculation as in
Lemma~\ref{lem:weighted-orbit-divisor-cancellation} gives
\[
 \operatorname{ord}_{\Xi^{(n)}}
 \left(\frac{F_\varphi^{[i]}}{F_\varphi^{[n]}}\right)
 =\sum_{k=n-i+1}^{0}(m_{k-1}-m_k+\mathbf1_{\{k=0\}})
 =m_{n-i}+1>0,
\]
where $m_0=0$ by the realization assumptions.
 \revblue{The summand $\mathbf 1_{\{k=0\}}$ is the $[\Xi]$-term of
$\divv(f_\varphi^{(n)})$.  No infinity term appears:
Lemma~\ref{lem:orbit-escape} shows that the characteristic orbit avoids the
infinity orbit and that the evaluation points are distinct.}  Thus this quotient is
regular and vanishes at $\Xi^{(n)}$.  Hence the left side of
\eqref{eq:orbital-monomial-coordinate} is zero in that range.  Therefore
\[
 \sum_{n\geqslant0}u^{q^n}\operatorname{orb}_{\varphi,n}
  \left(\frac{F_\varphi^{[i]}}{a(\one)}\right)
 =\left(\sum_{k\geqslant0}e_k(\varphi)
   \left(\frac{u}{\iota(a)}\right)^{q^k}\right)^{q^i}
 =\Exp\!\left(\frac{u}{\iota(a)}\right)^{q^i}.
\]
The series converges because $\Exp$ is entire.  Multiplication by the finite
set of coefficients $\alpha_i$ and Lemma~\ref{lem:rat-well-defined} prove
\eqref{eq:pellarin}.  The geometric coordinate has only the
$V_\varphi$-bounded orbit poles allowed by \eqref{eq:membership}; Lemma
\ref{lem:weighted-orbit-divisor-cancellation} therefore justifies all the
regularized specializations used above.
For $g(\one,\two)=\sum_j g_j(\one)b_j(\two)$ as in the statement,
applying \eqref{eq:pellarin} to each $g_js_\varphi\in\cM_{\varphi,K}$ and
interchanging the finite sums gives
$\mathcal E^\Xi_{\varphi,u}(g s_\varphi)
=\sum_j\mathcal E^{\mathrm{rat}}_{\varphi,u}(g_j s_\varphi)b_j$.
\end{proof}

If $g_x$ is regular at every orbit point, then
$\operatorname{orb}_{\varphi,n}(g_x)
=e_n(\varphi)g_x(\Xi^{(n)})$, so \eqref{eq:pellarin} is exactly the
Pellarin series $\sum_{n\geqslant0}e_n(\varphi)u^{q^n}g_x\bigl(\Xi^{(n)}\bigr)$.

\begin{theorem}\label{thm:coherence}
Let $W_0,W_1\subset U_{\Cinf}^{\mathrm{an}}$ be affinoids
such that coefficient Frobenius induces a continuous $q$-semilinear map
\[
 \mathcal O(W_0)\longrightarrow\mathcal O(W_1),\qquad
 b\longmapsto b^{(1)},
\]
and let $V\subset W_0\cap W_1$ be an affinoid.  If $g$ is
weighted-orbit-admissible on $W_0$ and $f_{\varphi}(\one)g^{(1)}$ is
weighted-orbit-admissible on $W_1$, then, after restriction to $V$,
\begin{equation}\label{eq:orbital-Frobenius-initial-term}
 \mathcal E^\Xi_{\varphi,u}
 \bigl(f_{\varphi}(\one)g^{(1)}s_\varphi\bigr)
 =\bigl(\mathcal E^\Xi_{\varphi,u}(g s_\varphi)\bigr)^{(1)}
   +u\,\beta_\varphi(g),
\end{equation}
where
\begin{equation}\label{eq:orbital-initial-coefficient}
 \beta_\varphi(g):=
 \left.(f_{\varphi}(\one)g^{(1)})\right|_{\one=\Xi}.
\end{equation}
In particular the two operations commute if $\beta_\varphi(g)=0$.
This holds whenever the first-variable pole of $g$ at $\Xi^{(-1)}$
is bounded by $V_\varphi$.
\end{theorem}
\begin{remark}
\revblue{Coefficient Frobenius need not preserve an arbitrary affinoid.  For
example, it maps a closed disk of radius $R$ to one of radius $R^q$.  We
therefore form the twisted series on the possibly different affinoid $W_1$ and
restrict the identity to the common piece $V$.}
\end{remark}

\begin{proof}
{Lemma~\ref{lem:orbital-evaluation-convergence} gives} normal convergence on $W_0$ and
continuity of coefficient Frobenius
$\mathcal O(W_0)\to\mathcal O(W_1)$ justify twisting the sum into
$\mathcal O(W_1)$.  Since
$F_\varphi^{[m]}=f_\varphi(F_\varphi^{[m-1]})^{(1)}$, regularized
specialization gives, for $m\geqslant1$,
\begin{equation}\label{eq:weighted-orbit-Frobenius-shift}
 \operatorname{orb}_{\varphi,m}(f_\varphi g^{(1)})
 =\operatorname{orb}_{\varphi,m-1}(g)^{(1)}.
\end{equation}
This identity is taken after the indicated cancellation and remains valid at
the exceptional Drinfeld-divisor indices.
 Indeed, after the rational cancellation $F_\varphi^{[m]}=f_\varphi(F_\varphi^{[m-1]})^{(1)}$, regularity of $g^{(1)}/(F_\varphi^{[m-1]})^{(1)}$ at $\Xi^{(m)}$ follows from $\operatorname{ord}_{\Xi^{(m)}}\bigl((F_\varphi^{[m-1]})^{(1)}\bigr)=-m_{m-1}$, which exactly cancels a $V_\varphi$-bounded pole of $g$ at $\Xi^{(m-1)}$.  Twisting
\eqref{eq:orbital-evaluation} and reindexing accounts for every term with
$m\geqslant1$.  The remaining term is $u\beta_\varphi(g)$, which proves
\eqref{eq:orbital-Frobenius-initial-term}.
If $m_{-1}=\operatorname{mult}_{\Xi^{(-1)}}(V_\varphi)$, then
\eqref{eq:divf} gives $\operatorname{ord}_\Xi(f_\varphi)=m_{-1}+1$.
Under the stated pole bound, $g^{(1)}$ has order at least $-m_{-1}$
there, so their product vanishes at $\Xi$.  Restriction to $V\subset W_0\cap W_1$ gives the
assertions on the common domain.
\end{proof}

\begin{remark}\label{rem:no-bifurcation}
The orbital construction enlarges the domain of evaluation; it does not
define a conflicting value on a specified overlap.  For example, on
$|t_{\two}|<|\theta|$ one has
\[
 \frac{1}{t_{\two}-t_{\one}}
 =-\sum_{j\geqslant0}\frac{t_{\two}^j}{t_{\one}^{j+1}},
\]
and completed tensor evaluation gives
\[
 -\sum_{j\geqslant0}\exp_C\!\left(\frac{u}{\theta^{j+1}}\right)t_{\two}^j
 =\sum_{n\geqslant0}\frac{u^{q^n}}{D_n^{C}(t_{\two}-\theta^{q^n})},
\]
which is exactly the orbital series.  
\end{remark}

\revsection{Cauchy kernels}{Cauchy kernels}
\label{sec:cauchy-gm-period-class}

The tautological differential coordinate $K_\varphi$ has first-variable
poles bounded by the Drinfeld divisor $V_\varphi$. This pole bound defines
a Cauchy fiber whose differences have integral-motive coefficients, so
period evaluation is constant on the whole fiber. The lifted Cauchy kernel
and $K_\varphi$ belong to this fiber and therefore have the same period
evaluation.

\subsection{\texorpdfstring{The relative different and trace-residue coordinates}{The relative different and trace-residue coordinates}}\label{subsec:GM}
Fix a separating element $u\in A$ and put
$\mathfrak{d}_{A/\Fq[u]}$ the relative differential ideal.  Set
\[
 D_u:=\{x\in A\otimes_{\Fq[u]}A:(a\otimes1-1\otimes a)x=0
 \text{ for every }a\in A\}.
\]
Multiplication on the two factors restricts to a map
$\mult_u:D_u\to A$.

\begin{proposition}
\label{prop:Du-generator-any}
Multiplication induces an isomorphism
\[
 \mult_u:D_u\xrightarrow{\;\sim\;}\mathfrak{d}_{A/\Fq[u]}.
\]
\end{proposition}

\begin{proof}
The standard different--diagonal identification
\cite[Proposition~3.5]{GM21} shows that multiplication maps $D_u$
onto $\mathfrak{d}_{A/\Fq[u]}$ and is injective.
\end{proof}

Let $u\in A$ be nonconstant and separating, put $B=\Fq[u]$ and $F=\Fq(u)$,
and choose a $B$-basis $e_1,\ldots,e_h$ of $A$.  The \emph{trace-dual
basis} $e_1',\ldots,e_h'$ in $K$ is defined by\revblue{:}
\begin{equation}\label{eq:trace-dual}
 \operatorname{Tr}_{K/F}(e_ie_j')=\delta_{ij}
 \qquad(1\leqslant i,j\leqslant h).
\end{equation}
\revblue{The trace pairing is nondegenerate because $K/F$ is finite separable.
Hence the dual basis exists and is unique, and it satisfies:}
\begin{equation}\label{eq:trace-dual-props}
 e_j'\in\mathfrak{d}_{A/\Fq[u]}^{-1}
 \qquad\text{and}\qquad
 \sum_{i=1}^h e_ie_i'=1,
\end{equation}
\revblue{This basis remains trace-dual after scalar extension to $\Cinf$.}

\revblue{Trace-dual bases make the isomorphism $\mult_u$ of
Proposition~\ref{prop:Du-generator-any} explicit.  Let $e_1,\ldots,e_h$ be an
$\Fq[u]$-basis of $A$, with trace-dual basis $e_1',\ldots,e_h'$.  For
$\mathfrak b\in\mathfrak{d}_{A/\Fq[u]}$, define the trace-dual lift:}
\begin{equation}\label{eq:different-diagonal-trace-class}
 \widehat{\mathfrak b}^{e}
 :=\sum_{r=1}^h(\mathfrak b e_r')\otimes e_r\in A\otimes_{\Fq}A,
\end{equation}
\revblue{Its class $[\widehat{\mathfrak b}^{e}]\in D_u$ does not depend on
the chosen trace-dual basis.}  Indeed
$e_r'\in\mathfrak{d}_{A/\Fq[u]}^{-1}$ makes the formula integral for
$\mathfrak b\in\mathfrak{d}_{A/\Fq[u]}$, and it has multiplication $\mathfrak b$.

For $0\ne\mathfrak b\in\mathfrak{d}_{A/\Fq[u]}$, choose an algebraic tensor
lift of this class and write it, with the order adapted to motive evaluation,
as
\begin{equation}\label{eq:GM-ordered-lift}
 \widehat{\mathfrak b}(\one,\two)
 =\sum_i b_i(\one)a_i(\two).
\end{equation}
Its class $[\widehat{\mathfrak b}]=[\widehat{\mathfrak b}^{e}]$
in $A\otimes_{\Fq[u]}A$ lies in $D_u$.  With this ordering, $b_i$
acts on the input and $a_i$ remains as a coefficient in the Tate algebra.

The cotangent--different identity is the multiplication isomorphism
\begin{equation}\label{eq:relative-different-multiplication}
 \mathfrak{d}_{A/\Fq[u]}\otimes_A\Omega_A
 \xrightarrow{\sim}A\,\rdu,\qquad b\otimes\eta\longmapsto b\eta.
\end{equation}

The possible non-diagonal poles of the resulting rational expression cancel
in codimension one.

\begin{lemma}
\label{lem:relative-Casimir-codimension-one-cancellation}
Let $u\in A$ be separating, put
$z_u:=u(\one)-u(\two)$, and let
$\widehat{\mathfrak b}\in A_{\one}\otimes_{\Fq}A_{\two}$ be a tensor whose image
in $A\otimes_{\Fq[u]}A$ belongs to $D_u$.  Then
$\widehat{\mathfrak b}/z_u$ is a section of $\OO(\Delta)$ on
$U_{\one}\times U_{\two}$: regular off the diagonal, with at most a
simple pole there.  The same holds after scalar extension.
\end{lemma}

\begin{proof}
Let $Z\ne\Delta$ be an irreducible component of $V(z_u)$.  It dominates the
generic point of the $u$-line: a component contained over one closed fibre
would be zero-dimensional.  Since $K/\Fq(u)$ is separable, the map defined by
$u$ is \revblue{\'{e}tale} over a dense open subset.  Hence $V(z_u)$ is reduced at the
generic point of $Z$, and
\begin{equation}\label{eq:relative-diagonal-reduced-order}
 \operatorname{ord}_Z(z_u)=1.
\end{equation}

After passage to the generic fibre, the reduced algebra
$K\otimes_{\Fq(u)}K$ is a product of fields.  The multiplication factor is
the diagonal component.  By the defining annihilator property of $D_u$, its
generic fibre is the separability-idempotent submodule supported on that diagonal
factor.  Therefore the image of $\widehat{\mathfrak b}$ in the factor
corresponding to $Z$ is zero.  The numerator thus lies in the height-one
prime of $Z$ and has order at least one there.  Together with
\eqref{eq:relative-diagonal-reduced-order}, this proves regularity of the
quotient at $Z$.

At the generic point of $\Delta$, the same \revblue{\'{e}tale} calculation shows that
$z_u$ is a local equation, so the quotient has at most a simple diagonal
pole.  These are all its possible codimension-one poles.  The surface
$U_{\one}\times U_{\two}$ is smooth, and $\OO(\Delta)$ is locally free; the
codimension-one section therefore extends uniquely across the remaining
closed points.
\end{proof}

For every $x\in K$,
\begin{equation}\label{eq:res-trace}
 \operatorname{Tr}_{\mathbb F_\infty/\Fq}
 \operatorname{res}_\infty(x\,\rd u)
 =
 \operatorname{res}_{\infty_u}
 \bigl(\operatorname{Tr}_{K/F}(x)\,\rd u\bigr).
\end{equation}
Here $\operatorname{res}_\infty$ is the residue at the closed point
$\infty$ of $X$ and $\operatorname{res}_{\infty_u}$ that at the infinite
place of $\Fq(u)$; this is the trace formula for residues at the unique
point above $\infty_u$, since the pole divisor of $u\in A$ is supported at
$\infty$.

\begin{lemma}
\label{lem:trace-residue-Casimir-coordinates}
Let $u\in A$ be nonconstant and separating, put $B=\Fq[u]$, and choose a
$B$-basis $e_1,\ldots,e_h$ of $A$ with trace-dual basis
$e_1',\ldots,e_h'$ in $K$ as in \eqref{eq:trace-dual}. Then
\[
 \Omega_A =\mathfrak{d}_{A/\Fq[u]}^{-1}A\,\rd u,
 \qquad (e_r'\,\rd u)_{1\leqslant r\leqslant h}
 \text{ is a $B$-basis of }\Omega_A ,
\]
and multiplication of the trace Casimir gives
\begin{equation}\label{eq:trace-Casimir-multiplication-one}
 \sum_{r=1}^h e_re_r'=1.
\end{equation}
\end{lemma}

\begin{proof}
The cotangent--different identity
\eqref{eq:relative-different-multiplication} gives the assertion about
$\Omega_A $ and its $B$-basis.  The tensor
$\sum_r e_r\otimes e_r'$ is the coevaluation tensor for the nondegenerate
trace pairing.  It is the separability idempotent, so its image under
multiplication is one, proving
\eqref{eq:trace-Casimir-multiplication-one}.
\end{proof}

\subsection{\texorpdfstring{Affine Cauchy kernels}{Affine Cauchy kernels}}\label{subsec:affine-cauchy-kernels}


For motive evaluation it is useful to record separately the orientation in
which the differential occurs in the second variable.

\begin{definition}
\label{def:normalized-affine-Cauchy-kernel}
For $0\ne\mathfrak b\in A$, the normalized affine $A$-Cauchy
fiber of residue $\mathfrak b$ is
\begin{equation}\label{eq:normalized-affine-Cauchy-fiber}
 \mathscr C_{A,\mathfrak b}
 :=\left\{
 \mathscr K\in
 H^0\!\left(U_{\one}\times U_{\two},
  \operatorname{pr}_{\two}^*\Omega_X^1\otimes\OO(\Delta)\right):
 \operatorname{res}_{\two=\one}\mathscr K=\mathfrak b
 \right\}.
\end{equation}
Its elements are called \emph{normalized affine Cauchy kernels
of $A$ of residue $\mathfrak b$}.  The same notation will be used after
scalar extension to $\Cinf$.  For $\mathfrak b=1$ this is the unit fiber
$\mathscr C_{A,1}=\mathscr C_A$.
\end{definition}

\begin{proposition}\label{prop:normalized-affine-Cauchy-torsor}
For $K\in\mathscr C_{A,\mathfrak b}$,
\begin{equation}\label{eq:normalized-affine-Cauchy-torsor}
 \mathscr C_{A,\mathfrak b}
 =K+A_{\one}\otimes_{\Fq}(\Omega_A )_{\two}.
\end{equation}
\end{proposition}
\begin{proof}
 Two members of $\mathscr C_{A,\mathfrak b}$ have the same diagonal
residue $\mathfrak b$.  Their difference is therefore regular on
$U_{\one}\times U_{\two}$ and defines a section of
$\operatorname{pr}_{\two}^*\Omega_X^1$.  Conversely, adding such a
differential preserves both the residue and the pole locus.   Hence $\mathscr C_{A,\mathfrak b}$ is a torsor under
$H^0\!\left(U_{\one}\times U_{\two},
 \operatorname{pr}_{\two}^*\Omega_X^1\right)
=A_{\one}\otimes_{\Fq}(\Omega_A)_{\two}$, the equality by
affineness.
\end{proof}

 The relative different gives the following three affine
Cauchy kernels. 

\begin{definition}\label{def:diagonal-lift-kernel}
For $0\ne\mathfrak b\in\mathfrak{d}_{A/\Fq[u]}$, choose an algebraic
lift $\widehat{\mathfrak b}$ as in \eqref{eq:GM-ordered-lift} and put
\begin{equation}\label{eq:diagonal-lift-kernel}
 K_{\mathfrak b}(\one,\two)
 :=\frac{\widehat{\mathfrak b}(\one,\two)}
        {u(\two)-u(\one)}\,\rdu(\two).
\end{equation}
This is a second-variable differential, regular on
$U_{\one}\times U_{\two}$ off the diagonal (Lemma~\ref{lem:relative-Casimir-codimension-one-cancellation}), with at most a simple pole
there, and
$\operatorname{res}_{\two=\one}K_{\mathfrak b}=\mathfrak b$,
since $\widehat{\mathfrak b}$ restricts to $\mathfrak b$ on
the diagonal.  {One trace-dual choice, with
$\mathfrak b$ in the second factor, is}
\begin{equation}\label{eq:trace-dual-lift-b}
 \widehat{\mathfrak b}(\one,\two)
 =\sum_{r=1}^h e_r(\one)(\mathfrak b e_r')(\two),
\end{equation}
for which $\widehat{\mathfrak b}(\one,\two)
=\mathfrak b(\two)\sum_r e_r(\one)e_r'(\two)$.
\end{definition}

\begin{definition}\label{def:trace-casimir-kernel}
With the trace-dual lift \eqref{eq:trace-dual-lift-b}, one has
$\widehat{\mathfrak b}(\one,\two)
=\mathfrak b(\two)\sum_r e_r(\one)e_r'(\two)$, where the second factor
is the trace Casimir{, whose class lies in
$A\otimes_{\Fq[u]}\mathfrak d_{A/\Fq[u]}^{-1}$ and is independent
of the chosen trace-dual basis.  Multiplication by
$\mathfrak b\in\mathfrak d_{A/\Fq[u]}$ gives the integral diagonal
class in $D_u$ of \eqref{eq:different-diagonal-trace-class}, with
multiplication $\mathfrak b$.}  The kernel $K_{\mathfrak b}$ specializes to $\mathfrak b(\two)K_u$, where
\begin{equation}\label{eq:signed-flipped-trace-Cauchy-representative}
 K_u(\one,\two)
 :=\frac{\sum_{r=1}^h e_r(\one)e_r'(\two)}
        {u(\two)-u(\one)}\,\rdu(\two).
\end{equation}
The dual basis elements sit in the second variable, where
$e_r'\rdu$ is regular ($e_r'$ lies in the inverse different).  The
numerator {multiplied by $\rdu(\two)$ is therefore
differential-valued and regular.  After a local trivialization of
$\Omega_A$ in the second variable,
Lemma~\ref{lem:relative-Casimir-codimension-one-cancellation}
applies to the resulting different-weighted Casimir and cancels
the off-diagonal poles.}  {Equation}~\eqref{eq:trace-dual-props} gives the diagonal
residue $\sum_re_re_r'=1$.  Hence
$K_u\in\mathscr C_{A,1}=\mathscr C_A\subset\mathscr C_\varphi^\Omega$.
{Changing the trace-dual basis changes $K_u$ by an element of
$A_{\one}\otimes_{\Fq}(\Omega_A)_{\two}$, by
Proposition~\ref{prop:normalized-affine-Cauchy-torsor}; thus the
diagonal residue and every period evaluation remain unchanged.}
\end{definition}


There exists a decomposition
\begin{equation}\label{eq:du-decomp}
 \rd u=\sum_{i=1}^m\mathfrak b_i\eta_i,
\end{equation}
for some $\mathfrak b_i \in \mathfrak{d}_{A/ \Fq[u]}$.

\begin{definition}\label{def:GM-ordered-Cauchy-lift}
Choose algebraic
lifts $\widehat{\mathfrak b}_i$ as in \eqref{eq:GM-ordered-lift}.
The differential-valued kernel attached to these choices is
\begin{equation}\label{eq:GM-ordered-Cauchy-lift}
 {K_{\{\eta_i,\widehat{\mathfrak b}_i\}}}
 :=\sum_{i=1}^m
   \frac{\widehat{\mathfrak b}_i(\one,\two)}
        {u(\two)-u(\one)}\,\eta_i(\two).
\end{equation}
It is a section of
$\operatorname{pr}_{\two}^*\Omega_X^1\otimes\OO(\Delta)$ on
$U_{\one}\times U_{\two}$ (the off-diagonal pole cancellation
being Lemma~\ref{lem:relative-Casimir-codimension-one-cancellation}) and is
called a \emph{lifted Cauchy kernel}.
\end{definition}

\begin{proposition}\label{prop:Kb-Kub-in-residue-b-fiber}
For every $0\ne\mathfrak b\in\mathfrak{d}_{A/\Fq[u]}$,
\begin{equation}\label{eq:Kb-Kub-in-fiber}
 K_{\mathfrak b}\in\mathscr C_{A,\mathfrak b}
 \qquad\text{and}\qquad
 K_u\,\mathfrak b(\two)\in\mathscr C_{A,\mathfrak b}.
\end{equation}
Moreover $K_{\{\eta_i,\widehat{\mathfrak b}_i\}}\in\mathscr C_{A,1}=\mathscr C_A$.
In particular, $\mathfrak b(\two)K_{\{\eta_i,\widehat{\mathfrak b}_i\}}\in\mathscr C_{A,\mathfrak b}$
as well, so both differences
\[
 \mathfrak b(\two)\,K_{\{\eta_i,\widehat{\mathfrak b}_i\}}-K_{\mathfrak b}
 \qquad\text{and}\qquad
 K_{\mathfrak b}-\mathfrak b(\two)\,K_u
\]
lie in $A_{\one}\otimes_{\Fq}(\Omega_A)_{\two}$, i.e., are regular
second-variable differentials.
\end{proposition}
\begin{proof}
For any algebraic lift, $\widehat{\mathfrak b}$ is regular on
$U_{\one}\times U_{\two}$, and
Lemma~\ref{lem:relative-Casimir-codimension-one-cancellation}
applied to the class $[\widehat{\mathfrak b}]\in D_u$
cancels the off-diagonal poles of
$(u(\two)-u(\one))^{-1}$; the diagonal residue is $\mathfrak b$.  Hence
$K_{\mathfrak b}\in\mathscr C_{A,\mathfrak b}$.  Since $K_u\in\mathscr C_{A}$
and $\mathfrak b(\two)$ is regular on $U_{\two}$, also
$K_u\,\mathfrak b(\two)\in\mathscr C_{A,\mathfrak b}$.


For the assembled kernel,
Lemma~\ref{lem:relative-Casimir-codimension-one-cancellation}
applied to the classes $[\widehat{\mathfrak b}_i]\in D_u$
removes the non-diagonal poles, and the diagonal residues sum to
$\sum_i\mathfrak b_i\eta_i/\rdu=1$. Hence
$K_{\{\eta_i,\widehat{\mathfrak b}_i\}}\in\mathscr C_A$ and  $ \mathfrak{b}(\two)K_{\{\eta_i,\widehat{\mathfrak b}_i\}}\in\mathscr C_{ A,\mathfrak{b} } $.  The two displayed
differences then restate the torsor structure of $\mathscr C_{A,\mathfrak b}$
(Proposition~\ref{prop:normalized-affine-Cauchy-torsor}).
\end{proof}

\subsection{\texorpdfstring{The Szeg\H o kernel and the $V_\varphi$-bounded Cauchy kernel}{The Szego kernel and the V-phi-bounded Cauchy kernel}}\label{sec:szego-kernel}

 Throughout, the divisor convention for differentials is
$
 \omega\in H^0\bigl(X_{\Cinf},\Omega_X^1(D)\bigr)
 $; equivalently, $
 \operatorname{div}(\omega)+D\geqslant0.
$
 
\begin{lemma}\label{lem:acyclic}
For every geometric point $q$ above $\infty$ with
$q\notin\operatorname{supp}(V_\varphi)$, the line bundle
\[
  \OO\bigl(V_\varphi-q \bigr)
\]
is acyclic.
\end{lemma}

\begin{proof}
Sections of $\OO(V_\varphi)$ are constant because $h^0(V_\varphi)=1$; a section of
$\OO(V_\varphi-q)$ is therefore a constant vanishing at
$q$, hence zero.  Riemann--Roch gives
$h^1=h^0-\deg+g-1=0-(g-1)+g-1=0$ as well.
\end{proof}

We recall from {\cite[Proposition~4.6]{Hu26}} the following facts, proved there in every genus.
For an acyclic line bundle $\LL$ on $X_{\Cinf}$, the \emph{Szeg\H o bundle}
$\cB_{\LL}:=\operatorname{pr}_{\one}^{*}\LL\otimes\operatorname{pr}_{\two}^{*}
\bigl(\LL^{-1}\otimes\Om\bigr)$ carries a unique section
$S_{\LL}\in H^0\bigl(X\times X,\cB_{\LL}(\Delta)\bigr)$, the \emph{normalized Szeg\H o
kernel}, whose residue along the diagonal $\Delta$ equals $1$:
\[
\res_{\two=\one}S_{\LL}=1.
\]
\begin{definition}
For an acyclic divisor line $\LL=\OO(D)$, let $e_{\LL}$ be its tautological
rational frame.  Define the second-variable meromorphic differential $K_{\LL}$ by
\begin{equation}\label{eq:tautological-Szego-coordinate}
 S_{\LL}=K_{\LL}(\one,\two)e_{\LL}(\one)\otimes e_{\LL}(\two)^{-1}.
\end{equation}
It has diagonal residue one and satisfies
$\operatorname{div}(K_{\LL})+\operatorname{pr}_{\one}^*D
-\operatorname{pr}_{\two}^*D+\Delta\geqslant0$.
\end{definition}
All subsequent kernel evaluations use this differential coordinate.

The affine regular kernel $\mathscr C_A$ of \S\ref{subsec:affine-cauchy-kernels} is too
small for the natural coordinate of the shifted kernel: its tautological
divisor frame permits poles along $V_\varphi$ in the first variable.  Put
\begin{equation}\label{eq:Vphi-affine-coefficient-module}
 \cM^{\mathrm{geo}}_\varphi
 :=\Gamma\bigl(U_{\Cinf},\OO_{U_{\Cinf}}(V_\varphi)\bigr).
\end{equation}
{This is the geometric model of the integral motive; its
sections may have arbitrary poles along the removed infinity fiber.}

\begin{definition}
\label{def:Vphi-bounded-Cauchy-fiber}
Let
\begin{align}
 \mathscr A_\varphi^\Omega
 &:=
 H^0\!\left(U_{\Cinf,\one}\times U_{\Cinf,\two},
  \operatorname{pr}_{\one}^*\OO(V_\varphi)
  \otimes\operatorname{pr}_{\two}^*\Omega_X^1
  \otimes\OO(\Delta)\right),                              \label{eq:Vphi-Cauchy-ambient}\\
 \mathscr C_\varphi^\Omega
 &:=\left\{\mathscr K\in\mathscr A_\varphi^\Omega:     \operatorname{res}_{\two=\one}\mathscr K=1\right\}.   \label{eq:Vphi-Cauchy-fiber}
\end{align}
Thus a member is a differential in the second variable, has a normalized
simple diagonal pole, is regular in that variable away from the diagonal,
and has first-variable poles bounded by $V_\varphi$.  The inclusion
$\OO\hookrightarrow\OO(V_\varphi)$ identifies
$\mathscr C_A$ with a subkernel of
$\mathscr C_\varphi^\Omega$.
\end{definition}

\begin{lemma}\label{lem:SzegoC}
The tautological differential coordinate $K_\varphi$ belongs to
$\mathscr C_\varphi^\Omega$.
\end{lemma}
\begin{proof}
    Write 
$ \LL_\varphi=\OO(V_\varphi-q_\varphi^{(-1)})$.  By definition,
\[
 S_{\LL_\varphi}
 =K_\varphi (\one,\two)e(\one)e(\two)^{-1}.
\]
The fact that this section is regular away from its simple diagonal pole
gives, on $U_{\Cinf}\times U_{\Cinf}$,
\[
 \operatorname{div}(K_\varphi )
 +\operatorname{pr}_{\one}^*V_\varphi
 -\operatorname{pr}_{\two}^*V_\varphi+\Delta\geqslant0.
\]
Thus $K_\varphi $ has first-variable poles bounded by $V_\varphi$ and no
non-diagonal pole in the second variable.  Its normalized residue is one,
so it belongs to $\mathscr C_\varphi^\Omega$.
\end{proof}
\begin{theorem}
\label{thm:Vphi-Cauchy-fiber-invariance}
For every
$K_0\in\mathscr C_\varphi^\Omega$ one has
\begin{equation}\label{eq:Vphi-Cauchy-torsor}
 \mathscr C_\varphi^\Omega
 = K_0 +
   (\cM^{\mathrm{geo}}_\varphi)_{\one}\otimes_{\Cinf}
   \Gamma(U_{\Cinf},\Omega_X^1)_{\two}.
\end{equation}
Moreover, for $\lambda\in\Lambda_\varphi$, by
Lemma~\ref{lem:orbital-evaluation-convergence}, the orbital evaluations
below are well defined and agree:
\begin{equation}\label{eq:Vphi-Cauchy-period-invariance}
 \mathcal E^\Xi_{\varphi,\lambda}
   (K_1s_\varphi)
 =\mathcal E^\Xi_{\varphi,\lambda}
   (K_2s_\varphi)
 \qquad(K_1,K_2\in\mathscr C_\varphi^\Omega);
\end{equation}
Equivalently, for $z\in\Cinf$ the value
$\mathcal E^\Xi_{\varphi,z}(Ks_\varphi)$ is independent of
$K\in\mathscr C_\varphi^\Omega$ \revblue{precisely when} $z\in\Lambda_\varphi$.
\end{theorem}

\begin{proof}
Lemma~\ref{lem:SzegoC} proves that $ \mathscr C_{\varphi}^\Omega$ is nonempty.

If $K_1,K_2$ lie in the kernel, their difference has zero
diagonal residue.  The residue exact sequence removes its possible simple
diagonal pole.  Since $U_{\Cinf}$ is affine, the Kunneth isomorphism then
gives
\[
 K_1-K_2\in
 (\cM^{\mathrm{geo}}_\varphi)_{\one}\otimes_{\Cinf}
 \Gamma(U_{\Cinf},\Omega_X^1)_{\two}.
\]
Conversely, adding any element of this tensor product preserves both the
pole bound and the normalized residue.  This proves
\eqref{eq:Vphi-Cauchy-torsor}.

For the period case, use \eqref{eq:Vphi-Cauchy-torsor} to write
\[
 K_1-K_2
 =\sum_j c_j(\one)\eta_j(\two),
 \qquad c_j\in\cM^{\mathrm{geo}}_\varphi,
 \quad \eta_j\in\Gamma(U_{\Cinf},\Omega_X^1).
\]
The membership criterion \eqref{eq:membership} gives
$c_js_\varphi\in\cM_\varphi$.  Lemma~\ref{lem:period-annihilator}
annihilates each coefficient at $\lambda$.  The automatic orbit estimates
of Proposition~\ref{thm:automatic-orbit-growth} and
Theorem~\ref{thm:pellarin} then
identify the resulting equality with the orbital evaluations on every
common affinoid. Meromorphic continuation proves \eqref{eq:Vphi-Cauchy-period-invariance}.
Conversely, choose $0\ne\eta\in\Omega_A$ and
$K_0\in\mathscr C_\varphi^\Omega$. Since $1\in\cM^{\mathrm{geo}}_\varphi$,
\eqref{eq:Vphi-Cauchy-torsor} places $K_0+1\otimes\eta$ in the same fiber.
Their evaluations differ by
\[
 \mathcal E^\Xi_{\varphi,z}\bigl((1\otimes\eta)s_\varphi\bigr)
 =\exp_\varphi(z)\eta,
\]
by Theorem~\ref{thm:pellarin} and the normalization in Proposition~\ref{prop:certificates}. If the whole fiber has a common value,
this differential is zero. Since $\eta\ne0$, one has
$\exp_\varphi(z)=0$, hence $z\in\Lambda_\varphi$.
\end{proof}

\subsection{\texorpdfstring{The Gazda--Maurischat transform}{The Gazda--Maurischat transform}}
Applying Cauchy-fiber invariance to the lifted kernel gives
the Gazda--Maurischat realization below. 

\begin{definition}\label{def:GM-transform}
For $z\in\Lie(\varphi)(\Cinf)=\Cinf$, define
\begin{equation}\label{eq:GM-transform}
 \GM_{u;\widehat{\mathfrak b}}(z)
 :=(\id\widehat\otimes\Exp)
 \left(
  \left(1\otimes\iota(u)-u\otimes1\right)^{-1}
  \sum_i a_i\otimes\iota(b_i)z
 \right).
\end{equation}
The inverse is taken in the standard Tate completion.  This is the rank-one
Drinfeld-module specialization of the resolvent construction in
\cite[Theorem~3.11]{GM21}.
\end{definition}

\begin{definition}\label{def:canonical-assembly}
For every period $\lambda\in\Lambda_\varphi$, choose any presentation
\eqref{eq:du-decomp} and any algebraic lift of each $\mathfrak b_i$, and put
\begin{equation}\label{eq:canonical-anderson}
\GM_u^\Omega(\lambda)
 :=-\sum_{i=1}^m\GM_{u;\widehat{\mathfrak b}_i}(\lambda)\otimes\eta_i(\two)
 \in\sfm(\varphi)\otimes_A\Omega_A(\two).
\end{equation}
\end{definition}
\begin{lemma}\label{lem:GM-transform-evaluations}
The transform evaluates the diagonal-lift kernel of
Definition~\ref{def:diagonal-lift-kernel}: for $z\in\Cinf$,
\begin{equation}\label{eq:GM-cauchy}
 \GM_{u;\widehat{\mathfrak b}}(z)
 =\widehat{\mathcal E}^{\mathrm{sep}}_{\varphi,z}
  \left(\frac{\widehat{\mathfrak b}(\one,\two)}
             {u_\one-u_\two}s_\varphi\right)
 =-\frac{1}{\rdu(\two)}\,
   \widehat{\mathcal E}^{\mathrm{sep}}_{\varphi,z}
   \bigl(K_{\mathfrak b}(\one,\two)s_\varphi\bigr),
\end{equation}
where on every closed $u_\two$-disk with $|u_\two|\leqslant\rho<|\iota(u)|$ the
fraction belongs to the separated evaluation completion of
Definition~\ref{def:evaluation-completion}.  For
$\lambda\in\Lambda_\varphi$,
\begin{equation}\label{eq:GM-period-evaluation}
 -\GM_{u;\widehat{\mathfrak b}}(\lambda)\,\rdu(\two)
 =\mathcal E_{\varphi,\lambda}^\Xi
  \bigl(K_{\mathfrak b}(\one,\two)s_\varphi\bigr),
\end{equation}
and the assembled transform is the evaluation
\begin{equation}\label{eq:GM-assembly-single-evaluation}
 \GM_u^\Omega(\lambda)
 =\mathcal{E}_{\varphi,\lambda}^\Xi
  \bigl(K_{\{\eta_i,\widehat{\mathfrak b}_i\}}(\one,\two)s_\varphi\bigr).
\end{equation}
\end{lemma}
\begin{proof}
The normally convergent Cauchy expansion on the disk gives the explicit
formula
\begin{equation}\label{eq:GM-cauchy-expansion}
 \GM_{u;\widehat{\mathfrak b}}(z)
 =\sum_i\sum_{n\geqslant0}
   a_i(\two)u_\two^n
   \Exp\!\left(
    \iota\!\left(\frac{b_i}{u^{n+1}}\right)z
   \right),
\end{equation}
evaluated termwise; this is the first equality of
\eqref{eq:GM-cauchy}. The second equality follows from 
\[
\frac{\widehat{\mathfrak b}(\one,\two)}{u_\one-u_\two}
=-K_{\mathfrak b}(\one,\two)/\rdu(\two).
\]
Restricting to the period $\lambda$ gives
\eqref{eq:GM-period-evaluation}, and summing it over the presentation
$\rdu=\sum_i\mathfrak b_i\eta_i$ with the lifts
$\widehat{\mathfrak b}_i$ gives
\eqref{eq:GM-assembly-single-evaluation}.
\revblue{Consequently, $\GM_{u;\widehat{\mathfrak b}}(\lambda)$ depends on
$\mathfrak b$ rather than on the chosen lift.  Two lifts differ by
$(u_\one-u_\two)H(\one,\two)$ with $H\in A\otimes_{\Fq}A$.
Lemma~\ref{lem:period-annihilator} then shows that every first-factor
coefficient of $Hs_\varphi$ has zero evaluation at $\lambda$.}
\end{proof}

Theorem~\ref{thm:canonical-assembly} below identifies the
assembled transform with the orbital generating differential
$\mathcal H_\varphi$.

\begin{lemma}\label{lem:GM-assembly-special}
For every $\lambda\in\Lambda_\varphi$,
\begin{equation}\label{eq:GM-assembly-special}
 \GM_u^\Omega(\lambda)\in\sfm(\varphi)\otimes_A\Omega_A(\two).
\end{equation}
\end{lemma}
\begin{proof}
Put $K:=K_{\{\eta_i,\widehat{\mathfrak b}_i\}}$ and
$\omega:=\GM_u^\Omega(\lambda)=\mathcal E_{\varphi,\lambda}^\Xi(Ks_\varphi)$
(Lemma~\ref{lem:GM-transform-evaluations}).  For $a\in A$ the diagonal
defect
\[
 \mathscr D_a(K):=a(\one)K-a(\two)K
\]
is integral.  Indeed the class $[\widehat{\mathfrak b}_i]=[\widehat{\mathfrak b}^{e}_i]$
lies in $D_u$, and $D_u$ is annihilated by $a\otimes1-1\otimes a$, so
\[
 (a\otimes1-1\otimes a)\widehat{\mathfrak b}_i
 =(u\otimes1-1\otimes u)H_i,
 \qquad H_i\in A\otimes_{\Fq}A,
\]
and therefore
\[
 \mathscr D_a(K)
 =\sum_{i=1}^m
  \frac{(a\otimes1-1\otimes a)\widehat{\mathfrak b}_i}
       {u(\two)-u(\one)}\,\eta_i(\two)
 =-\sum_{i=1}^m H_i(\one,\two)\,\eta_i(\two),
\]
whose first-factor coefficients lie in $A\subset\cM_\varphi$.
Its period evaluation vanishes (Lemma~\ref{lem:period-annihilator}).
{For each first-factor coefficient $c\in K$ of the affine
representative, both
$\varphi_a\bigl(\mathcal E^{\mathrm{rat}}_{\varphi,\lambda}(c\,s_\varphi)\bigr)$
and $\mathcal E^{\mathrm{rat}}_{\varphi,\lambda}
\bigl(\varphi_a^!(c\,s_\varphi)\bigr)$
equal $\exp_\varphi(\iota(ac)\lambda)$ by
\eqref{eq:functional} and $\varphi_a\varphi_b=\varphi_b\varphi_a$.
Thus the first-variable action passes through evaluation as
$\varphi_a$, while $a(\two)$ acts as a second-variable scalar:}
\[
 0
 =\widehat{\mathcal{E}}^{\mathrm{sep}}_{\varphi,\lambda}\bigl(\mathscr D_a(K)s_\varphi\bigr)
 =\varphi_a(\omega)-a(\two)\,\omega.
\]
Thus $\varphi_a(\omega)=a(\two)\,\omega$ for every $a\in A$,
which is precisely the defining eigenrelation of
$\sfm(\varphi)\otimes_A\Omega_A$
(Definition~\ref{def:special-functions}).
\end{proof}

\subsection{The generating differential and its characteristic residue}

\begin{definition}\label{def:kernel}
For $z\in\Cinf$, define the generating differential
    \[
 \mathcal{H}_{\varphi}(z)
 :=\mathcal E^\Xi_{\varphi,z}
   (K_\varphi s_\varphi).
\]
The convergence of $ \mathcal{H}_{\varphi}(z)$ is guaranteed by Lemma~\ref{lem:orbital-evaluation-convergence}.
\end{definition}
\begin{theorem}
\label{thm:canonical-assembly}
For every $\lambda\in\Lambda_\varphi$,
\begin{equation}
\label{eq:canonical-assembly-kernel}
 \GM_u^\Omega(\lambda)
 =\mathcal H_\varphi(\lambda) \in \sfm(\varphi) \otimes_{A} \Omega_A(\two).
\end{equation}
Moreover, with the trace-dual lift \eqref{eq:trace-dual-lift-b},
for every $0\ne\mathfrak b\in\mathfrak{d}_{A/\Fq[u]}$,
\begin{equation}\label{eq:GM-b-different}
 \GM_{u;\widehat{\mathfrak b}}(\lambda)\,\rdu(\two)
 =-\mathfrak b(\two)\,\mathcal H_\varphi(\lambda).
\end{equation}
The common value is independent of the separating element $u$ and of the
choices of the $\eta_i$ and $\mathfrak b_i$.
\end{theorem}

\begin{proof}
Lemma~\ref{lem:GM-transform-evaluations} gives
$\GM_u^\Omega(\lambda)
=\mathcal E^\Xi_{\varphi,\lambda}
  \bigl(K_{\{\eta_i,\widehat{\mathfrak b}_i\}}s_\varphi\bigr)$.
Both $K_{\{\eta_i,\widehat{\mathfrak b}_i\}}$ and $K_\varphi$ lie in
$\mathscr C_\varphi^\Omega$
(Proposition~\ref{prop:Kb-Kub-in-residue-b-fiber} and Lemma~\ref{lem:SzegoC}); by Theorem~\ref{thm:pellarin} and the period
invariance of Theorem~\ref{thm:Vphi-Cauchy-fiber-invariance}, this common
evaluation equals
$\mathcal E^\Xi_{\varphi,\lambda}(K_\varphi s_\varphi)
=\mathcal H_\varphi(\lambda)$.
For \eqref{eq:GM-b-different},
Lemma~\ref{lem:GM-transform-evaluations} and the trace-dual identity
$K_{\mathfrak b}=\mathfrak b(\two)K_u$ give
\[
 -\GM_{u;\widehat{\mathfrak b}}(\lambda)\,\rdu(\two)
 =\mathcal E^\Xi_{\varphi,\lambda}(K_{\mathfrak b}s_\varphi)
 =\mathfrak b(\two)\mathcal E^\Xi_{\varphi,\lambda}(K_u s_\varphi)
 =\mathfrak b(\two)\mathcal H_\varphi(\lambda),
\]
\revblue{the final identification with $\mathfrak b(\two)\mathcal H_\varphi(\lambda)$ uses the same period invariance.}
\revblue{This argument} also shows directly that the value is independent of all choices.

\end{proof}

\begin{lemma}
\label{lem:orbital-vertical-regularity}
For $n\geqslant1$ and every affinoid $W\subset U_{\Cinf}^{\mathrm{an}}$
with $\Xi^{(n)}\notin W$, the specialization
$\operatorname{orb}_{\varphi,n}(K_\varphi)$ is analytic on $W$.
\end{lemma}

\begin{proof}
Apply Proposition~\ref{thm:automatic-orbit-growth}(ii) to $R=K_\varphi$,
with $E:=\{n\geqslant {0}:\Xi^{(n)}\in\operatorname{supp}(V_\varphi)\cup W\}$; this
set is finite by Lemma~\ref{lem:orbit-escape} and the finiteness of
$\operatorname{supp}(V_\varphi)$.   The divisor inequality of
Lemma~\ref{lem:SzegoC},
$\operatorname{div}(K_\varphi)+\operatorname{pr}_{\one}^*V_\varphi
-\operatorname{pr}_{\two}^*V_\varphi+\Delta\geqslant0$, read along the vertical
prime divisor $\{\Xi^{(n)}\}\times X_{\Cinf}$ (to which neither
$-\operatorname{pr}_{\two}^*V_\varphi$ nor $\Delta$ contributes), gives
\[
 \operatorname{ord}_{\{\Xi^{(n)}\}\times X_{\Cinf}}(K_\varphi)\geqslant-m_n;
\]
in particular
$K_\varphi$ is regular on $\{\Xi^{(n)}\}\times W$ for $n\notin E$.
Proposition~\ref{thm:automatic-orbit-growth}(ii) applies.
\end{proof}

\begin{proposition}\label{lem:residue}
For every $z\in\Cinf$, $\mathcal H_\varphi(z)$ continues
meromorphically to a neighborhood of $\Xi$, with at most a simple pole,
and
\[
 \operatorname{res}_\Xi\mathcal H_\varphi(z)=z.
\]
In particular, $\mathcal H_\varphi(\lambda)\ne0$ for every nonzero period.
\end{proposition}
\begin{proof}
By Lemma~\ref{lem:orbit-escape}, the characteristic orbit is
injective and exits every affinoid; choose a closed affinoid neighborhood
$W_\Xi$ of $\Xi$ with
$W_\Xi\cap\{\Xi^{(n)}\}_{n\geqslant0}=\{\Xi\}$.

By Lemma~\ref{lem:orbital-vertical-regularity},
$\operatorname{orb}_{\varphi,n}(K_\varphi)$ is analytic on $W_\Xi$ for
every $n\geqslant1$, including the finitely many $n$ with
$\Xi^{(n)}\in\operatorname{supp}(V_\varphi)$.

For the remaining $n$, the divisor bound of Lemma~\ref{lem:SzegoC} puts no
pole of $K_\varphi$ on $\{\Xi^{(n)}\}\times W_\Xi$, and
Proposition~\ref{thm:automatic-orbit-growth}(iii) gives
$\|\operatorname{orb}_{\varphi,n}(K_\varphi)\|_{W_\Xi}
\leqslant|e_n(\varphi)|\,C^{q^n}$ for some $C\geqslant1$.  Since
$|e_n(\varphi)|^{1/q^n}\to0$
\eqref{eq:automatic-exponential-coefficient-decay}, for $|z|\leqslant R$ the
tails satisfy
$\|z^{q^n}\operatorname{orb}_{\varphi,n}(K_\varphi)\|_{W_\Xi}\leqslant2^{-q^n}$
for all sufficiently large $n$; hence the series
$T(z,\cdot):=\sum_{n\geqslant1}z^{q^n}\operatorname{orb}_{\varphi,n}(K_\varphi)$
converges normally on $W_\Xi$, uniformly for bounded $z$, and is analytic
there.

Thus $\mathcal H_\varphi(z)=zK_\varphi(\Xi,\cdot)+T(z,\cdot)$ on $W_\Xi$,
with $T(z,\cdot)$ analytic, while $zK_\varphi(\Xi,\cdot)$ has at most a
simple pole at $\two=\Xi$, with residue $z$ because the relative diagonal
residue of $K_\varphi$ is one (Lemma~\ref{lem:SzegoC}).  Therefore
$\operatorname{res}_\Xi\mathcal H_\varphi(z)=z$ and the continuation has
at most a simple pole there; on the overlap with any off-orbit affinoid,
both expressions coincide with the same convergent orbital series.
Finally, for $0\ne\lambda\in\Lambda_\varphi$ this residue is
$\lambda\ne0$, so $\mathcal H_\varphi(\lambda)\ne0$.
\end{proof}
\begin{remark}\label{rem:GM-characteristic-residue}
For a period $\lambda$, Theorem~\ref{thm:canonical-assembly}
gives the same residue assertion for $\GM_u^\Omega(\lambda)$.
The same characteristic-residue assertion is developed in the
Gazda--Maurischat formalism \cite[Theorem~5.8]{GM24}.
\end{remark}

\subsection{\texorpdfstring{Rational Anderson differentials}{Rational Anderson differentials}}\label{subsec:rational-anderson-differentials}
\begin{definition}
A rational Anderson differential is a solution
$Q\in\operatorname{Frac}(\mathbb T_A)\otimes_A\Omega_A$ of
\begin{equation}\label{eq:localized-homogeneous-equation}
 Q^{(1)}=f_\varphi Q.
\end{equation}
\end{definition}

\begin{theorem}\label{thm:rational-anderson-continuation}
Let $Q$ be a rational Anderson differential.
\begin{enumerate}[label={\textup{(\roman*)}},leftmargin=2.2em]
\item\label{item:ratand-form}
$Q$ belongs to
\begin{equation}\label{eq:localized-homogeneous-membership}
 \sfm(\varphi)\otimes_A\Omega_A^{\rat}.
\end{equation}
For any fixed $0\ne\lambda\in\Lambda_\varphi$, $Q$ is uniquely of the form
$Q=r\mathcal H_\varphi(\lambda)$ with $r\in K$.
\item\label{item:ratand-continuation}
$Q$ has a canonical meromorphic continuation to a neighborhood of $\Xi$,
with at most a simple pole there.
\item\label{item:ratand-residue}
The characteristic residue is a $K$-linear isomorphism
\begin{equation}\label{eq:localized-characteristic-residue-isomorphism}
 \operatorname{res}_\Xi:
 \sfm(\varphi)\otimes_A\Omega_A^{\rat}
 \xrightarrow{\sim}\iota(K)\Lambda_\varphi.
\end{equation}
\end{enumerate}
\end{theorem}
\begin{proof}
\emph{\ref{item:ratand-form}.}
By Theorem~\ref{thm:canonical-assembly}, $\mathcal H_\varphi(\lambda)$
is an integral special differential. Its residue is $\lambda\ne0$
by Proposition~\ref{lem:residue}, so it is nonzero.
Lemma~\ref{lem:scalar-shtuka-equation} gives its homogeneous equation.
Their quotient belongs to $\operatorname{Frac}(\mathbb T_A)$
and is fixed by coefficient Frobenius. Lemma~\ref{lem:general-Tate-Frobenius-constants}
places it in $K$.  Conversely, multiplication by $r\in K$ preserves the
equation, since coefficient Frobenius fixes $K$.  Both claims follow.

\emph{\ref{item:ratand-continuation}.}
Write $Q=r\mathcal H_\varphi(\lambda)$ as in
part~\ref{item:ratand-form}.  The rational function $r\in K$ is
regular at $\Xi$, and the orbital construction continues
$\mathcal H_\varphi(\lambda)$ meromorphically on the affine analytic
curve, with a locally finite set of possible poles.
Proposition~\ref{lem:residue} bounds the pole order of
$\mathcal H_\varphi(\lambda)$ at $\Xi$ by one, so their product has at
most a simple pole there.  Their product supplies the continuation.

Two such representations agree on the affine Tate domain.  Both extend
meromorphically on the connected smooth analytic curve
$U_{\Cinf}^{\mathrm{an}}$, so the identity principle makes their
continuations equal, in particular near $\Xi$.

\emph{\ref{item:ratand-residue}.}
Fix $0\ne\lambda\in\Lambda_\varphi$.
By parts~\ref{item:ratand-form} and~\ref{item:ratand-continuation},
every element has a unique continuation of the form
$r\mathcal H_\varphi(\lambda)$, $r\in K$.
Its pole at $\Xi$ is at most simple, so its residue is
$r(\Xi)\lambda=\iota(r)\lambda$ by Proposition~\ref{lem:residue}.
The lattice has rank one, hence
$\iota(K)\Lambda_\varphi=\iota(K)\lambda$.
The displayed map is therefore an isomorphism, with inverse
$\iota(r)\lambda\mapsto r\mathcal H_\varphi(\lambda)$.
\end{proof}

\begin{proposition}\label{prop:isogeny-period-action}
 Let $\eta:\varphi\to\varphi'$ be a separable isogeny, acting
coefficientwise on the Tate factor of $\sfm(\varphi)$.  Then for every
$\lambda\in\Lambda_\varphi$,
\begin{equation}\label{eq:isogeny-period-action}
 \eta\bigl(\mathcal{H}_{\varphi}(\lambda)\bigr)
 =\mathcal{H}_{\varphi'}(\partial\eta\,\lambda).
\end{equation}
\end{proposition}
\begin{proof}
  Since
$\varphi'_a\circ\eta=\eta\circ\varphi_a$, the eigenrelation in
$\sfm(\varphi)$ gives
\[
\varphi'_a\bigl(\eta(\mathcal H_\varphi(\lambda))\bigr)= \eta \circ \varphi_a  (\mathcal H_\varphi(\lambda)) 
=(1\otimes a)\,\eta\bigl(\mathcal H_\varphi(\lambda)\bigr).\]
On the other hand, Proposition~\ref{prop:isogeny-lattice-transport} \revblue{gives} 
\[ 
\partial \eta \lambda \in \Lambda_{\varphi'} = \iota(\ii_J )^{-1} \partial \eta \Lambda_\varphi\revblue{,}
\]
where $\ii_J$ denotes the corresponding ideal of the \revblue{isogeny datum}. From Theorem~\ref{thm:canonical-assembly}, both sides of Equation~\eqref{eq:isogeny-period-action} lie in $\sfm(\varphi')\otimes_A\Omega_A$.

The right side has characteristic residue $\partial\eta\,\lambda$
(Proposition~\ref{lem:residue}).  For the left side, write
$\eta=\sum_i a_i\tau^i$. We know that
$\mathcal H_\varphi(\lambda)^{(i)}=F_\varphi^{[i]}\mathcal
H_\varphi(\lambda)$, and $F_\varphi^{[i]}$ vanishes at $\Xi$ for
$i\geqslant1$ while $\mathcal H_\varphi(\lambda)$ has at most a simple pole
there. So only the constant term of $\eta$ contributes to the residue:
\[ 
\operatorname{res}_\Xi\bigl(\eta(\mathcal H_\varphi(\lambda))\bigr)
=a_0\lambda=\partial\eta\,\lambda.\]
The residue isomorphism of
Theorem~\ref{thm:rational-anderson-continuation}(iii) gives the
equality. 
\end{proof}
As a special case, the isogeny $\eta=\varphi_a$ gives the following
corollary.
\begin{corollary}\label{cor:Hphi-linearity}
For every $\lambda\in\Lambda_\varphi$ and $a\in A$,
\begin{equation}\label{eq:Hphi-linearity}
 \mathcal H_\varphi(\iota(a)\lambda)
 =\varphi_a\bigl(\mathcal H_\varphi(\lambda)\bigr)
 =a(\two)\,\mathcal H_\varphi(\lambda).
\end{equation}
\end{corollary}
\section{Anchor map and special functions}
\label{sec:anchor}
{Residue duality and the exponential give an isomorphism
\[
 \operatorname{Hom}_A(\Omega_A,\Lambda_\varphi)
 \xrightarrow{\ \sim\ }\sfm(\varphi),\qquad
 \jmath\longmapsto\omega_\jmath.
\]
The trace expansion identifies $\omega_\jmath$ with the
Gazda--Maurischat expression, and the local logarithm constructs the
inverse (Theorem~\ref{thm:direct-logarithmic-special-line}).
For an isomorphism $\jmath$, Section~\ref{sec:unit-special} proves
that $\omega_\jmath$ is a Tate-algebra unit.}

\subsection{The residue--Pontryagin class}
\label{subsec:truncated-Cauchy-A}

By residue duality 
\[
 \langle-,-\rangle_\infty:A \times \overline{\Omega} \longrightarrow\Fq
\]
we conclude that $\overline{\Omega}$ is the compact Pontryagin dual of the
discrete $\Fq$-space $A$.  

 Explicitly, 
\begin{equation}\label{eq:residue-dual-coordinate-line}
 A^\vee_{\mathrm{res}}
 :=\operatorname{Hom}_{\Fq}(A,\Fq)
 \simeq \overline\Omega,\qquad
 \bar\xi(a)=
 \operatorname{Tr}_{\mathbb F_\infty/\Fq}
 \operatorname{res}_\infty(a\xi).
\end{equation}
We use the quotient topology on $\overline\Omega$ and the topology of pointwise
convergence on $A^\vee_{\mathrm{res}}$; residue duality identifies these
topologies \cite[Theorem~8, with $I=A$]{Poo96}.  Thus, for any ordered
$\Fq$-basis $(a_i)_{i\geqslant0}$ of $A$,
\[
 A^\vee_{\mathrm{res}}=\prod_{i\geqslant0}\Fq a_i^*,
 \qquad a_i^*(a_j)=\delta_{ij}.
\]
The coordinate vectors $a_i^*$ tend to zero, and every $\bar\xi$ is the
limit of $\sum_{i<N}\bar\xi(a_i)a_i^*$.  Equivalently, annihilators of finite
dimensional subspaces of $A$ form a neighborhood basis at zero.

\begin{lemma}\label{lem:residue-dual-trace-basis}
Let $u\in A$ be nonconstant and separating, and choose
trace-dual bases $(e_r),(e_r')$ of $A/\Fq[u]$ as in
Lemma~\ref{lem:trace-residue-Casimir-coordinates}.  Then the
$\Fq$-basis $(e_ru^n)_{r,n}$ of $A$ and the family
\begin{equation}\label{eq:direct-trace-residue-dual-family}
 -e_r'u^{-n-1}\,\rd u\pmod{\Omega_A }
 \qquad(1\leqslant r\leqslant h,\ n\geqslant0)
\end{equation}
in $\overline\Omega$ are dual with respect to the residue pairing:
$\langle e_su^m,-e_r'u^{-n-1}\rd u\rangle_\infty
=\delta_{rs}\delta_{mn}$.
\end{lemma}
\begin{proof}
By \eqref{eq:res-trace} with
$x=-e_se_r'u^{m-n-1}$ (the pole divisor of $u\in A$ being supported at
$\infty$, so $\infty$ is the unique point above $\infty_u$) gives
\begin{align*}
 &\operatorname{Tr}_{\mathbb F_\infty/\Fq}
  \operatorname{res}_\infty
  \left(-e_su^m e_r'u^{-n-1}\,\rd u\right)\\
 &\qquad=
 \operatorname{res}_{\infty_u}
 \left(-\operatorname{Tr}_{K/\Fq(u)}(e_se_r')
 u^{m-n-1}\,\rd u\right)
 =\delta_{rs}\delta_{mn}.
\end{align*}
Here uniqueness follows because the pole divisor of $u\in A$ is supported at
$\infty$.
\end{proof}

\begin{definition}{Write
$(\overline{\Omega})_{\one}\widehat\otimes_{\Fq}A_{\two}$ for the
completed tensor product for the quotient topology on
$\overline{\Omega}$.  For an ordered $\Fq$-basis $(a_i)_{i\geqslant0}$ of
$A$, the \emph{residue--Pontryagin element} is
\begin{equation}\label{eq:residue-Pontryagin-element}
 \mathbf P_A
 :=\sum_{i\geqslant0} a_i^*\otimes a_i
 \;\in\; (\overline{\Omega})_{\one}\widehat\otimes_{\Fq}A_{\two},
\end{equation}
\revblue{The series converges because $a_i^*\to0$.  It is independent of the
chosen basis.  Contraction with the residue pairing identifies it with the
identity endomorphism of $\overline\Omega$.}}
\end{definition}
\begin{proposition}\label{prop:pontryagin-explicit}
In the completed tensor,
\[
 \mathbf P_A
 =\sum_{r=1}^h\sum_{n\geqslant0}
 \left(-e_r'u^{-n-1}\,\rd u\bmod\Omega_A\right)_{\one}
 \otimes (e_ru^n)_{\two} .
\]
For $A=\Fq[t]$ this specializes to the display below.
\end{proposition}
\begin{proof}
{\revblue{Lemma~\ref{lem:residue-dual-trace-basis} shows that the two families
are dual for the residue pairing.  The dual-coordinate vectors tend to zero, so
the displayed series converges in
$(\overline{\Omega})_{\one}\widehat\otimes_{\Fq}A_{\two}$.  It is therefore
the coevaluation \eqref{eq:residue-Pontryagin-element}.}}
\end{proof}
\begin{example}
For the polynomial ring $ A = \Fq[t]$, we have
\[
  \left\langle t^i,\; -\frac{1}{t^{j+1}}\,\rdt \right\rangle_{\infty}
  =
  \res_{t=\infty}\left(-\frac{1}{t^{1 + (j -i )}}\,\rdt\right)
  =
  \delta_{i,j}.
\]
Therefore the residue--Pontryagin class is defined by\revblue{:}
\[
  \mathbf{P}_{A}
  =
  \sum_{i\geqslant0}\left(-\frac{\rdt_\one}{t_\one^{i+1}}\bmod\Omega_A\right)\otimes
  t_\two^i.
\]
A natural lift of
\(\left(\frac{1}{t_\one^{i+1}}\bmod A\right)\rdt_\one\)
is
\(\frac{1}{t_\one^{i+1}}\rdt_\one\).
Under this lift, \(\mathbf{P}_{A}\) admits the meromorphic lift
\[
  K(\one,\two)
  =
  -\frac{1}{t_\one-t_\two}\,\rdt_\one
  =
  \sum_{i\geqslant 0} -\frac{t_\two^i}{t_\one^{i+1}}\,\rdt_\one.
\]
In general, the residue--Pontryagin class does not determine a unique
meromorphic lift.
\end{example}

Recall that $\mathbb T_A=\Cinf\widehat\otimes_{\Fq}A$.  In these coordinates,
\begin{equation}\label{eq:Gauss-basis-norm}
 \mathbb T_A=
 \left\{\sum_{i\geqslant0}x_i\otimes a_i:x_i\longrightarrow0\right\},
 \qquad
 \left\|\sum_{i\geqslant0}x_i\otimes a_i\right\|_{\mathrm G}
 =\sup_i|x_i|.
\end{equation}
The following duality turns a continuous function on residue classes into a
Tate coordinate.


\begin{lemma}
\label{lem:Tate-continuous-residue-duality}
Contraction gives a canonical isometric isomorphism
\begin{equation}\label{eq:Tate-continuous-dual-identification}
 \begin{aligned}
 \mathscr{F}:\mathbb T_A&\xrightarrow{\ \sim\ }
 \operatorname{Hom}^{\mathrm{cont}}_{\Fq}
       (A^\vee_{\mathrm{res}},\Cinf),\\
 \mathscr{F}\!\left(\sum_{i\geqslant0}x_i\otimes a_i\right)(\bar\xi)
   &=\sum_{i\geqslant0}x_i\bar\xi(a_i).
 \end{aligned}
\end{equation}
The continuous dual carries the supremum norm.  The inverse is
\begin{equation}\label{eq:Tate-continuous-dual-inverse}
 \mathscr{F}^{-1}(\epsilon)=\sum_{i\geqslant0}\epsilon(a_i^*)\otimes a_i.
\end{equation}
Both maps are independent of the basis.  Multiplication by $1\otimes a$
corresponds to precomposition by $\bar\xi\mapsto a\bar\xi$, where
$(a\bar\xi)(b)=\bar\xi(ab)$.
\end{lemma}

\begin{proof}
For $\omega=\sum_i x_i\otimes a_i$, the series defining $\mathscr{F}(\omega)$
converges uniformly on $A^\vee_{\mathrm{res}}$, since
$\bar\xi(a_i)\in\Fq$ and $x_i\to0$.  Its finite sums are continuous.
Conversely, continuity of $\epsilon$ and $a_i^*\to0$ imply
$\epsilon(a_i^*)\to0$.  Applying $\epsilon$ to the finite-coordinate
truncations of $\bar\xi$ gives
\[
 \epsilon(\bar\xi)=
 \lim_{N\to\infty}\sum_{i<N}\bar\xi(a_i)\epsilon(a_i^*).
\]
It follows that $\mathscr{F}\mathscr{F}^{-1}=\mathrm{id}$; evaluation on $a_i^*$ gives
$\mathscr{F}^{-1}\mathscr{F}=\mathrm{id}$.  The ultrametric inequality and these same coordinate
vectors show that
$\|\mathscr{F}(\omega)\|_{\sup}=\sup_i|x_i|$.
Contraction on the dense subspace $\Cinf\otimes_{\Fq}A$ is intrinsic, so
its continuous extension and its inverse are basis independent.
The assertion about multiplication follows from
$\bar\xi(ab)=(a\bar\xi)(b)$.
\end{proof}

\revsubsection{The anchor map}{The anchor map}

An \emph{anchor map} is an $A$-linear map
$\jmath:\Omega_A\to\Lambda_\varphi$.  {It is called a
\emph{unit anchor} if it is an isomorphism.} 
Write $\jmath_K$ and $\jmath_\infty$ for its extensions
to ${\Omega_A^{\rat}}$ and $K_\infty\otimes_A\Omega_A$, with scalar action through
$\iota$.  {Unit} anchors exist exactly when
$[\Omega_A]=[\Lambda_\varphi]$; their choices then form an
$\Fq^\times$-torsor.

\begin{proposition} 
\label{prop:isogeny-direct-period-map}
Let $J:\varphi\to\psi$ have free target lattice
$\Lambda_\psi=\iota(A) \tilde{\pi}$, and put $z_J= \tilde{\pi}/h_J(\Xi)$.
\begin{enumerate}[label=\textup{(\roman*)}]
\item\label{item:anchor-lattice}
The source lattice is $\Lambda_\varphi=\iota(\mathfrak i_J)z_J$.
\item\label{item:anchor-map}
If $\Omega_A=\mathfrak i_J\nu_J$ for a rational differential $\nu_J$,
then
\begin{equation}\label{eq:anchor}
 \jmath_J(a\nu_J)=\iota(a)z_J\qquad(a\in\mathfrak i_J)
\end{equation}
is a {unit anchor}.  Its localized extension sends $\nu_J$ to $z_J$.
\end{enumerate}
\end{proposition}
\begin{proof}
For \ref{item:anchor-lattice}, multiply
$h_J(\Xi)^{-1}\Lambda_\psi=\iota(\mathfrak i_J^{-1})\Lambda_\varphi$
by $\iota(\mathfrak i_J)$ and use $\mathfrak i_J\mathfrak i_J^{-1}=A$.
For \ref{item:anchor-map}, the modules have presentations
$\mathfrak i_J\nu_J$ and $\iota(\mathfrak i_J)z_J$, so the map is
$A$-linear and bijective.  Tensoring with $K$ proves the last assertion;
the differential $\nu_J$ itself need not belong to $\Omega_A$.
\end{proof}

\begin{corollary}
\label{cor:isogeny-direct-period-map-exact-scope}
Every rank-one source admits an isogeny with free target lattice.
For any such $J$, an equality $\Omega_A=\mathfrak i_J\nu_J$ exists
exactly when $[\Lambda_\varphi]=[\Omega_A]$.
\end{corollary}
\begin{proof}
Write $\Lambda_\varphi=\iota(\mathfrak a)\rho$ with
$\mathfrak a$ integral, by rescaling a fractional-ideal presentation.
The ideal isogeny for $\mathfrak a$ has target lattice
$c\,\iota(\mathfrak a^{-1})\Lambda_\varphi=c\,\iota(A)\rho$.
For a fixed free target, part~\ref{item:anchor-lattice} gives the source
class $[\mathfrak i_J]$.  Any isomorphism
$\mathfrak i_J\to\Omega_A$ extends to their one-dimensional $K$-spaces,
where it is multiplication by a rational differential $\nu_J$.
\end{proof}

\begin{example}
If both $\Omega_A=A\nu$ and $\Lambda_\varphi=\iota(A)\tilde{\pi}_\varphi$ are free,
the identity isogeny gives $a\nu\mapsto\iota(a)\tilde{\pi}_\varphi$.
Freeness of $\Omega_A$ alone does not imply freeness of the period lattice.
\end{example}

\begin{example}[The HHY datum]
In the rational model of \cite{HHY26}, the removed point on $\mathbb P^1$
has degree $\Ninf\geqslant2$ and roots $\eta,\eta^q,\ldots,\eta^{q^{\Ninf-1}}$.
Let $\mathfrak i_\infty$ be the ideal of the retained usual infinity, and
put $\theta=\iota(t)$.
The source module $\Psi$ has shtuka function
\[
 f_\Psi(t)=\frac1{t-\eta}-\frac1{\theta-\eta}.
\]
\revblue{The target module $\Phi$ is the $\eta^{(1)}$-type Drinfeld module of
\cite{HHY26}.  Its monic twisted polynomial is:}
\[
 \eta_J=\bigl(\tau+(\theta-\eta)^{-1}\bigr)^2
 =\tau^2+\bigl((\theta^q-\eta^q)^{-1}+(\theta-\eta)^{-1}\bigr)\tau
 +(\theta-\eta)^{-2}
\]
\revblue{This is a scalar multiple of the isogeny $\lambda:\Psi\to\Psi^{(2)}$
in \cite[(120)]{HHY26} and defines a separable isogeny $\eta_J:\Psi\to\Phi$.
In particular, $\partial\eta_J=(\theta-\eta)^{-2}$.}
The resulting isogeny datum is
$J=[\eta_J;h_J;\mathfrak i_J,\mu_J]:\Psi\longrightarrow\Phi$
with $\mathfrak i_J=\mathfrak i_\infty^2$ and $\mu_J=1$.
The multiplier $h_J$ is defined by the motive frames
$\eta_{J}^{!}(s_\Phi)=h_Js_\Psi$, and is given explicitly by
\begin{equation}\label{eq:active-P1-hJ}
 h_J(t)=\frac{\theta-\eta^q}
 {(\theta-\eta)(t-\eta)(t-\eta^q)},\qquad
 h_J(\theta)=(\theta-\eta)^{-2}.
\end{equation}
Thus $h_J(\Xi)=h_J(\theta)=(\theta-\eta)^{-2}$.
Since $\Omega_A=\mathfrak i_\infty^2\rd t$ and
$\Lambda_\Phi=\iota(A)\pi_\Phi$, the anchor is the restriction of
$\rd t\mapsto(\theta-\eta)^2\pi_\Phi$ to $\mathfrak i_\infty^2\rd t$.
Writing $\jmath_*$ for this anchor, the same computation gives
\[
 \operatorname{Hom}_A(\Omega_A,\Lambda_\Phi)
 =\iota(\mathfrak i_\infty^{-2})(\theta-\eta)^{-2}\jmath_* .
\]
\end{example}

\begin{proposition}\label{prop:anchor-isogeny}
Let $J':\varphi\to\varphi'$ be a separable isogeny and let
$\jmath_*:\Omega_A\xrightarrow{\sim}\Lambda_\varphi$ be a {unit anchor}.
Every anchor for $\varphi'$ is uniquely of the form
\[
 \jmath'=\iota(a){h_{J'}(\Xi)}\jmath_*,
 \qquad a\in\mathfrak i_{J'}^{-1}.
\]
\end{proposition}
\begin{proof}
The lattice formula identifies $\Lambda_{\varphi'}$ with
${h_{J'}(\Xi)}\iota(\mathfrak i_{J'}^{-1})\jmath_*(\Omega_A)$.
After extending to $K$, any $K$-linear map to $K\otimes_A\Lambda_{\varphi'}$ is multiplication by a
unique $a\in K$, followed by ${h_{J'}(\Xi)}\jmath_*$.
It takes $\Omega_A$ into the target lattice exactly when
$a\Omega_A\subseteq\mathfrak i_{J'}^{-1}\Omega_A$.
Localizing at each prime and cancelling a local basis of the
invertible module $\Omega_A$ gives $a\in\mathfrak i_{J'}^{-1}$.
\end{proof}

\subsection{The special function associated with a unit anchor}

\begin{definition}\label{def:Fjmath}
    Let $ 0\ne\jmath\in\operatorname{Hom}_A(\Omega_A,\Lambda_\varphi)$ be an anchor map of  $\varphi$.
Define
\begin{equation}\label{eq:residue-dual-exponential-map}
 \mathscr{F}_{\varphi,\jmath}:A^\vee_{\mathrm{res}} = \widehat{\Omega}_A/ \Omega_A \longrightarrow\Cinf,
 \qquad
\bar\xi\longmapsto\exp_\varphi\bigl(\jmath_\infty(\xi)\bigr),
\end{equation}
where $\xi$ is any representative in $\widehat{\Omega}_A$.
\end{definition}
 One verifies that $\mathscr{F}_{\varphi,\jmath}$ is well defined and continuous.
\begin{definition}\label{def:omega-jmath}
 We define $\omega_\jmath:=\mathscr{F}^{-1}(\mathscr{F}_{\varphi,\jmath})\in\TT_A$, the tensor corresponding to $\mathscr{F}_{\varphi,\jmath}$.
\end{definition}
More precisely,  {let} $(a_i)_{i\geqslant0}$ be any ordered $\Fq$-basis of $A$, let
$(a_i^*)_{i\geqslant0}$ be its coordinate family in
$A^\vee_{\mathrm{res}}$. Choose
$\eta_i\in K_\infty\otimes_A\Omega_A$ representing $a_i^*$ under
\eqref{eq:residue-dual-coordinate-line}. In other words, 
\[
    \sum_{i=0}^\infty \eta_i \otimes a_i  = \mathbf{P}_A.
\]
Then
\begin{equation}\label{eq:generating-function-residue-dual-expansion}
 \omega_\jmath
 =\sum_{i\geqslant0}\mathscr{F}_{\varphi,\jmath}(\eta_i)\otimes a_i
 \quad\text{in }\mathbb T_A.
\end{equation}
With trace-dual representatives, this expansion comes from a
Cauchy representative.  The completed residue class is canonical,
although its meromorphic lift need not be.
\begin{proposition}
\label{prop:generating-unit-candidate}
The tensor $\omega_\jmath$ is nonzero and is a special function:
\begin{equation}\label{eq:direct-candidate-special-eigenrelation}
 \varphi_a(\omega_\jmath)=(1\otimes a)\omega_\jmath
 \quad(a\in A),
 \qquad
 \omega_\jmath\in\sfm(\varphi).
\end{equation}

\end{proposition}

\begin{proof}
Changing $\xi$ in \eqref{eq:residue-dual-exponential-map} by an element of
$\Omega_A$ changes $\jmath_\infty(\xi)$ by a period, so $\mathscr{F}_{\varphi,\jmath}$
is well defined.  It is continuous because $\jmath$ and
$\exp_\varphi$ are continuous.  If it vanished identically, then
$\jmath_\infty(K_\infty\otimes_A\Omega_A)$ would be contained in the
discrete lattice $\Lambda_\varphi$.  A nonzero $K_\infty$-line contains
arbitrarily small nonzero elements, so this is impossible.  Hence
$\omega_\jmath\ne0$.

For $a\in A$, let $a\bar\eta$ denote the functional
$b\mapsto\bar\eta(ab)$.  Under residue duality it is represented by
$a\eta$.  Exponential equivariance therefore gives
\[
 \varphi_a\bigl(\mathscr{F}_{\varphi,\jmath}(\bar\eta)\bigr)
 =\exp_\varphi\bigl(\iota(a)\jmath_\infty(\eta)\bigr)
 =\mathscr{F}_{\varphi,\jmath}(a\bar\eta).
\]
Under \eqref{eq:Tate-continuous-dual-identification}, coefficientwise
$\varphi_a$ acts on the values, whereas multiplication by $1\otimes a$ is
precomposition by $\bar\eta\mapsto a\bar\eta$.  This proves
\eqref{eq:direct-candidate-special-eigenrelation}.  

Lemma~\ref{lem:scalar-shtuka-equation} now gives
\begin{equation}\label{eq:direct-candidate-scalar-shtuka}
 \omega_\jmath^{(1)}=f_\varphi\omega_\jmath.
\end{equation}
The continuous-duality inverse gives
\eqref{eq:generating-function-residue-dual-expansion}; a change of
representative changes its image under $\jmath$ by a period.
\end{proof}

\begin{lemma}\label{lem:omega_jmath}
For trace-dual bases $(e_r),(e_r')$ of $A/\Fq[u]$, put $\lambda=\jmath(\rdu)$.
Then
\[
 \widehat{\mathcal E}^{\mathrm{sep}}_{\varphi,\lambda}
 \left(
 \frac{\sum_{r=1}^h e_r'(\one)e_r(\two)}
 {u(\two)-u(\one)}\,  s_\varphi\right)
 =\omega_\jmath .
\]
\end{lemma}
\begin{proof}
Index the basis $(e_ru^n)$ as $(a_i)$, and let $(a_i^*)$ be its
residue-dual coordinate family.  By
\eqref{eq:generating-function-residue-dual-expansion},
\[
 \omega_\jmath
 =\sum_i\exp_\varphi\bigl(\jmath_\infty(a_i^*)\bigr)\otimes a_i.
\]
Expand the denominator in the region $|u(\two)|<|u(\one)|$:
\[
 \frac{1}{u(\two)-u(\one)}
 =-\sum_{n\geqslant0}\frac{u(\two)^n}{u(\one)^{n+1}}.
\]
The definition of separated evaluation then gives
\[
 \widehat{\mathcal E}^{\mathrm{sep}}_{\varphi,\lambda}
 \left(
 \frac{\sum_{r=1}^h e_r'(\one)e_r(\two)}
 {u(\two)-u(\one)}\,s_\varphi\right)
 =-\sum_{r,n}
 \Exp\!\left(\frac{\iota(e_r')\lambda}{\iota(u)^{n+1}}\right)
 \otimes e_ru^n .
\]
By Lemma~\ref{lem:residue-dual-trace-basis}, $a_i^*$ is the coordinate
class attached to $-e_r'u^{-n-1}$ in the first variable.  Hence
Definition~\ref{def:levelI}, exponential equivariance
\eqref{eq:functional}, and the defining choice of $\lambda$ give, for each
$i=(r,n)$,
\[
 -\Exp\!\left(\frac{\iota(e_r')\lambda}{\iota(u)^{n+1}}\right)
 =\exp_\varphi\bigl(\jmath_\infty(a_i^*)\bigr).
\]
The coefficients decay geometrically.  Hence the series converges in the
evaluation completion and may be evaluated term by term.  Substitution
of the last display gives
\[
 \widehat{\mathcal E}^{\mathrm{sep}}_{\varphi,\lambda}
 \left(
 \frac{\sum_{r=1}^h e_r'(\one)e_r(\two)}
 {u(\two)-u(\one)}\,s_\varphi\right)
 =\sum_i\exp_\varphi\bigl(\jmath_\infty(a_i^*)\bigr)\otimes a_i
 =\omega_\jmath .
\]
\end{proof}

\subsection{\texorpdfstring{The relation between $\omega_\jmath$ and the Gazda--Maurischat transform}{The relation between omega-jmath and the Gazda--Maurischat transform}}

Fix a nonconstant separating $u\in A$ and put
$ {\mathfrak{d}}_u= {\mathfrak{d}}_{A/\Fq[u]}$.
The direct representation of $\omega_\jmath$ is mediated by the
Gazda--Maurischat period isomorphism.  With $-\delta^{\mathrm{GM}}_u$
denoting this isomorphism (as normalized in \cite{GM24}; the convention
of \cite{GM21} differs by an overall minus sign),
Lemma~\ref{lem:GM-transform-evaluations} identifies it with the
transform \eqref{eq:GM-transform}:
\begin{equation}\label{eq:GM-delta-identification}
 \GM_{u;\widehat{\mathfrak b}}(\lambda)
 =-\delta^{\mathrm{GM}}_u\bigl(\iota(\mathfrak b)\lambda\bigr)
 \in\sfm(\varphi).
\end{equation}
This identification is exactly the resolvent formula in
\cite[Theorem~3.11]{GM21}, after applying the multiplication isomorphism
$D_u\simeq\mathfrak{d}_{A/\Fq[u]}$ of
Proposition~\ref{prop:Du-generator-any}.

Moreover, for $ \lambda' \in \iota(\mathfrak{b}) \Lambda_\varphi \subset \mathfrak{d}_u \Lambda_\varphi $, we have
\begin{equation}\label{eq:deltau}
\delta^{\mathrm{GM}}_u\bigl( \lambda' \bigr) = \mathcal{E}_{\varphi,\lambda'/\iota(\mathfrak{b})} \left( \widehat{\mathfrak{b}}(\one,\two) \frac{1}{u_\two - u_\one}  s_\varphi\right)=\mathcal{E}_{\varphi,\lambda'} \left( \frac{\widehat{\mathfrak{b}}(\one,\two)}{\mathfrak{b}(\one)} \frac{1}{u_\two - u_\one}  s_\varphi
    \right).
\end{equation}
Here $\lambda'/\iota(\mathfrak b)$ is a source period.
Proposition~\ref{prop:GM-residue-dual-coordinate} below expresses
the residue-dual coordinate without choosing a single different element.

\begin{lemma}\label{lem:domain}
The period $\jmath(\rdu)$ belongs to
$\iota( {\mathfrak{d}}_u)\Lambda_\varphi$, the domain of
$-\delta^{\mathrm{GM}}_u$.
\end{lemma}
\begin{proof}
Choose $\sum_s\mathfrak b_s\eta_s=\rdu$ by
\eqref{eq:relative-different-multiplication}.
Then $\jmath(\rdu)=\sum_s\iota(\mathfrak b_s)\jmath(\eta_s)$,
with $\mathfrak b_s\in {\mathfrak{d}}_u$ and
$\jmath(\eta_s)\in\Lambda_\varphi$.
\end{proof}

\begin{proposition}
\label{prop:GM-residue-dual-coordinate}
Let $u\in A$ be nonconstant and separating.
Then
\begin{equation}
 \omega_\jmath=\delta^{\mathrm{GM}}_u\bigl(\jmath(\rdu)\bigr).
\end{equation}
\end{proposition}

\begin{proof}

Use the trace-dual bases in the preceding proof and choose the ordered lifts
$\widehat{\mathfrak b}^{e}_s(\one,\two)
:=\sum_r(\mathfrak b_s e_r')(\one)e_r(\two)$ of the classes
$[\widehat{\mathfrak b}^{e}_s]\in D_u$,
\begin{equation}\label{eq:different-trace-Casimir-ordered-lift}
 \widehat{\mathfrak b}^{e}_s(\one,\two)
 =\sum_r(\mathfrak b_s e_r')(\one)e_r(\two).
\end{equation}
They are integral because $\mathfrak b_s\in\mathfrak{d}_{A/\Fq[u]}$ and
$e_r'\in\mathfrak{d}_{A/\Fq[u]}^{-1}$, and
Proposition~\ref{prop:Du-generator-any} identifies their diagonal classes.
The identification \eqref{eq:GM-delta-identification} and the
normally convergent expansion \eqref{eq:GM-cauchy-expansion} now give
\begin{align}
  \delta^{\mathrm{GM}}_u\bigl(\jmath(\rdu)\bigr)
 &=-\sum_s\GM_{u;\widehat{\mathfrak b}^e_s}\bigl(\jmath_\infty(\eta_s)\bigr)  \label{eq:onestep-trace-basis-expansion} \\
 &=-\sum_{r,n}
   \exp_\varphi\!\left(
     \sum_s\iota\!\left(\frac{\mathfrak b_s e_r'}{u^{n+1}}\right)
                   \jmath_\infty(\eta_s)\right)\otimes e_ru^n\notag\\
 &=-\sum_{r,n}
   \exp_\varphi\!\left(
     \jmath_\infty\!\left(\frac{e_r'\,\rdu}{u^{n+1}}\right)\right)
       \otimes e_ru^n
 = \omega_\jmath. \nonumber
\end{align}
 The final identification with $\omega_\jmath$ uses Lemma~\ref{lem:omega_jmath} (or the expansion of theresidue-dual
formula \eqref{eq:generating-function-residue-dual-expansion}).
\end{proof}

\begin{theorem}
 \label{thm:identity-dual-factorization}
For every $\eta\in\Omega_A$, we have the identity
\begin{equation}\label{eq:canonical-family-scalar-factorization}
 \mathcal H_{\varphi}(\jmath\eta)=\omega_\jmath \otimes \eta \in \sfm(\varphi)\otimes_A\Omega_A.
\end{equation}
The identity holds on the Tate domain
and for the corresponding meromorphic continuations.
\end{theorem}

\begin{proof}
Theorem~\ref{thm:canonical-assembly} and the construction of $\omega_\jmath$
imply both sides of \eqref{eq:canonical-family-scalar-factorization} belong to $\sfm(\varphi)\otimes_A\Omega_A$.
Firstly, we show the identity holds for $ \eta = \rd u $ for a nonconstant separating $u\in A$.
Fix a decomposition
$\rdu=\sum_s\mathfrak b_s\eta_s$ with $\mathfrak b_s\in\mathfrak{d}_{A/\Fq[u]}$ as in \eqref{eq:relative-different-multiplication}.  By $K$-linearity of $\jmath$,
\[
 \jmath(\rdu)=\sum_s\iota(\mathfrak b_s)\jmath(\eta_s).
\]
Proposition~\ref{prop:GM-residue-dual-coordinate} gives
$\omega_\jmath=\delta^{\mathrm{GM}}_u(\jmath(\rdu))$, and the
computation \eqref{eq:onestep-trace-basis-expansion} of that identity
expands
\[
  \omega_\jmath
 =- \sum_s\GM_{u;\widehat{\mathfrak b}^e_s}\bigl(\jmath_\infty(\eta_s)\bigr)\\
\]

On the other hand, Equation~\eqref{eq:GM-b-different}
and
Corollary~\ref{cor:Hphi-linearity} give
\[
 -\GM_{u;\widehat{\mathfrak b}^e_s}(\jmath(\eta_s))\otimes\rdu(\two)
 =\mathcal H_\varphi\bigl(\jmath(\mathfrak b_s\eta_s)\bigr).
\]
Multiplying the display for $\omega_\jmath$ by $\rdu(\two)$ yields
\[
 \omega_\jmath\otimes\rdu
 =\sum_s\mathcal H_\varphi\bigl(\jmath(\mathfrak b_s\eta_s)\bigr)
 =\mathcal H_\varphi\Bigl(\sum_s\iota(\mathfrak b_s)\jmath(\eta_s)\Bigr)
 =\mathcal H_\varphi(\jmath(\rdu)),
\]
which confirms the identity for $\eta=\rdu$.

Both sides of \eqref{eq:canonical-family-scalar-factorization} are rational Anderson differentials, so
Theorem~\ref{thm:rational-anderson-continuation}(ii) supplies their
canonical meromorphic continuations; the identity proved on the Tate
domain persists for them by the identity principle on the connected curve
$U_{\Cinf}^{\mathrm{an}}$. Moreover, the residue isomorphism of
Proposition~\ref{lem:residue} gives
\begin{equation}\label{eq:residue-dual-coordinate-residue}
  \res_\Xi \omega_\jmath \rd u =  \res_\Xi \mathcal H_\varphi(\jmath(\rd u)) = \jmath(\rd u). 
\end{equation}
 
For the general case, let $\eta=b\,\rdu$ with $b\in K$.
From \eqref{eq:residue-dual-coordinate-residue}, the residue of $ \omega_\jmath  \otimes \eta $ is given by\revblue{:}  
\[  
\res_\Xi \omega_\jmath \otimes \eta = \jmath(\rd u) b(\Xi)  = \jmath(\eta).
\]
It is identically equal to the residue of $\mathcal H_\varphi(\jmath(\eta))$ by
 Proposition~\ref{lem:residue}. Now Theorem~\ref{thm:rational-anderson-continuation}\ref{item:ratand-residue} confirms the identity \eqref{eq:canonical-family-scalar-factorization} for all $\eta\in\Omega_A$.

\end{proof}

\begin{remark}\label{rem:GM24-main-theorem}
Proposition~\ref{prop:GM-residue-dual-coordinate} and
Theorems~\ref{thm:identity-dual-factorization}
and~\ref{thm:canonical-assembly} give directly
\[
 \delta^{\mathrm{GM}}_u\bigl(\jmath(\rdu)\bigr)\,\rdu
 =\omega_\jmath\otimes\rdu
 =\mathcal H_\varphi\bigl(\jmath(\rdu)\bigr)
 =\GM_u^\Omega\bigl(\jmath(\rdu)\bigr).
\]
This recovers, in essence, the main theorem of \cite{GM24}.
\end{remark}

\subsection{The logarithmic inverse}

We now recover $\jmath$ from $\omega_\jmath$.  The construction uses the
continuous duality and the local Drinfeld logarithm.

For $\omega\in\sfm(\varphi)$, let
$\mathscr{F}_\omega = \mathscr{F}(\omega):A^\vee_{\mathrm{res}}\to\Cinf$ be the continuous
$\Fq$-linear map
corresponding to $\omega$ under
\eqref{eq:Tate-continuous-dual-identification}.
\begin{proposition}
\label{prop:stabilized-logarithmic-lift}
Let $\omega\in\sfm(\varphi)$.
\begin{enumerate}[label=\textup{(\roman*)},leftmargin=2.2em]
\item\label{item:logarithmic-lift-existence} There is a unique continuous
$K_\infty$-linear map
\begin{equation}\label{eq:logarithmic-lift-characterization}
 \widetilde{\mathscr{L}}_\omega:\widehat{\Omega}_A\longrightarrow\Cinf,
 \qquad
 \exp_\varphi\bigl(\widetilde{\mathscr{L}}_\omega(\eta)\bigr)
 =\mathscr{F}_\omega(\bar\eta),
\end{equation}
where $\bar\eta$ is the class of $\eta$ in
$\overline\Omega\simeq A^\vee_{\mathrm{res}}$.  

\item\label{item:logarithmic-lift-formula}
For any nonconstant $u\in A$, if $n$ is sufficiently
large that the local Drinfeld logarithm is defined at the indicated value,
then
\begin{equation}\label{eq:direct-logarithmic-lift-formula}
 \widetilde{\mathscr{L}}_\omega(\eta)
 =\iota(u)^n
  \log_\varphi\!\left(
   \mathscr{F}_\omega(\overline{u^{-n}\eta})
  \right).
\end{equation}
The right-hand side is independent of $n$ and $u$.
\end{enumerate}
\end{proposition}

\begin{proof}
\emph{\ref{item:logarithmic-lift-formula}: construction and stabilization.}
Since $\exp_\varphi(z)=z+O(z^q)$, it restricts on a sufficiently small open
disc about zero to an $\Fq$-linear analytic isomorphism; denote its inverse
there by $\log_\varphi$.  Exponential equivariance implies, for each $a\in A$,
\begin{equation}\label{eq:local-logarithm-equivariance}
 \log_\varphi(\varphi_a(x))=\iota(a)\log_\varphi(x)
\end{equation}
whenever $x$ is sufficiently small.  The radius may be decreased for the
finitely many scalars used in a given calculation.

The special-function identity and
\eqref{eq:Tate-continuous-dual-identification} give
\begin{equation}\label{eq:special-continuous-map-equivariance}
 \mathscr{F}_\omega(a\bar\eta)
 =\varphi_a\bigl(\mathscr{F}_\omega(\bar\eta)\bigr)
 \qquad(a\in A).
\end{equation}
Fix a nonconstant $u\in A$ and set
\[
 x_n(\eta):=\mathscr{F}_\omega(\overline{u^{-n}\eta}).
\]
Since $|\iota(u)|>1$, the class of $u^{-n}\eta$ tends to zero in
$\overline\Omega$, so $x_n(\eta)\to0$.  Moreover
\[
 x_n(\eta)=\varphi_u(x_{n+1}(\eta)).
\]
For large $n$, equation~\eqref{eq:local-logarithm-equivariance} gives
\[
 \iota(u)^n\log_\varphi(x_n(\eta))
 =\iota(u)^{n+1}\log_\varphi(x_{n+1}(\eta)),
\]
which proves stabilization in $n$.

\emph{\ref{item:logarithmic-lift-existence}: the lifting property and uniqueness.}
For such $n$, global exponential equivariance gives
\begin{align*}
 \exp_\varphi\!\left(\iota(u)^n\log_\varphi(x_n(\eta))\right)
 &=\varphi_{u^n}(x_n(\eta))
 =\mathscr{F}_\omega(\bar\eta).
\end{align*}
This proves \eqref{eq:logarithmic-lift-characterization}.  Additivity follows
by taking one sufficiently large $n$ for two inputs.  Applying
\eqref{eq:local-logarithm-equivariance} also to $a\in A$ gives
\[
 \widetilde{\mathscr{L}}_\omega(a\eta)
 =\iota(a)\widetilde{\mathscr{L}}_\omega(\eta).
\]
Thus the lift is $K$-linear after division by nonzero elements of $A$.  Near
zero the stabilized formula may be taken with $n=0$, so the lift is
continuous.  Density of $K$ in $K_\infty$ upgrades $K$-linearity to
$K_\infty$-linearity.

The lift is unique: the difference of two such lifts has image in the
discrete lattice $\Lambda_\varphi$, whereas a nonzero $K_\infty$-linear image
would be a nonzero $K_\infty$-line.  Uniqueness also makes
\eqref{eq:direct-logarithmic-lift-formula} independent of $u$.
\end{proof}

The following result \revblue{shows} that 
$ \mathscr{L}_{\varphi} $ coincides with the residue map at $\Xi $.

\begin{theorem}
\label{thm:direct-logarithmic-special-line}
\begin{enumerate}[label=\textup{(\roman*)},leftmargin=2.2em]
\item
Restriction to the integral differential lattice gives a canonical \(A\)-module isomorphism
\begin{equation}\label{eq:direct-logarithmic-special-line-isomorphism}
 \mathscr{L}_\varphi:
 \sfm(\varphi)\xrightarrow{\ \sim\ }
 \operatorname{Hom}_A(\Omega_A,\Lambda_\varphi),
 \qquad
 \omega\longmapsto
 \widetilde{\mathscr{L}}_\omega|_{\Omega_A}.
\end{equation}
Its inverse sends \(\jmath\) to the residue-dual exponential tensor
\(\omega_\jmath\) of Proposition~\ref{prop:generating-unit-candidate}.
\item
Consequently, freeness of $\sfm(\varphi)$ \revblue{is equivalent to}
$\Lambda_\varphi\simeq\Omega_A$.
\item
Moreover, $\omega_\jmath$ is an $A$-basis of
$\sfm(\varphi)$ if and only if $\jmath$ is an isomorphism.
\end{enumerate}
\end{theorem}

\begin{proof}
\emph{(i).}
If \(\alpha\in\Omega_A\), then \(\bar\alpha=0\), so
\eqref{eq:logarithmic-lift-characterization} gives
\(\widetilde{\mathscr{L}}_\omega(\alpha)\in\Lambda_\varphi\).  Thus restriction
defines the map \(\mathscr{L}_\varphi\).  It is \(A\)-linear in \(\omega\): under
\eqref{eq:Tate-continuous-dual-identification}, multiplication by
\(1\otimes a\) replaces \(\mathscr{F}_\omega(\bar\eta)\) by
\(\mathscr{F}_\omega(a\bar\eta)\), and uniqueness identifies the corresponding
lift with \(\eta\mapsto\widetilde{\mathscr{L}}_\omega(a\eta)\).

Conversely, let
\(\jmath\in\operatorname{Hom}_A(\Omega_A,\Lambda_\varphi)\) and extend it
\(K_\infty\)-linearly to \(\widehat{\Omega}_A\).  For \(\jmath\ne0\), Proposition
\ref{prop:generating-unit-candidate} constructs \(\omega_\jmath\) with
\[
 \mathscr{F}_{\varphi,\jmath}(\bar\eta)
 =\exp_\varphi(\jmath_\infty(\eta));
\]
for \(\jmath=0\), put \(\omega_\jmath=0\).  The uniqueness just proved yields
\(\widetilde{\mathscr{L}}_{\omega_\jmath}=\jmath\), hence
\(\mathscr{L}_\varphi(\omega_\jmath)=\jmath\).  In the other direction, start
with \(\omega\) and put
\(\jmath=\widetilde{\mathscr{L}}_\omega|_{\Omega_A}\).  Since \(\Omega_A\) spans
\(\widehat{\Omega}_A\) over \(K_\infty\), one has
\(\jmath=\widetilde{\mathscr{L}}_\omega\); therefore
\(\mathscr{F}_{\varphi,\jmath}=\mathscr{F}_\omega\) and
\(\omega_\jmath=\omega\).  This proves
\eqref{eq:direct-logarithmic-special-line-isomorphism}.

\emph{(ii).}
Finally, \(\operatorname{Hom}_A(\Omega_A,\Lambda_\varphi)\) is free exactly when the two
invertible modules \(\Omega_A\) and \(\Lambda_\varphi\) have the same ideal class;
through \eqref{eq:direct-logarithmic-special-line-isomorphism} this
is assertion (ii).

\emph{(iii).}
Through \eqref{eq:direct-logarithmic-special-line-isomorphism},
$\sfm(\varphi)=A\omega_\jmath$ means that
$\operatorname{Hom}_A(\Omega_A,\Lambda_\varphi)$ is generated by $\jmath$;
being invertible, it is generated by $\jmath$ exactly when $\jmath$ is an
isomorphism.
\end{proof}
 
\begin{corollary} \label{cor:characteristic-residue}
The period and differential coordinates are related as follows.
\begin{enumerate}[label=\textup{(\roman*)},leftmargin=2.2em]
\item\label{item:characteristic-residue-map} The ordinary characteristic residue map
\begin{equation}
\label{eq:direct-characteristic-residue-isomorphism}
\operatorname{res}_\Xi\colon
\operatorname{sf}(\varphi)\otimes_A\Omega_A
\xrightarrow{\ \sim\ }\Lambda_\varphi
\end{equation}
is an isomorphism, and for \(\alpha\in\Omega_A\), \(\omega\in\operatorname{sf}(\varphi)\),
\begin{equation}
 \label{eq:characteristic-residue-anchor}
\operatorname{res}_\Xi(\omega\otimes\alpha)
=
\mathscr{L}_\varphi(\omega)(\alpha).
\end{equation}
\item\label{item:characteristic-residue-inverse} Its inverse is the evaluation on periods
\begin{equation}
 \label{eq:evaluation-on-periods-inverse}
\mathcal{H}_\varphi\big|_{\Lambda_\varphi}\colon
\Lambda_\varphi
\xrightarrow{\ \sim\ }
\operatorname{sf}(\varphi)\otimes_A\Omega_A.
\end{equation}
\end{enumerate}
\end{corollary}

\begin{proof}
\emph{\ref{item:characteristic-residue-map}.}
Since \(\Omega_A\) is invertible, the evaluation map
\[
 \operatorname{ev}\colon
 \operatorname{Hom}_A(\Omega_A,\Lambda_\varphi)\otimes_A\Omega_A
 \xrightarrow{\ \sim\ }\Lambda_\varphi,
 \qquad f\otimes_A\alpha\longmapsto f(\alpha),
\]
\revblue{This map is an isomorphism.  Localization at each prime of \(A\)
turns it into the multiplication isomorphism.  All modules involved are finitely
generated over the noetherian ring \(A\).}  Combined with
Theorem~\ref{thm:direct-logarithmic-special-line}, the composite
\[
 \mathscr R:=\operatorname{ev}\circ
 \bigl(\mathscr L_\varphi\otimes_A\mathrm{id}_{\Omega_A}\bigr)\colon
 \operatorname{sf}(\varphi)\otimes_A\Omega_A
 \xrightarrow{\ \sim\ }\Lambda_\varphi,
 \qquad \omega\otimes_A\alpha\longmapsto
 \mathscr L_\varphi(\omega)(\alpha),
\]
is therefore a well-defined \(A\)-module isomorphism.

We identify this map with the analytic characteristic residue.
By Theorem~\ref{thm:direct-logarithmic-special-line}, every
\(\omega\in\operatorname{sf}(\varphi)\) is \(\omega_\jmath\) for
\(\jmath=\mathscr L_\varphi(\omega)\).  For a pure tensor
\(\omega_\jmath\otimes_A\alpha\),
Theorem~\ref{thm:identity-dual-factorization} and
Proposition~\ref{lem:residue} give
\[
 \operatorname{res}_\Xi(\omega_\jmath\otimes\alpha)
 =\jmath(\alpha)
 =\mathscr R(\omega_\jmath\otimes\alpha).
\]
\revblue{Thus the analytic characteristic residue agrees with \(\mathscr R\)
on pure tensors.  It respects the balancing relations and therefore descends to
the tensor product, where it again coincides with \(\mathscr R\).}  This proves \eqref{eq:direct-characteristic-residue-isomorphism}
and  {\eqref{eq:characteristic-residue-anchor}}.

\emph{\ref{item:characteristic-residue-inverse}.}
For \(\lambda\in\Lambda_\varphi\),
Theorem~\ref{thm:canonical-assembly} gives
\(\mathcal H_\varphi(\lambda)
 \in\operatorname{sf}(\varphi)\otimes_A\Omega_A\),
and Proposition~\ref{lem:residue} gives
\(\operatorname{res}_\Xi\bigl(\mathcal H_\varphi(\lambda)\bigr)=\lambda\).
Hence
\(\operatorname{res}_\Xi\circ
 \mathcal H_\varphi\big|_{\Lambda_\varphi}
 =\mathrm{id}_{\Lambda_\varphi}\);
since \(\operatorname{res}_\Xi\) is an isomorphism by \eqref{eq:direct-characteristic-residue-isomorphism},
its inverse is \(\mathcal H_\varphi\big|_{\Lambda_\varphi}\),
proving  {\eqref{eq:evaluation-on-periods-inverse}}.  The \(A\)-linearity of the inverse is
automatic.
\end{proof}



\begin{remark}
Theorem~\ref{thm:unit-special-arbitrary-degree-public} will show that every
$A$-basis of $\operatorname{sf}(\varphi)$ belongs to $\mathbb T_A^\times$.
\end{remark}

\subsection{Residue-dual property of special functions}
\label{subsec:hg-residue-dual-special-functions}
\revblue{Our normalization uses the residue pairing
\eqref{eq:residue-dual-coordinate-line} from
Section~\ref{subsec:truncated-Cauchy-A}.  The same pairing identifies
\(K_\infty/A\) with the compact dual of \(\Omega_A\).}
Lemma~\ref{lem:Tate-continuous-residue-duality} identifies continuous
functions on {$\overline\Omega$} with their Tate coordinates.

For \(\jmath\in\operatorname{Hom}_A(\Omega_A,\Lambda_\psi)\), extend
\(\jmath\) through \(\iota\) to
\(\jmath_\infty:\widehat{\Omega}_A\to\Cinf\).  Its coordinate
\(\omega_\jmath\) is the residue-dual exponential
\eqref{eq:residue-dual-exponential-map}.  Theorem
\ref{thm:direct-logarithmic-special-line} gives
\begin{equation}
\label{eq:hg-direct-special-line-isomorphism}
 \operatorname{Hom}_A(\Omega_A,\Lambda_\psi)
 \xrightarrow{\ \sim\ }\sfm(\psi),
 \qquad \jmath\longmapsto\omega_\jmath.
\end{equation}
Its inverse is the stabilized local logarithm
\eqref{eq:direct-logarithmic-lift-formula}.  No principality assumption
on either lattice enters this isomorphism.  The characteristic
normalization is
\begin{equation}
\label{eq:hg-jmath-characteristic-residue}
 \operatorname{res}_\Xi(\nu\omega_\jmath)=\jmath(\nu),
 \qquad
 \operatorname{res}_\Xi:
 \sfm(\psi)\otimes_A\Omega_A\xrightarrow{\ \sim\ }\Lambda_\psi.
\end{equation}
These are Theorem~\ref{thm:identity-dual-factorization} with
Proposition~\ref{lem:residue}, and
\eqref{eq:direct-characteristic-residue-isomorphism}.  A finite dual
frame of \(\Omega_A\) identifies
\(\sfm(\psi)\otimes_A\Omega_A\) with all differential-valued
special functions in
\(\Cinf\widehat\otimes_{\Fq}\Omega_A\): each contraction is a
scalar special function, and the frame reconstructs the differential.
Thus its characteristic residue determines it uniquely.

Let \(\hbar \in\operatorname{Hom}_A(\Omega_A,A)\), viewed as
\(1\otimes \hbar\) on \(\Cinf\widehat\otimes_{\Fq}\Omega_A\).
View \(\hbar\) as a \(K_\infty\)-linear map 
\[
\hbar:\widehat\Omega_A \to K_\infty.
\]
  Since
\(\widehat\Omega_A={K_\infty\otimes_A\Omega_A}\) is a one-dimensional
\(K_\infty\)-vector space, this extension satisfies\revblue{:}
\begin{equation}
\label{eq:hg-contraction-symmetry}
 \hbar(\xi)\,\eta=\hbar(\eta)\,\xi
 \qquad(\xi,\eta\in\widehat\Omega_A)\revblue{.}
\end{equation}
\revblue{Both sides are \(K_\infty\)-bilinear in \((\xi,\eta)\).  To check
equality, write \(\xi=x\nu_0\) and \(\eta=y\nu_0\) for any nonzero basis
\(\nu_0\).}  As \(\hbar(\Omega_A)\subseteq A\), the extension
descends to a continuous map
\(\overline{\Omega}\to K_\infty/A\), again denoted \(\hbar\),
and we write \(\hbar_*g:=g\circ\hbar\) for precomposition.  The residue
pairing then gives the commutative diagram
\begin{equation}\label{eq:dualgraph}
\begin{tikzcd}
    \Cinf\widehat\otimes_{\Fq}\Omega_A \ar[d,"\hbar"] \ar[r, "\check{\mathscr{F}}"] &\operatorname{Hom}^{\cont}(K_\infty/ A , \Cinf ) \ar[d, "\hbar_*"]\\
    \Cinf\widehat\otimes_{\Fq}A = \TT_A \ar[r, "\mathscr{F}"] &\operatorname{Hom}^{\cont}( \overline{\Omega} , \Cinf )
\end{tikzcd}
\end{equation}
\revblue{Both composites are continuous and \(\Fq\)-linear in the upper-left
entry.  We therefore only need to compare them on a pure tensor
\(x\otimes\nu\).}  At \(\bar\xi=\xi+\Omega_A\),
\begin{align*}
 \mathscr F\bigl(x\otimes\hbar(\nu)\bigr)(\bar\xi)
 &=x\,\langle\hbar(\nu),\xi\rangle_\infty
 \overset{\eqref{eq:hg-contraction-symmetry}}{=}
 x\,\langle\hbar(\xi),\nu\rangle_\infty\\
 &=\bigl(\check{\mathscr F}(x\otimes\nu)\bigr)\bigl(\hbar(\xi)+A\bigr),
\end{align*}
which is \(\bigl(\hbar_*\check{\mathscr F}(x\otimes\nu)\bigr)(\bar\xi)\).
This is the form taken by the last assertion of
Lemma~\ref{lem:Tate-continuous-residue-duality} under the
pairing \eqref{eq:residue-dual-coordinate-line}.

The period lattice \(\Lambda_\psi\subset\Cinf\) is an
\(\iota(A)\)-submodule: exponential equivariance
\(\psi_a(\exp_\psi(z))=\exp_\psi(\iota(a)z)\) gives
\(\iota(A)\Lambda_\psi\subseteq\Lambda_\psi\).  For the rest of this
subsection, fix a period \(\lambda\in\Lambda_\psi\); no principality or
freeness assumption on \(\Lambda_\psi\) is made.  The letter \(\lambda\)
in \(F_\lambda\), \(W_\lambda\) and \(\jmath_\hbar\) below stands for
this fixed period.

Let \(F_{\lambda}\) denote the continuous map
\begin{equation}
\label{eq:hg-direct-special-function-map}
 F_{\lambda}:K_\infty/A\longrightarrow\Cinf,
 \qquad c+A\longmapsto\exp_\psi(\lambda\,\iota(c)).
\end{equation}
It is well defined: changing \(c\) by \(a\in A\) changes \(\lambda\,\iota(c)\)
by \(\lambda\iota(a)\in\iota(A)\lambda\subseteq\Lambda_\psi=\ker\exp_\psi\).  It is
continuous because \(\exp_\psi\) is.  
The following result gives the  {functional} interpretation of the  {universal} property of $ \mathcal{H}_{\varphi }$ (see  {Theorem~\ref{thm:identity-dual-factorization}}).
\begin{proposition}
\label{prop:dualelement}
Let \(\lambda\in\Lambda_\psi\). For
\(\hbar \in\operatorname{Hom}_A(\Omega_A,A)\), we denote
\[
\jmath_\hbar (\nu)= \lambda \cdot \iota(\hbar (\nu))
\qquad(\nu\in\Omega_A),
\]
so \(\jmath_\hbar\in\operatorname{Hom}_A(\Omega_A,\Lambda_\psi)\).
Then, under the diagram \eqref{eq:dualgraph},
\begin{equation}\label{eq:dualelement}
\begin{tikzcd}
    \mathcal{H}_{\psi}(\lambda) \ar[d,"\hbar"] \ar[r, "\check{\mathscr{F}}"] & F_{\lambda} \ar[d, "\hbar_*"]\\
    \omega_{\jmath_\hbar } \ar[r, "\mathscr{F}"] &\mathscr{F}_{\psi, \jmath_{\hbar}}
\end{tikzcd}
\end{equation}
where \(\mathscr{F}_{\psi, \jmath_{\hbar}}\) is defined as in
Definition~\ref{def:Fjmath}.   
\end{proposition}

\begin{proof}
  We put
\[
 W_{\lambda}
 :=\check{\mathscr F}^{-1}\bigl(F_{\lambda}\bigr)
 \in\Cinf\widehat\otimes_{\Fq}\Omega_A .
\]
  We need to check that:
  \begin{enumerate}
    \item[(i)] $\hbar_*  F_{\lambda} = \mathscr{F}_{\psi,\jmath_\hbar}$;
    \item[(ii)] $\omega_{\jmath_\hbar} = (1 \otimes \hbar) W_{\lambda}$;
    \item[(iii)] $W_{\lambda} = \mathcal{H}_\psi(\lambda)$.
\end{enumerate}

{By definition, $\check{\mathscr F}(W_\lambda)=F_\lambda$;
the bottom arrow of \eqref{eq:dualelement} is the definition}
\(\omega_{\jmath_\hbar}=\mathscr{F}^{-1}(\mathscr F_{\psi,\jmath_\hbar})\)
of the residue-dual exponential coordinate
\eqref{eq:residue-dual-exponential-map}.

For \(\bar\xi=\xi+\Omega_A\in\overline{\Omega}\), the
definition of \(\hbar_*\) gives
\[
 (\hbar_*F_{\lambda})(\bar\xi)
 =F_{\lambda}\bigl(\hbar(\xi)+A\bigr)
 =\exp_\psi\bigl(\lambda\,\iota(\hbar(\xi))\bigr)
 =\exp_\psi\bigl(\jmath_{\hbar,\infty}(\xi)\bigr)
 =\mathscr F_{\psi,\jmath_\hbar}(\bar\xi),
\]
\revblue{the final identification with $\mathscr F_{\psi,\jmath_\hbar}(\bar\xi)$ being Definition~\ref{def:Fjmath}; this proves (i).}
For \(\jmath_\hbar=0\) both sides vanish.

Applying \(\mathscr{F}^{-1} \) to (i) and using the
commutativity of \eqref{eq:dualgraph},
\begin{equation}
\label{eq:hg-jmath-principal-identification}
 \omega_{\jmath_\hbar}
 =\mathscr{F}^{-1}\bigl(\mathscr F_{\psi,\jmath_\hbar}\bigr)
 =\mathscr{F}^{-1}\bigl(\hbar_*F_{\lambda}\bigr)
 =(1\otimes\hbar)\check{\mathscr F}^{-1}
   \bigl(F_{\lambda}\bigr)
 =(1\otimes\hbar)W_{\lambda}.
\end{equation}
{This proves (ii).}
{To reconstruct $W_\lambda$, choose a finite dual frame
$(\nu_\ell,\hbar_\ell)$ of $\Omega_A$:}
\begin{equation}
\label{eq:hg-frame-two-identities}
 \sum_\ell\hbar_\ell(\nu)\nu_\ell=\nu
 \quad(\nu\in\Omega_A),
 \qquad
 \sum_\ell\hbar_\ell(\nu_\ell)=1 .
\end{equation}
The first identity is the frame property.  Since \(\Omega_A\) is
invertible, \(\operatorname{End}_A(\Omega_A)=A\); under this
identification the first identity says that the Casimir element
\(\sum_\ell\hbar_\ell\otimes\nu_\ell\) is the identity, and evaluating
it gives the second.  Contracting \(W_{\lambda}\) by each
\(\hbar_\ell\) and reconstructing with the frame, (ii) gives
\begin{equation}
\label{eq:hg-projective-frame-reconstruction}
 W_{\lambda}
 =\sum_\ell\bigl((1\otimes\hbar_\ell)W_{\lambda}\bigr)
   \otimes\nu_\ell
 =\sum_\ell\omega_{\jmath_{\hbar_\ell}}\otimes\nu_\ell .
\end{equation}
Each \(\omega_{\jmath_{\hbar_\ell}}\) is a scalar special function, so
\(W_{\lambda}\in\sfm(\psi)\otimes_A\Omega_A\).

We now evaluate the frame expansion.
By  {Theorem~\ref{thm:identity-dual-factorization}}
applied to each pair \((\nu_\ell,\jmath_{\hbar_\ell})\),
\[
 \mathcal H_\psi\bigl(\jmath_{\hbar_\ell}(\nu_\ell)\bigr)
 =\omega_{\jmath_{\hbar_\ell}}\otimes\nu_\ell .
\]
The period evaluation \(\mathcal H_\psi\) is \(\Fq\)-linear, hence
additive, and \eqref{eq:hg-frame-two-identities} gives
\[
 \sum_\ell\jmath_{\hbar_\ell}(\nu_\ell)
 =\lambda\,\iota\Bigl(\sum_\ell\hbar_\ell(\nu_\ell)\Bigr)
 =\lambda.
\]
Summing over \(\ell\) therefore gives
\[
 \mathcal H_\psi(\lambda)
 =\sum_\ell\mathcal H_\psi\bigl(\jmath_{\hbar_\ell}(\nu_\ell)\bigr)
 =\sum_\ell\omega_{\jmath_{\hbar_\ell}}\otimes\nu_\ell
 =W_{\lambda}.
\]
Equivalently, \(\mathcal H_\psi(\lambda)\) is the unique
element of \(\sfm(\psi)\otimes_A\Omega_A\) of characteristic residue
\(\lambda\), since \(\operatorname{res}_\Xi\) is the isomorphism
in \eqref{eq:hg-jmath-characteristic-residue}; as
\(W_\lambda=\mathcal H_\psi(\lambda)\),
Proposition~\ref{lem:residue} gives
\(\operatorname{res}_\Xi W_{\lambda}=\lambda\).  This proves (iii).
\end{proof}
 
\begin{corollary}\label{cor:hg-shtuka-trace-identification}
      The identity $\hbar(\mathcal{H}_\psi(\lambda))=\omega_{\jmath_\hbar}$ holds for any $\lambda\in\Lambda_\psi$ and any $ \hbar \in \operatorname{Hom}_A(\Omega_A,A) $.
\end{corollary}
\begin{proof}
The preceding proof identifies $\mathcal H_\psi(\lambda)$ with $W_\lambda=\check{\mathscr F}^{-1}F_\lambda$ and gives $\omega_{\jmath_\hbar}=(1\otimes\hbar)W_\lambda$; contracting the former identity by $1\otimes\hbar$ yields $\hbar(\mathcal H_\psi(\lambda))=(1\otimes\hbar)W_\lambda=\omega_{\jmath_\hbar}$.
\end{proof}

 



\section{Isogeny-adapted kernels}
\label{sec:onestep}
An isogeny transports the target's tautological Cauchy coordinate to a
kernel $K_J$ whose evaluation is the generating differential
$\mathcal G_J(z)$. The identity isogeny recovers $K_\varphi$ and
$\mathcal H_\varphi$; a general isogeny enlarges the homogeneous parameter
lattice from the source periods to the inverse image of the target periods.

{The defect equation also determines the scalar forcing factor
$\delta_J$ and the logarithm coefficients.  The associated
Papanikolas-type series telescopes after one Frobenius twist, and its
normalized form specializes at $\Xi$ to the Drinfeld logarithm.}

\subsection{Normalized shtuka differential}\label{sec:shtuka-differential}

Fix a rank-one Drinfeld module $\psi$ with geometric data
$(V_\psi,q_\psi,f_\psi)$, so that
\[
 h^0\bigl(X_{\Cinf},\OO(V_\psi)\bigr)=1,
 \qquad
 q_\psi^{(-1)}\notin\operatorname{supp}(V_\psi).
\]
Put
\begin{equation}\label{eq:psi-shifted-line}
 \LL_\psi:=\OO\bigl(V_\psi-[q_\psi^{(-1)}]\bigr).
\end{equation}

\begin{definition}
\label{def:shtuka-diff}
A rational differential $\widetilde\lambda_\psi$ is the
\emph{normalized shtuka differential} attached to the chosen geometric data
if
\begin{equation}
 \widetilde\lambda_\psi
 \in H^0\!\left(
 X_{\Cinf},
 \Omega_X^1\bigl(-V_\psi+[q_\psi^{(-2)}]+[q_\psi^{(-1)}]\bigr)
 \right),                                      \label{eq:shtuka-differential-divisor} 
\end{equation}
and 
\begin{equation}\label{eq:shtuka-differential-residue}
\operatorname{res}_{\Xi}
 \left(\frac{\widetilde\lambda_\psi^{(1)}}{f_\psi}\right)
 =1.           
\end{equation}                               
\end{definition}

\begin{theorem}[Szeg\H o realization of the shtuka differential]
\label{prop:exist-shtuka}
The line bundle $\LL_\psi$ is acyclic.  There is a unique normalized shtuka
differential, and it is characterized by
\begin{equation}\label{eq:szego-shtuka}
 \frac{\widetilde\lambda_\psi^{(1)}(\two)}
      {f_\psi(\two)}
 =K_\psi(\Xi,\two),
\end{equation}
where $K_\psi$ is the tautological coordinate of $S_{\LL_\psi}$ in
\eqref{eq:tautological-Szego-coordinate}. In particular, the
normalized Szeg\H o kernel supplies the differential directly.  Moreover,
\begin{equation}\label{eq:shtuka-differential-infinity-residue}
 \operatorname{res}_{q_\psi^{(-1)}}
 \left(\frac{\widetilde\lambda_\psi^{(1)}}{f_\psi}\right)
 =-1.
\end{equation}
\end{theorem}

\begin{proof}
Put $r=q_\psi^{(-1)}$.
Since $r\notin\operatorname{supp}(V_\psi)$, Lemma~\ref{lem:acyclic}
gives $H^0(\LL_\psi)=H^1(\LL_\psi)=0$.  For $\MM_\psi=\LL_\psi^{-1}\otimes\Omega_X^1$,
Serre duality gives $H^0(\MM_\psi)=H^1(\MM_\psi)=0$.
The tautological frame is regular and nonzero at $\Xi$, since
$\Xi\notin\operatorname{supp}(V_\psi)\cup\{r\}$.
The residue sequence is therefore
\begin{equation}\label{eq:shtuka-residue-exact-sequence}
 0\longrightarrow \MM_\psi\longrightarrow\MM_\psi([\Xi])
 \xrightarrow{\operatorname{res}_\Xi}\Cinf{}\mid_\Xi
 \longrightarrow0,
\end{equation}
The last term is supported at $\Xi$. Residue identifies the global
sections of $\MM_\psi([\Xi])$ with $\Cinf$.
Let $\gamma$ be its unique element of residue one, and define
$\widetilde\lambda_\psi^{(1)}=f_\psi\gamma$.
Coefficient Frobenius is bijective, so this defines
$\widetilde\lambda_\psi$ uniquely.

The divisor of $\gamma$ is bounded below by
$V_\psi-[r]-[\Xi]$.  Adding \eqref{eq:divf} gives
\[
 \operatorname{div}(f_\psi\gamma)
 \geqslant V_\psi^{(1)}-[r]-[q_\psi],
\]
the twist of \eqref{eq:shtuka-differential-divisor}.
The residue condition holds by construction.
Conversely, any differential in Definition~\ref{def:shtuka-diff},
after twisting and dividing by $f_\psi$, belongs to $H^0(\MM_\psi([\Xi]))$
and has residue one.  It must therefore give the same $\gamma$.

The specialization $K_\psi(\Xi,\cdot)$ has this same divisor bound
and residue one by the tautological-coordinate definition.  The
isomorphism furnished by \eqref{eq:shtuka-residue-exact-sequence}
thus identifies it with $\gamma$, proving \eqref{eq:szego-shtuka}.
The only possible poles of $\gamma$ are simple poles at $\Xi$ and $r$.
Its residue at $\Xi$ is one, so the global residue theorem gives
$\operatorname{res}_r\gamma=-1$.  This proves
\eqref{eq:shtuka-differential-infinity-residue} and the stated sign convention.
\end{proof}

\begin{corollary}\label{cor:shtuka-residue-uniqueness}
{The space of differentials
\begin{equation}\label{eq:curve-zeta-differential-line}
 \mathscr S_\psi:=
 H^0\!\left(X_{\Cinf},
      \Omega_X^1(-V_\psi+[\Xi]+[q_\psi^{(-1)}])\right)
\end{equation}
is one-dimensional, and
$\operatorname{res}_\Xi:\mathscr S_\psi\xrightarrow{\sim}\Cinf$ is an
isomorphism; in particular, two elements of $\mathscr S_\psi$ with the
same residue at $\Xi$ are equal.  Its element of residue one is
\[
 K_\psi(\Xi,\cdot)=\frac{\widetilde\lambda_\psi^{(1)}}{f_\psi}.
\]}
\end{corollary}
\begin{proof}
{The cohomology calculation
$H^0(\MM_\psi)=H^1(\MM_\psi)=0$ in the proof of
Theorem~\ref{prop:exist-shtuka} and the residue exact sequence
\eqref{eq:shtuka-residue-exact-sequence} give the isomorphism.
Membership of $\widetilde\lambda_\psi^{(1)}/f_\psi$ follows from
Definition~\ref{def:shtuka-diff} after one coefficient twist by
subtracting \eqref{eq:divf}; its residue is one by
\eqref{eq:shtuka-differential-residue}; and the equality with
$K_\psi(\Xi,\cdot)$ is \eqref{eq:szego-shtuka}.}
\end{proof}

\subsection{Defect equation}
Let $K(\one,\two)$ be a second-variable meromorphic differential.
Characteristic specialization below is applied to the product
$f_\varphi(\one)K^{(1)}$, after cancelling its divisor factors.

  For an isogeny datum $J=[\eta_J;h_J;\ii_J,\mu_J]:\varphi\to\psi$,
{Theorem~\ref{thm:isogeny-adapted-szego-kernel} shows
that the Frobenius defect of the tautological coordinate of the
isogeny-adapted line bundle factors through the single first-factor
module $h_J(\one)$.}  Motive evaluation then produces the
Anderson generating differential and identifies its homogeneous
parameter lattice.  The target-period lattice will also determine exactly where this
family agrees with the source's canonical period differential.

\begin{notation}
The Frobenius defect of the differential coordinate $K(\one,\two)$ is
    \[
   \mathscr{D}_{\varphi}(K) :=   f_\varphi(\one)K^{(1)}(\one,\two)
 -f_\varphi(\two)K(\one,\two).
    \]
\end{notation}

\begin{theorem}\label{thm:identity-kernel-defect}
{For the shifted shtuka line
$\LL_\psi=\OO\bigl(V_\psi-[q_\psi^{(-1)}]\bigr)$ and its
normalized kernel $K_\psi(\one,\two)$,
\[
 \mathscr D_\psi(K_\psi)
 :=f_\psi(\one)K_\psi(\one,\two)^{(1)}-f_\psi(\two)K_\psi(\one,\two)
 =-f_\psi(\two)K_{\psi}(\Xi,\two)
 =-\widetilde{\lambda}_{\psi}^{(1)}(\two).
\]}
\end{theorem}

\begin{proof}
{\revblue{Set $D=V_\psi-[q_\psi^{(-1)}]$.  The tautological divisor bound is
$\operatorname{div}_\one(K_\psi)\geqslant-D-[\two]$.  Twisting it gives
$\operatorname{div}_\one(K_\psi^{(1)})\geqslant-V_\psi^{(1)}+[q_\psi]-[\two]$.
Adding $\operatorname{div}(f_\psi)=V_\psi^{(1)}-V_\psi+[\Xi]-[q_\psi]$
gives lower bounds for both terms on the whole curve.  The first is bounded
below by $-V_\psi+[\Xi]-[\two]$, and the second by
$-V_\psi+[q_\psi^{(-1)}]-[\two]$.  The normalized diagonal residue is
twist-invariant.  Hence the two terms have the same diagonal residue, and their
diagonal poles cancel.}  Hence, with $F:=\Cinf(X_\two)$,
\[
 \mathscr D_\psi(K_\psi)
 \in H^0\bigl(X_F,\OO(V_\psi)_F\bigr)\otimes_F\Omega^1_{F/\Cinf}.
\]
Flat base change and $h^0\bigl(X_{\Cinf},\OO(V_\psi)\bigr)=1$ make
$\mathscr D_\psi(K_\psi)$ independent of $\one$.  {The first
divisor bound and $\Xi\notin\operatorname{supp}(V_\psi)$ show that
the product $f_\psi(\one)K_\psi^{(1)}(\one,\two)$ vanishes at
$\one=\Xi$.  Evaluation there therefore gives}
$\mathscr D_\psi(K_\psi)=-f_\psi(\two)K_\psi(\Xi,\two)$, and
\eqref{eq:szego-shtuka} identifies this with
$-\widetilde\lambda_\psi^{(1)}(\two)$.}
\end{proof}

\begin{notation}\label{not:KJ}
Let $(V_\psi,q_\psi,f_\psi)$ be the fixed geometric realization of the
target module and put
\begin{equation}\label{eq:isogeny-adapted-line}
 D_J:=V_\psi-\operatorname{div}(h_J)-[q_\psi^{(-1)}],\qquad
 \mathcal{L}_J :=\OO(D_J).
\end{equation}
Notice that multiplication by $h_J$ identifies
$\OO(V_\psi-[q_\psi^{(-1)}])$ with $ {\mathcal{L}_J}$ since 
\[ 
\operatorname{div}(h_Jg)+D_J
=\operatorname{div}(g)+V_\psi-[q_\psi^{(-1)}].
\]
Thus $ {\mathcal{L}_J}$ is acyclic. Let $K_J$ be the tautological
second-variable differential coordinate of $S_{ {\mathcal{L}_J}}$, as in
\eqref{eq:tautological-Szego-coordinate}.

\end{notation}

The relation 
$ {\mathcal{L}_J}=\LL_\psi\otimes\OO(-\operatorname{div}(h_J))$ gives the exact
identity
\begin{equation}\label{eq:isogeny-kernel-coordinate-change}
 K_J(\one,\two)
 =\frac{h_J(\one)}{h_J(\two)}K_\psi(\one,\two).
\end{equation}

{The intertwiner relation transports the identity-kernel defect
to $K_J$.}
\begin{theorem}[Isogeny-adapted Szeg\H o kernel]
\label{thm:isogeny-adapted-szego-kernel}
The kernel \(K_J\) in Notation~\ref{not:KJ} {satisfies}
\begin{equation}
\label{eq:LJ-defect}
 \mathscr D_\varphi(K_J)
 =-h_J(\one)\tilde{\delta}_J(\two),
\end{equation}
and
\begin{equation}\label{eq:Dvarphi}
 \tilde{\delta}_J(\two)=
 \frac{f_\varphi(\two)K_J(\Xi,\two)}{h_J(\Xi)}
 =\frac{\widetilde\lambda_\psi^{(1)}(\two)}
       {h_J^{(1)}(\two)}.
\end{equation}
\end{theorem}

\begin{proof}
{Transport reduces the statement to the identity case.  By
\eqref{eq:isogeny-kernel-coordinate-change} and
\eqref{eq:intertwiner-difference}, it follows that 
\[
 f_\varphi(\one)K_J^{(1)}
 =\frac{h_J(\one)}{h_J^{(1)}(\one)}\,f_\psi(\one)\,
   \frac{h_J^{(1)}(\one)}{h_J^{(1)}(\two)}\,K_\psi^{(1)}
 =\frac{h_J(\one)}{h_J^{(1)}(\two)}\,f_\psi(\one)K_\psi^{(1)},\]
 and
 \[
 f_\varphi(\two)K_J
 =\frac{h_J(\one)}{h_J^{(1)}(\two)}\,f_\psi(\two)K_\psi.
\]
Subtracting and applying Theorem~\ref{thm:identity-kernel-defect} give
\[
 \mathscr D_\varphi(K_J)
 =\frac{h_J(\one)}{h_J^{(1)}(\two)}\,\mathscr D_\psi(K_\psi)
 =-\frac{h_J(\one)}{h_J^{(1)}(\two)}\,\widetilde\lambda_\psi^{(1)}(\two)
 =-h_J(\one)\,\tilde\delta_J(\two),
\]
which is \eqref{eq:LJ-defect} with the second equality of
\eqref{eq:Dvarphi}.  \revblue{For the remaining equality, specialize the transported identity
$f_\psi(\one)K_\psi^{(1)}
=f_\psi(\two)K_\psi-\widetilde\lambda_\psi^{(1)}(\two)$ at $\one=\Xi$.
The first-variable divisor bound in the proof of
Theorem~\ref{thm:identity-kernel-defect} shows that this specialization is
defined.  It vanishes by \eqref{eq:szego-shtuka}.}  Since $h_J$ is regular at $\Xi$,
the transport formula above gives
$\bigl(f_\varphi(\one)K_J^{(1)}\bigr)\big|_{\one=\Xi}=0$, and
evaluating \eqref{eq:LJ-defect} at $\one=\Xi$ yields
$\tilde\delta_J(\two)=f_\varphi(\two)K_J(\Xi,\two)/h_J(\Xi)$, the
first equality of \eqref{eq:Dvarphi}.}
\end{proof}

\subsection{\texorpdfstring{{Divisor of \(\delta_J\) for arbitrary \(\Ninf\)}}{Divisor of delta-J for arbitrary N}}
\label{sec:divisor-DJ}

Let $\Ninf\geqslant1$, and let $J:\varphi\to\psi$ have free target
lattice $\Lambda_\psi=\iota(A) \tilde{\pi} $.  Suppose
$\Omega_A=\mathfrak i_J\nu_J$ with $\nu_J\in{\Omega_A^{\rat}}$.
The scalar forcing factor is
\[
 \delta_J:=\frac{\widetilde\lambda_\psi^{(1)}}
                    {h_J^{(1)}\nu_J}.
\]
It is distinct from the differential forcing term
$\widetilde\lambda_\psi^{(1)}/h_J^{(1)}$.

By \eqref{eq:shtuka-differential-divisor} there is a unique effective
divisor $Z_\psi$ of degree $g$ such that
\begin{equation}\label{eq:div-lambda}
 \operatorname{div}(\widetilde\lambda_\psi^{(1)})
 =V_\psi^{(1)}+Z_\psi-[q_\psi]-[q_\psi^{(-1)}].
\end{equation}
Indeed, subtracting the lower bound leaves an effective divisor of
degree $(2g-2)-g+2=g$.

\begin{lemma}\label{lem:denom}
Let ${D_{\ii_J}}$ be the pullback of $\operatorname{div}_U(\mathfrak i_J)$
to $U_{\Cinf}$.  Then
\[
 \operatorname{div}(h_J^{(1)})|_U
 ={D_{\ii_J}}+V_\psi^{(1)}|_U-V_\varphi^{(1)}|_U,
 \qquad
 \operatorname{div}(\nu_J)|_U=-{D_{\ii_J}}.
\]
\end{lemma}
\begin{proof}
The first identity is the twist of
\eqref{eq:isogeny-intertwiner-affine-divisor-main}; the ideal divisor
${D_{\ii_J}}$ is defined over $\Fq$ and hence fixed by twisting.
At each affine point the equality
$\Omega_A=\mathfrak i_J\nu_J$ says that multiplication by a local
generator of $\mathfrak i_J$ makes $\nu_J$ a basis of the regular
differential module.  The resulting orders give the second identity.
\end{proof}

\begin{proposition}\label{prop:div-DJ}
  The function $\delta_J$ is regular on $U_{\Cinf}$, with
\[
 \operatorname{div}(\delta_J)|_U
       =(V_\varphi^{(1)}+Z_\psi)|_U.
\]
\end{proposition}
\begin{proof}
 Subtract the divisors of $h_J^{(1)}$ and $\nu_J$ from
\eqref{eq:div-lambda}.  Lemma~\ref{lem:denom} cancels the ideal divisor
and the target Drinfeld divisor on $U$, leaving the displayed effective
divisor.

\end{proof}

\begin{corollary}\label{cor:N=1}
If $\Ninf=1$, then
\[
 \operatorname{div}(\delta_J)
 =V_\varphi^{(1)}+Z_\psi-2g[\infty].
\]
\end{corollary}
\begin{proof}
Here $q_\psi^{(-1)}=q_\psi=\infty$, and the realization excludes
$\infty$ from both Drinfeld divisors.
The differential $\widetilde\lambda_\psi^{(1)}/f_\psi$ has a simple
pole of residue $-1$ at $\infty$, while
$\operatorname{ord}_\infty(f_\psi)=-1$.
Thus $\widetilde\lambda_\psi^{(1)}$ has order $-2$ there and
$Z_\psi$ has no infinity component.
The affine divisor has degree $2g$, and the only remaining point is
$\infty$.  Degree zero of the principal divisor proves the formula.
\end{proof}

{When $\Ninf=1$, Corollary~\ref{cor:N=1}, twisted back by
$(-1)$, gives
$\operatorname{div}(\delta_J^{(-1)})=V_\varphi+Z_\psi^{(-1)}-2g[\infty]$:}
the
function $\delta_J^{(-1)}$ exhibits $V_\varphi$ and $Z_\psi^{(-1)}$ as dual
divisors, in the sense of \cite[Lemma~6.4]{Pel26}, where each shtuka divisor
$V$ has a unique dual $V^*$ and a rational function, unique up to a scalar,
with $\operatorname{Div}(\delta)=V+V^*-2g\infty$.  Comparing with
\cite[Theorem~5.3]{Pel26}, the algebraic factor $\delta$ of the zeta identity
$\omega^{(1)}Z_{X,A}=\tilde\pi\,\delta^{(1)}$ plays the same dual role.

\revsubsection{The $J$-adapted generating functions}{The J-adapted generating functions}

\begin{lemma}\label{lem:Admissibility}
Let $W \subset U_{ \Cinf}^{\mathrm{an}}$ be an affinoid disjoint from
the characteristic orbit and from the second-variable poles of $K_J$ and
$f_\varphi(\two)K_J$.
Then $K_J$ and $f_{\varphi}(\one)K_J^{(1)}$ are weighted-orbit-admissible
on $W$ in the sense of Definition~\ref{def:orbital-domain};
in particular, for every $u \in  \Cinf$ the orbital evaluations
\[
\mathcal E^{\Xi}_{\varphi,u}(K_J s_\varphi)
\quad\text{and}\quad
\mathcal E^{\Xi}_{\varphi,u}(f_{\varphi}(\one) K_J^{(1)} s_\varphi)
\]
exist and converge normally on $W$.
\end{lemma}

\begin{proof}
{The first-variable pole bounds follow directly from
transport: by \eqref{eq:isogeny-kernel-coordinate-change} and
\eqref{eq:intertwiner-difference},
\[
 K_J=\frac{h_J(\one)}{h_J(\two)}K_\psi,
 \qquad
 f_\varphi(\one)K_J^{(1)}
 =\frac{h_J(\one)}{h_J^{(1)}(\two)}\,f_\psi(\one)K_\psi^{(1)};
\]
the divisor of the intertwiner satisfies\revblue{:}
\[
 \operatorname{div}(h_J)-V_\psi+V_\varphi\geqslant0
 \qquad\text{on }U_{\Cinf}
\]
\revblue{The first bound uses integrality of $h_Js_\varphi$.  For the other
two, the proof of Theorem~\ref{thm:identity-kernel-defect} gives
$\operatorname{div}_\one(K_\psi)\geqslant-V_\psi-[\two]$ and
$\operatorname{div}_\one\bigl(f_\psi(\one)K_\psi^{(1)}\bigr)
\geqslant-V_\psi-[\two]$ on $U_{\Cinf}$.}  Hence, on $U_{\Cinf}$,
\begin{equation}\label{eq:KJ-pole-bounds}
 \operatorname{div}_{\one}(K_J)\geqslant-V_\varphi-[\two],
 \qquad
 \operatorname{div}_{\one}\!\bigl(f_{\varphi}(\one)K_J^{(1)}\bigr)
 \geqslant-V_\varphi-[\two].
\end{equation}
Both are therefore fixed meromorphic sections} whose non-diagonal first-variable
poles are bounded by $V_\varphi$, and both are regular on
$\{\Xi^{(n)} : n \notin E\} \times W$, where
\[ 
E := \{\, n \geqslant 0 : \Xi^{(n)} \in \operatorname{supp}(V_\varphi) \,\}\]
is finite.
\revblue{Parts (ii) and (iii) of Proposition~\ref{thm:automatic-orbit-growth}
apply verbatim.  They give:}
\[
\operatorname{orb}_{\varphi,n}(g)\text{ is defined for }g=K_J\text{ and }g=f_{\varphi}(\one)K_J^{(1)}
\]
\revblue{The exceptional poles are regularized by
$(F_\varphi^{[n]})^{-1}$, using
Lemma~\ref{lem:weighted-orbit-divisor-cancellation}.  The same proposition also
gives the growth bound:}
\[
\bigl\| \operatorname{orb}_{\varphi,n}(g) \bigr\|_W
\leqslant |e_n(\varphi)|\, C_W^{q^n}
\qquad (n \notin E).
\]
By Lemma~\ref{lem:orbital-evaluation-convergence}, the orbital
series converge normally on $W$.
\end{proof}

\begin{definition}
\label{def:isogeny-generating-differential}
For the kernel $ K_J $ in 
Theorem~\ref{thm:isogeny-adapted-szego-kernel},
we define 
\begin{equation}\label{eq:GJ-generating-differential}
 \mathcal{G}_{J}(u)
 :=\mathcal E^\Xi_{\varphi,u}\bigl(K_J(\one,\two)s_\varphi\bigr).
\end{equation}
\end{definition}
The series converges by Lemma~\ref{lem:orbital-evaluation-convergence},
since $K_J$ is weighted-orbit-admissible by Lemma~\ref{lem:Admissibility}.
\begin{proposition}\label{prop:GJ-Hpsi}
For every $z\in\Cinf$,
\begin{enumerate}[label=\textup{(\roman*)}]
\item\label{item:GJ-target}
\begin{equation}\label{eq:J-kernel-target-comparison}
 \mathcal G_J(z)=h_J(\two)^{-1}\mathcal H_\psi(h_J(\Xi)z);
\end{equation}
\item\label{item:GJ-residue}
$\mathcal G_J(z)$ has at most a simple pole at $\Xi$ and
\begin{equation}\label{eq:HJ-characteristic-residue}
 \operatorname{res}_\Xi\mathcal G_J(z)=z.
\end{equation}
\end{enumerate}
For $J=\mathrm{id}_\varphi$, the family is $\mathcal H_\varphi$.
\end{proposition}
\begin{proof}
For \ref{item:GJ-target}, iterating the intertwiner identity gives
$h_J^{(n)}F_\varphi^{[n]}=h_JF_\psi^{[n]}$.  Hence, as rational functions,
\[
 \frac{h_J(\one)K_\psi(\one,\two)}{F_\varphi^{[n]}(\one)}
 =h_J^{(n)}(\one)\frac{K_\psi(\one,\two)}{F_\psi^{[n]}(\one)}.
\]
The quotient on the right is regularized for the target realization,
and $h_J^{(n)}(\Xi^{(n)})=h_J(\Xi)^{q^n}$.
Thus the $n$-th weighted specialization equals
$h_J(\Xi)^{q^n}\operatorname{orb}_{\psi,n}(K_\psi)$, including
the exceptional divisor indices.  Multiply by $z^{q^n}$,
sum the normally convergent series, and use
\eqref{eq:isogeny-kernel-coordinate-change} to obtain
\eqref{eq:J-kernel-target-comparison}.
For \ref{item:GJ-residue}, Proposition~\ref{lem:residue} gives residue
$h_J(\Xi)z$ for the numerator.  The denominator is regular and nonzero at
$\Xi$, with value $h_J(\Xi)$, so the residue of the quotient is $z$.
The identity case has $h_J=1$.
\end{proof}

\begin{theorem}\label{thm:frob}
For every $z\in\Cinf$, one has the meromorphic differential identity
\begin{equation}\label{eq:6.13}
 \mathcal G_J(z)^{(1)}-f_\varphi\mathcal G_J(z)
 =-\widetilde\delta_J\exp_\psi(h_J(\Xi)z),
 \qquad
 \widetilde\delta_J=\frac{\widetilde\lambda_\psi^{(1)}}{h_J^{(1)}}.
\end{equation}
\end{theorem}
\begin{proof}
Choose affinoids $W_0$, $W_1$, and $V\subset W_0\cap W_1$ as in
Theorem~\ref{thm:coherence}.  Lemma~\ref{lem:Admissibility} gives normal
convergence in the second-variable differential module.
{The proof of
Theorem~\ref{thm:isogeny-adapted-szego-kernel} gives
$\bigl(f_\varphi(\one)K_J^{(1)}\bigr)\big|_{\one=\Xi}=0$.
Consequently the initial coefficient $\beta_\varphi(K_J)$ is zero,
even when the separate factor $K_J^{(1)}$ has a pole there.}
Theorem~\ref{thm:coherence} therefore gives
\[
 \mathcal G_J(z)^{(1)}-f_\varphi(\two)\mathcal G_J(z)
 =\mathcal E^\Xi_{\varphi,z}
       \bigl(\mathscr D_\varphi(K_J)s_\varphi\bigr).
\]
By \eqref{eq:LJ-defect}, the expression on the right is
$-\widetilde\delta_J(\two)\mathcal E^\Xi_{\varphi,z}(h_J(\one)s_\varphi)$.
The integral motive element $h_Js_\varphi$ represents $\eta_J$.
The weighted Pellarin identity and exponential compatibility give
\[
 \mathcal E^\Xi_{\varphi,z}(h_Js_\varphi)
 =\eta_J(\exp_\varphi z)=\exp_\psi(h_J(\Xi)z).
\]
This proves the equation on each common affinoid, and these identities
glue as meromorphic differentials.
\end{proof}

\begin{definition}
We define the special lattice of $ \mathcal{G}_J $:
\begin{equation}\label{eq:H-reconstructs-target-parameter-lattice}
 \Lambda_{\mathcal{G}_J}  =\{u\in\Cinf:\ \mathcal{G}_{J}(u)^{(1)}=f_\varphi \mathcal{G}_{J}(u)\}.
\end{equation}
 \end{definition}

\begin{corollary}
\label{cor:H-source-lattice-recovery}
 Let $\ii_J$ be the ideal of the isogeny datum $J$.  Then
\begin{equation}\label{eq:H-parameter-lattice-isogeny-correction}
 \Lambda_{\mathcal{G}_J}
   = h_J(\Xi)^{-1} \Lambda_\psi =\iota(\ii_J^{-1})\Lambda_\varphi
\end{equation}
In particular, if $\Lambda_\psi=\iota(A)\tilde{\pi}$ and $z_J:=h_J(\Xi)^{-1} \tilde{\pi}$,
then
\begin{equation}
\label{eq:H-primitive-source-lattice-recovery}
 \Lambda_{\mathcal{G}_J}=Az_J,
 \qquad
 \Lambda_\varphi=\iota(\ii_J)z_J.
\end{equation}
For the identity isogeny, using the shifted predecessor kernel,
$\ii_J=A$ and the homogeneity locus is $\Lambda_\varphi$ itself.
\end{corollary}

\begin{proof}
Since $\tilde{\delta}_J$ is a nonzero meromorphic differential, one has the
exact equivalence
\begin{enumerate}
    \item $\mathcal{G}_{J}(u)^{(1)}=f_\varphi \mathcal{G}_{J}(u)$;
    \item $\exp_\psi(h_J(\Xi)u)=0$;
    \item $u\in h_J(\Xi)^{-1}\Lambda_\psi$.
\end{enumerate}

The ideal-isogeny lattice formula
\eqref{eq:isogeny-lattice-transport-main} gives
\eqref{eq:H-parameter-lattice-isogeny-correction}.

In particular, when the
target lattice is free, its pullback by $h_J(\Xi)$ is generated by $z_J$,
which proves \eqref{eq:H-primitive-source-lattice-recovery} and the
identity-isogeny specialization.
\end{proof}



\begin{theorem}
\label{thm:localized-J-kernel-identification}
Let $J=[\eta_J;h_J;\ii_J,\mu_J]$ be any separable isogeny datum and let
$0\ne z\in h_J(\Xi)^{-1}\Lambda_\psi$.  The kernel value is a rational Anderson differential
\begin{equation}\label{eq:localized-J-kernel-membership}
 \mathcal{G}_{J}(z)\in
 \sfm(\varphi)\otimes_A\Omega_A^{\rat}.
\end{equation}
 Then $\mathcal{G}_{J}(z)$ is the unique element in $\sfm(\varphi)\otimes_A\Omega_A^{\rat}$ whose localized residue is
$ z $.
\end{theorem}

\begin{proof}
Let \(z\in h_J(\Xi)^{-1}\Lambda_\psi\). Since \(h_J(\Xi)z\in\Lambda_\psi=\ker(\exp_\psi)\), Theorem~\ref{thm:frob} gives
\[
\mathcal G_J(z)^{(1)}-f_\varphi \mathcal G_J(z)
=-\tilde{\delta}_J\exp_\psi(h_J(\Xi)z)=0.
\]
Hence
\[
\mathcal G_J(z)^{(1)}=f_\varphi \mathcal G_J(z),
\]
while \eqref{eq:J-kernel-target-comparison} places it in
$\operatorname{Frac}(\mathbb T_A)\otimes_A\Omega_A$:
the numerator is a target-period differential and $h_J$ is rational.
It is therefore a rational Anderson differential.
By part~\ref{item:ratand-form} of
Theorem~\ref{thm:rational-anderson-continuation},
\[
\mathcal G_J(z)\in \operatorname{sf}(\varphi)\otimes_A\Omega_A^{\rat}.
\]

Next, Proposition~\ref{prop:GJ-Hpsi} gives
\[
\mathcal G_J(z)=h_J(\two)^{-1}\mathcal H_\psi(h_J(\Xi)z).
\]
Since \(h_J(\Xi)z\in\Lambda_\psi\), Proposition~\ref{lem:residue} yields
\[
\operatorname{res}_{\Xi}\mathcal H_\psi(h_J(\Xi)z)=h_J(\Xi)z.
\]
We obtain
\[
\operatorname{res}_{\Xi}\mathcal G_J(z)
=h_J(\Xi)^{-1}\operatorname{res}_{\Xi}\mathcal H_\psi(h_J(\Xi)z)
=h_J(\Xi)^{-1}(h_J(\Xi)z)=z.
\]

Finally, by
part~\ref{item:ratand-residue} of
Theorem~\ref{thm:rational-anderson-continuation}, the residue map
\[
\operatorname{res}_{\Xi}:\operatorname{sf}(\varphi)\otimes_A\Omega_A^{\rat}
\longrightarrow K\otimes_A\Lambda_\varphi
\]
is an isomorphism. Thus any element of \(\operatorname{sf}(\varphi)\otimes_A\Omega_A^{\rat}\) with residue \(z\) must equal \(\mathcal G_J(z)\). Therefore \(\mathcal G_J(z)\) is the unique such element.
\end{proof}

\subsection{\texorpdfstring{Universal property of \(\mathcal{G}_J\)-functions}{Universal property of G-J-functions}}
\label{subsec:primitive-G}

The primitive \(\mathcal{G}_J\)-identity extends the factorization
$\mathcal H_\varphi(\jmath(\eta))=\omega_\jmath\otimes\eta$
from $\Omega_A$ to $A\nu_J=\ii_J^{-1}\Omega_A$ under the isogeny
data below. \revblue{Throughout, special functions are viewed in \(\mathbb T_A\).  This
realizes \(\sfm(\varphi)\otimes_A\Omega_A^{\rat}\) inside
\(\operatorname{Frac}(\mathbb T_A)\otimes_A\Omega_A\).
Differential-valued identities are read on affinoids $V\subset W_0\cap W_1$,
as in Theorem~\ref{thm:coherence}, so the indicated second-variable poles are
avoided.}

\begin{notation}
\label{not:primitive-G}
Let \(J=[\eta_J;h_J;\ii_J,\mu_J]:\varphi\to\psi\) be a separable
isogeny datum with free target lattice
{\(\Lambda_\psi=\iota(A)\tilde\pi\)} and
\(\ii_J=\operatorname{Ann}_A(\ker\eta_J)\); \(h_J(\Xi)\ne0\).
Assume the differential-ideal match of the anchor-map subsection,
\begin{equation}
\label{eq:primitive-G-diff-match}
\Omega_A=\ii_J\nu_J
\end{equation}
for some rational differential \(\nu_J\in \Omega_A^{\rat}\). Define
\begin{equation}
\label{eq:primitive-G-zJ}
{z_J:=\frac{\tilde\pi}{h_J(\Xi)}.}
\end{equation}
Then, by \eqref{eq:H-primitive-source-lattice-recovery},
\begin{equation}
\label{eq:primitive-G-lattices}
h_J(\Xi)^{-1}\Lambda_\psi=\iota(A)z_J,
\qquad
\Lambda_\varphi=\iota(\ii_J)z_J .
\end{equation}
Let
\[
\jmath:\; A\nu_J=\ii_J^{-1}\Omega_A \longrightarrow \iota(A)z_J=h_J(\Xi)^{-1}\Lambda_\psi
\]
be an arbitrary \(A\)-linear map. Its restriction to
\(\Omega_A=\ii_J\nu_J\) lands in \(\Lambda_\varphi=\iota(\ii_J)z_J\),
hence defines an element
\[
\jmath|_{\Omega_A}\in \operatorname{Hom}_A(\Omega_A,\Lambda_\varphi).
\]
By Theorem~\ref{thm:direct-logarithmic-special-line} there is a unique
special function \(\omega_\jmath\in\sfm(\varphi)\) such that
\[
\mathscr L_\varphi(\omega_\jmath)=\jmath|_{\Omega_A}.
\]
When \(\jmath\) is an isomorphism mapping $\nu_J$ to $z_J$, write $\jmath_*$ for it.
Its special function generates $\sfm(\varphi)$ by
the basis criterion of Theorem~\ref{thm:direct-logarithmic-special-line};
Section~\ref{sec:unit-special} will prove it is a Tate unit.
\end{notation}

\begin{theorem}
\label{thm:primitive-G}
Retain Notation~\ref{not:primitive-G}. For every
\[
\eta\in A\nu_J=\ii_J^{-1}\Omega_A,
\]
the following hold.
\begin{enumerate}[label=\textup{(\roman*)},leftmargin=2.2em]
\item\label{item:primitive-G-factorization}
\begin{equation}
\label{eq:primitive-G-identity}
\mathcal{G}_{J}\bigl(\jmath(\eta)\bigr)
=
\omega_\jmath\otimes\eta
\qquad\text{in }
\sfm(\varphi)\otimes_A\Omega_A^{\rat},
\end{equation}
both sides having characteristic residue
\[
\jmath(\eta)\in h_J(\Xi)^{-1}\Lambda_\psi.
\]
\item\label{item:primitive-G-integrality} If \(\eta\in\Omega_A\), then
\(\jmath(\eta)\in\Lambda_\varphi\), and both sides of
\eqref{eq:primitive-G-identity} lie in the integral Anderson differential
$\sfm(\varphi)\otimes_A\Omega_A$.
\end{enumerate}
\end{theorem}

\begin{proof}
\emph{\ref{item:primitive-G-factorization}.}
Extend the restricted anchor to $\jmath_K:{\Omega_A^{\rat}}\to\iota(K)\Lambda_\varphi$.
On $A\nu_J$ this agrees with the original $\jmath$.
The localized residue is $K$-linear, so
\[
 \operatorname{res}_\Xi(\omega_\jmath\otimes\eta)
 =\jmath_K(\eta)=\jmath(\eta).
\]
For example, this follows by choosing $0\ne b\in A$ with
$b\eta\in\Omega_A$, applying
\eqref{eq:characteristic-residue-anchor} to $b\eta$,
and dividing by $\iota(b)$.
The value $\jmath(\eta)$ lies in $h_J(\Xi)^{-1}\Lambda_\psi$.
If it is nonzero, Theorem~\ref{thm:localized-J-kernel-identification}
identifies $\mathcal G_J(\jmath(\eta))$ as the unique homogeneous
differential with this residue.  If it is zero, either $\eta=0$
or $\jmath=0$, and both sides vanish.
This proves \eqref{eq:primitive-G-identity}.
\emph{\ref{item:primitive-G-integrality}.}
For $\eta\in\Omega_A$, its image is a source period and the tensor
$\omega_\jmath\otimes\eta$ is integral, proving the last assertion.
\end{proof}
\begin{corollary}
\label{cor:isogeny-onestep-period-factorization}
Suppose $\jmath_K(\nu_J)=z_J$ and the restriction
$\jmath:\Omega_A=\ii_J\nu_J\to\Lambda_\varphi$ is an isomorphism.
\begin{enumerate}[label=\textup{(\roman*)},leftmargin=2.2em]
\item\label{item:primitive-coordinate}
For every $a\in A$,
$\mathcal G_J(\iota(a)z_J)=a\omega_\jmath\otimes\nu_J$.
The residue map identifies
$\nu_J\operatorname{sf}(\varphi)$ with $\iota(A)z_J$, with inverse
$\mathcal G_J$.
\item\label{item:primitive-exact-integrality}
This value belongs to $\operatorname{sf}(\varphi)\otimes_A\Omega_A$
if and only if $a\in\ii_J$.
\end{enumerate}
\end{corollary}
\begin{proof}
\emph{\ref{item:primitive-coordinate}.}
The primitive identity gives the formula, and
Theorem~\ref{thm:direct-logarithmic-special-line} gives
$\operatorname{sf}(\varphi)=A\omega_\jmath$.
Residue sends its basis $\nu_J\omega_\jmath$ to $z_J\ne0$, proving the
isomorphism and the inverse formula.

\emph{\ref{item:primitive-exact-integrality}.}
Since $\omega_\jmath$ is an $A$-basis, \revblue{integrality means that}
$a\nu_J\in\Omega_A=\ii_J\nu_J$\revblue{, or equivalently that} $a\in\ii_J$.
\end{proof}

\begin{definition}\label{def:normalized-period}
{A unit anchor $\jmath:\Omega_A\xrightarrow{\sim}\Lambda_\varphi$,
together with a differential $\nu_J$ satisfying
$\Omega_A=\mathfrak i_J\nu_J$, normalizes the fundamental period
$\tilde\pi:=h_J(\Xi)\,\jmath_K(\nu_J)$, where $\jmath_K$ denotes the
$K$-linear extension of $\jmath$.  It is uniquely determined by the
pair $(\jmath,\nu_J)$.}
\end{definition}

\begin{corollary}
\label{cor:primitive-G-GH}
Retain Notation~\ref{not:primitive-G}. For every \(\eta\in\Omega_A\),
\[
\mathcal{G}_{J}(\jmath(\eta))
=\mathcal{H}_{\varphi}(\jmath(\eta))
=\omega_\jmath\otimes\eta
\in\sfm(\varphi)\otimes_A\Omega_A.
\]
Equivalently, if \(\jmath=\jmath_*\) is an isomorphism, then
\(\mathcal G_{J}(\lambda)=\mathcal{H}_\varphi(\lambda)\) for every
\(\lambda\in\Lambda_\varphi\). In particular, the scalar factorization
\eqref{eq:canonical-family-scalar-factorization} holds verbatim with
\(\mathcal{H}_{\varphi}(\jmath_*(\cdot))\) replaced by
\(\mathcal G_{J}(\jmath_*(\cdot))\), as does the $A$-linearity of
Corollary~\ref{cor:Hphi-linearity}.
\end{corollary}

\begin{proof}
For \(\eta\in\Omega_A\) we have
\(\jmath(\eta)\in\Lambda_\varphi\subseteq h_J(\Xi)^{-1}\Lambda_\psi\) by
\eqref{eq:primitive-G-lattices}, so Theorem~\ref{thm:primitive-G} gives
\(\mathcal{G}_{J}(\jmath(\eta))=\omega_\jmath\otimes\eta\), which equals
\(\mathcal{H}_{\varphi}(\jmath(\eta))\) by
\eqref{eq:canonical-family-scalar-factorization}. The \(\lambda\)-form uses
that \(\jmath_*\) is an isomorphism. The balanced identity for the
\(G\)-family reads
\(\mathcal{G}_{J}(\iota(a)z_J)=a\,\nu_J\,\omega_{\jmath_*}\) for
\(0\ne a\in A\), which follows by taking $\eta=a\nu_J$ in
\eqref{eq:primitive-G-identity}.
\end{proof}

\begin{example}[The HHY datum]
\label{ex:primitive-G-HHY}
Take \(\varphi=\Psi\), \(\psi=\Phi\) and the isogeny datum
\(J=[\eta_J,h_J;\ii_\infty^2,1]:\Psi\to\Phi\) of the anchor-map
subsection, with \(h_J(\theta)=(\theta-\eta)^{-2}\), so
\(h_J(\Xi)=(\theta-\eta)^{-2}\) and \(\Omega_A=\ii_\infty^2\,\rdt\); i.e.\
\eqref{eq:primitive-G-diff-match} holds with \(\nu_J=\rdt\), and
\(z_J=(\theta-\eta)^2\tilde{\pi}_\Phi\) by \eqref{eq:primitive-G-zJ}. Then
\(A\nu_J=A\,\rdt\). Let \(\jmath: A\,\rdt \to \iota(A)z_J\) be the \(A\)-linear
map determined by \(\jmath(\rdt)=z_J\). Then Theorem~\ref{thm:primitive-G}
gives
\[
\mathcal{G}_{J}(\iota(z)z_J)
=\omega_\jmath\otimes z\,\rdt
=z(t)\,\omega_\jmath\,\rdt
\qquad(z\in A),
\]
and in particular the generator identity
\(\mathcal G_{J}(z_J)=\omega_\jmath\otimes\rdt\), whose characteristic
residue is \(\operatorname{res}_\Xi\bigl(\rdt\,\omega_\jmath\bigr)=z_J\).
For \(a\in\ii_\infty^2\) one has
\(\jmath(a\,\rdt)=\iota(a)z_J\in\Lambda_\Psi=\iota(\ii_\infty^2)z_J\) by
\eqref{eq:primitive-G-lattices}, and
\[
\mathcal{G}_{J}(\iota(a)z_J)
=\mathcal{H}_\Psi\bigl(\iota(a)z_J\bigr)
=a\,\rdt\,\omega_\jmath
\]
by Corollary~\ref{cor:primitive-G-GH}, whereas for
\(a\notin\ii_\infty^2\) (for instance \(a=1\), i.e.\ \(u=z_J\),
which is not a period) the value \(\mathcal{H}_\Psi(z_J)\) is
inhomogeneous, while the \(G\)-identity persists: this realizes the
maximal extension.
\end{example}

\subsection{Residue logarithms and characteristic specialization}
\label{sec:residue-log-formula}

Choose a nonzero differential $\nu\in\Omega_A$.  Since its divisor
is defined over $\Fq$, it is regular and nonzero at the generic
characteristic point $\Xi$, and coefficient Frobenius fixes it.  Write
\begin{equation}\label{eq:shtuka-differential-coordinate}
 \widetilde\lambda_\psi=\alpha\,\nu .
\end{equation}

\begin{definition}[Anderson--Thakur residue logarithm]
\label{def:log}
Define the formal $\Fq$-linear series
\begin{equation}\label{eq:residue-logarithm}
 \log_\psi(Z)
 :=\sum_{i\geqslant0}
 \operatorname{res}_{\Xi}\!\left(
  \frac{\widetilde\lambda_\psi^{(i+1)}}
       {f_\psi f_\psi^{(1)}\cdots f_\psi^{(i)}}
 \right)Z^{q^i}.
\end{equation}
\end{definition}

\begin{lemma}
\label{lem:one-step-coefficient-recursion}
{Let $f$ be a {nonzero} meromorphic function, $c$ a meromorphic
differential, and
\[
 B(Z)=\sum_{i\geqslant0}Q_iZ^{q^i}
\]
a formal $\Fq$-linear series with differential coefficients satisfying
\[
 {B(Z)^{(1)}-fB(Z)=-c\,Z,}
 \qquad
 {B(Z)^{(1)}:=\sum_{i\geqslant0}Q_i^{(1)}Z^{q^{i+1}}.}
\]
{Here the twist acts on the differential coefficients and sends
$Z$ to $Z^q$.}
Then
\[
 Q_0=\frac{c}{f},
 \qquad
 Q_i=\frac{Q_{i-1}^{(1)}}{f}\quad(i\geqslant1),
 \qquad\text{i.e.}\qquad
 Q_i=\frac{c^{(i)}}{f f^{(1)}\cdots f^{(i)}}.
\]}
\end{lemma}
\begin{proof}
{Compare the coefficients of $Z$ and of $Z^{q^i}$ ($i\geqslant1$)
in {$B(Z)^{(1)}-fB(Z)=-c\,Z$}; \revblue{the closed formula for $Q_i$} follows by induction.}
\end{proof}

\begin{theorem}
\label{prop:log-is-log}
The series in \eqref{eq:residue-logarithm} is the formal logarithm of
$\psi$. Its linear coefficient is one, and it converges on the usual logarithm disk.
\end{theorem}

\begin{proof}
{Write
\[
 \log_\psi^{\mathrm{std}}(Z)=\sum_{i\geqslant0}\ell_iZ^{q^i}
\]
for the formal logarithm of $\psi$, the composition inverse of the
exponential series, and put
\[
 B(Z):=\mathcal H_\psi\bigl(\log_\psi^{\mathrm{std}}(Z)\bigr)
 =\sum_{i\geqslant0}Q_iZ^{q^i},
 \qquad
 \mathcal H_\psi(w)=\sum_{n\geqslant0}A_nw^{q^n},
\]
where the second display is the orbital expansion of the generating
differential (Definition~\ref{def:kernel}).  {Explicitly,
\[
 Q_i=\sum_{n=0}^{i}A_n\ell_{i-n}^{q^n},
\]
so each $Q_i$ is a finite sum of rational differentials.}
The identity-isogeny case of Theorem~\ref{thm:frob},
together with
$\exp_\psi\bigl(\log_\psi^{\mathrm{std}}(Z)\bigr)=Z$, gives
\[
 {B(Z)^{(1)}-f_\psi B(Z)=-\widetilde\lambda_\psi^{(1)}Z.}
\]
{Lemma~\ref{lem:one-step-coefficient-recursion} with
$f=f_\psi$ and $c=\widetilde\lambda_\psi^{(1)}$ gives}
\[
 Q_i=\frac{\widetilde\lambda_\psi^{(i+1)}}{F_\psi^{[i+1]}}
 \qquad(i\geqslant0).
\]
{The coefficientwise residue computation in
Proposition~\ref{lem:residue} gives
$\operatorname{res}_\Xi A_0=1$ and
$\operatorname{res}_\Xi A_n=0$ for $n\geqslant1$.
Taking residues in the finite sum defining $Q_i$ therefore gives
$\operatorname{res}_\Xi Q_i=\ell_i$.  Thus, as formal series,}
\[
 \log_\psi^{\mathrm{std}}(Z)
 =\sum_{i\geqslant0}\operatorname{res}_\Xi\!
  \left(\frac{\widetilde\lambda_\psi^{(i+1)}}{F_\psi^{[i+1]}}\right)Z^{q^i},
\]
which is \eqref{eq:residue-logarithm}.  The linear coefficient is
$\operatorname{res}_\Xi(\widetilde\lambda_\psi^{(1)}/f_\psi)=1$ by
\eqref{eq:shtuka-differential-residue}, {and the coefficient identity gives convergence on the usual
logarithm disk.}}
\end{proof}

\begin{remark}{Anderson's residue formula is presented in
\cite[Proposition~0.3.8]{Tha93}.
Chung proves it for arbitrary infinity degree with the
characteristic-residue normalization used here
\cite[Theorem~II.20]{Chu21}.
The infinity normalization is discussed in
\cite[Remark~II.21]{Chu21} and \cite[item~13]{Tha93Corr}.}
\end{remark}

\begin{definition}
\label{def:large-log}
On a domain on which the coefficients are defined, put
\begin{equation}\label{eq:large-log}
 \operatorname{Log}_{\alpha,f_\psi}(\two;Z)
 :=Z+\sum_{i\geqslant1}
 \frac{\alpha^{(i+1)}(\two)}
      {\alpha^{(1)}(\two)
       f_\psi^{(1)}(\two)\cdots f_\psi^{(i)}(\two)}Z^{q^i}.
\end{equation}
It is first a formal series in $Z$; on an affinoid it is analytic precisely
when its terms tend to zero in the corresponding Banach algebra.
\end{definition}

{For $f_\psi=t-\theta$ and $\alpha=1$, specializing $Z$
gives the Carlitz deformed logarithms of \cite[\S6.1.1]{Pap08}.}

\begin{proposition}
\label{prop:large-log}
One has
\begin{equation} \label{eq:large-log-specialization}
 \operatorname{Log}_{\alpha,f_\psi}(\Xi;Z)
  =\log_\psi(Z).                                 
\end{equation}
{This identity holds} formally and hence on every common domain of normal
convergence.
\end{proposition}

\begin{proof}
Put
\[
 \gamma_i=\frac{\widetilde\lambda_\psi^{(i+1)}}{F_\psi^{[i+1]}}
 \qquad(i\geqslant0).
\]
Telescoping the shtuka divisor gives
\[
 \operatorname{div}(\gamma_i)
 \geqslant V_\psi-\sum_{j=0}^{i}[\Xi^{(j)}]
       +\sum_{j=0}^{i}[q_\psi^{(j)}]
       -[q_\psi^{(i-1)}]-[q_\psi^{(i)}].
\]
In particular, $\gamma_i$ has at most a simple pole at $\Xi$.
The differential $\gamma_0$ has residue one there.
Thus the rational quotient $\gamma_i/\gamma_0$ is regular at $\Xi$,
and its value is $\operatorname{res}_\Xi\gamma_i$. Indeed, in a local coordinate $t$ at $\Xi$ the quotient is $(a_{-1}t^{-1}+O(1))/(t^{-1}+O(1))$ with $a_{-1}=\operatorname{res}_\Xi\gamma_i$, hence regular with value $a_{-1}$.
On the chosen differential chart this quotient is precisely
\[
 \frac{\alpha^{(i+1)}}{\alpha^{(1)}
                 f_\psi^{(1)}\cdots f_\psi^{(i)}}.
\]
Here the quotient is simplified as a rational function before evaluation;
individual factors may have zeros or poles.
The residue formula for $\log_\psi$ now proves
\eqref{eq:large-log-specialization}.
\end{proof}

\subsection{Papanikolas-type generating series}
\label{sec:large-logarithms}

\begin{definition}
\label{def:one-step-H}
Let $f$ and $\alpha$ be meromorphic functions on $U_{\Cinf}$, {with $f\ne0$,} and put
\[
F^{[i+1]} :=\prod_{j=0}^{i}f^{(j)},\qquad i\geqslant0.
\]
On an affinoid $W\subset U_{\Cinf}^{\mathrm{an}}$ on which all displayed
coefficients are analytic, define
\begin{equation}\label{eq:H-series}
 \mathcal{P}_{\alpha,f}(\two;Z)
 :=\sum_{i=0}^{\infty}
 \frac{\alpha^{(i+1)}(\two)}{F^{[i+1]}(\two)}Z^{q^i}
\end{equation}
whenever the terms tend to zero in $\mathcal O(W)$.
\end{definition}

The domain is the union of the pairs $(W,Z)$ for which the terms tend to zero;
the divisor of $f$ alone need not give a single Tate algebra valid for every
$Z\in\Cinf$.

\begin{proposition}
\label{prop:defect-H}
Let $B\subset\Cinf$ be bounded.  If
\begin{equation}\label{eq:H-normal-convergence-condition}
 \sup_{Z\in B}
 \left\|
  \frac{\alpha^{(i+1)}}{F^{[i+1]}}Z^{q^i}
 \right\|_W\longrightarrow0,
\end{equation}
then \eqref{eq:H-series} converges normally on $W\times B$,
{and the following identity holds meromorphically:}
\begin{equation}\label{eq:defect-H}
 {\mathcal{P}_{\alpha,f}(\two;Z)^{(1)}}
 -f(\two)\mathcal{P}_{\alpha,f}(\two;Z)
 =-\alpha^{(1)}(\two)Z.
\end{equation}
\end{proposition}

\begin{proof}
Over a complete nonarchimedean Banach algebra, a series converges if and only
if its terms tend to zero.  Condition
\eqref{eq:H-normal-convergence-condition} gives uniform convergence on
$W\times B$.  {Put $Q_i=\alpha^{(i+1)}/F^{[i+1]}$.
The rational identity $Q_i^{(1)}=fQ_{i+1}$ implies
$dQ_i^{(1)}=dfQ_{i+1}$ for a nonzero analytic denominator $d$ on
$W$ such that $df$ is analytic.  The assumed convergence, with
index $i+1$, therefore gives normal convergence of
$d\sum_iQ_i^{(1)}Z^{q^{i+1}}$.
Thus the twisted series is well defined meromorphically, and}
\[
 {\mathcal{P}_{\alpha,f}(\two;Z)^{(1)}}
 =\sum_{i\geqslant0}
 \frac{\alpha^{(i+2)}}{f^{(1)}\cdots f^{(i+1)}}Z^{q^{i+1}},
\]
whereas
\[
 f\mathcal{P}_{\alpha,f}
 =\alpha^{(1)}Z+
 \sum_{i\geqslant1}
 \frac{\alpha^{(i+1)}}{f^{(1)}\cdots f^{(i)}}Z^{q^i}.
\]
After shifting the index in the first sum, every term except the initial one
cancels.  This is the normally convergent form of the one-step defect
equation for the Papanikolas-type series.
\end{proof}
\begin{proposition} \label{prop:H-large-log} 
  The relation of the Papanikolas-type series to the deformed logarithm is
  \begin{equation}\label{eq:H-large-log} 
  \mathcal{P}_{\alpha,f_\psi}(\two;Z)
 =\frac{\alpha^{(1)}(\two)}{f_\psi(\two)}
   \operatorname{Log}_{\alpha,f_\psi}(\two;Z).   
   \end{equation}
\end{proposition}
\begin{proof}
Multiplying \eqref{eq:large-log} by $\alpha^{(1)}/f_\psi$
proves \eqref{eq:H-large-log} term by term.
\end{proof}

\subsection{Isogeny transport of the \texorpdfstring{deformed logarithm}{deformed logarithm}}

Let $J=[\eta_J;h_J;\ii_J,\mu_J]:\varphi\to\psi$ be a separable isogeny datum.

\begin{lemma}
\label{lem:isogeny-H}
On every common domain of normal convergence,
\begin{equation}\label{eq:isogeny-H}
 \mathcal{P}_{\alpha,f_\psi}(\two;Z)
 =h_J(\two)\,
  \mathcal{P}_{\alpha/h_J,f_\varphi}(\two;Z).
\end{equation}
\end{lemma}

\begin{proof}
Iteration of $h_J^{(1)}f_\varphi=h_Jf_\psi$ gives
\[
h_J^{(i+1)}\prod_{j=0}^{i}f_\varphi^{(j)}
 =h_J\prod_{j=0}^{i}f_\psi^{(j)}.
\]
After substituting this identity, the two series in
\eqref{eq:isogeny-H} agree term by term.
\end{proof}

\begin{corollary}
\label{cor:isogeny-log}
One has
\begin{equation}\label{eq:isogeny-large-log}
 \mathcal{P}_{\alpha/h_J,f_\varphi}(\two;Z)
 =\frac{\alpha^{(1)}(\two)}
 {h_J^{(1)}(\two)f_\varphi(\two)}
 \operatorname{Log}_{\alpha,f_\psi}(\two;Z).
\end{equation}
\end{corollary}

\begin{proof}
Combine \eqref{eq:H-large-log}, \eqref{eq:isogeny-H}, and
$h_J^{(1)}f_\varphi=h_Jf_\psi$.
\end{proof}

\begin{theorem}
\label{thm:defect-phi}
{Under \eqref{eq:H-normal-convergence-condition} for
$(\alpha/h_J,f_\varphi)$, the following identity holds meromorphically:}
\begin{equation}\label{eq:defect-phi}
 {\mathcal{P}_{\alpha/h_J,f_\varphi}(\two;Z)^{(1)}}
 -f_\varphi(\two)\mathcal{P}_{\alpha/h_J,f_\varphi}(\two;Z)
 =-\frac{\alpha^{(1)}(\two)}{h_J^{(1)}(\two)}Z.
\end{equation}
\end{theorem}

\begin{proof}
Apply Proposition~\ref{prop:defect-H} with the coefficient
$\alpha/h_J$ and the shtuka function $f_\varphi$.
\end{proof}

\revsubsection{The relation between $\mathcal{G}_J$ and Papanikolas-type series}{The relation between G-J and Papanikolas-type series}
Let $\mathbb D_\psi$ be the convergence disk of $\log_\psi$, and put
$\mathbb{M}_A:=\mathbb T_A\otimes_A\Omega_A$ throughout this subsection. Give the finite projective module $\mathbb{M}_A$
a Banach norm, and give $h_J^{-1}\mathbb{M}_A$ the norm transported by multiplication
by $h_J$. Normal convergence below means uniform convergence in this fixed
module on each closed parameter disk contained in $\mathbb D_\psi$.

The bridge follows by composing the kernel family with the local logarithm.
Its Frobenius equation then determines every coefficient of the resulting
power series.

\begin{theorem}
\label{thm:HJ-equal}
Write $\widetilde\lambda_\psi=\alpha\nu$ with $0\ne\nu\in{\Omega_A^{\rat}}$.
\begin{enumerate}[label=\textup{(\roman*)},leftmargin=2.2em]
\item\label{item:bridge-logarithm-branch}
The series $\mathcal{P}_{\alpha/h_J,f_\varphi}(Z)\nu$ converges normally
in $h_J^{-1}\mathbb{M}_A$ for $Z\in \mathbb D_\psi$, and
\begin{equation}\label{eq:HJ-logarithm-branch}
 \mathcal{P}_{\alpha/h_J,f_\varphi}(Z)\nu
 =\mathcal G_J\bigl(h_J(\Xi)^{-1}\log_\psi(Z)\bigr).
\end{equation}
\item\label{item:bridge-exponential-branch}
Let $\mathbb D_0$ be a disk about $0$ such that
$\exp_\psi(h_J(\Xi)\mathbb D_0)\subset \mathbb D_\psi$ and
$\log_\psi(\exp_\psi(h_J(\Xi)z))=h_J(\Xi)z$ for $z\in \mathbb D_0$. Then
\begin{equation}\label{eq:HJ-local-large-log-equality}
 \mathcal{G}_J(z)
 =\mathcal{P}_{\alpha/h_J,f_\varphi}
 \bigl(\exp_\psi(h_J(\Xi) z)\bigr)\,\nu_{\two}
 \qquad(z\in \mathbb D_0).
\end{equation}
\end{enumerate}
\end{theorem}

\begin{proof}
\emph{\ref{item:bridge-logarithm-branch}.}
Write the orbital expansion of the target family as
\[
 \mathcal H_\psi(w)=\sum_{n\geqslant0}A_nw^{q^n},\qquad
 A_n=\operatorname{orb}_{\psi,n}(K_\psi).
\]
The coefficients $A_n$ belong to $\mathbb{M}_A$: the target kernel has no
non-diagonal affine pole in its second variable, and the characteristic
orbit is disjoint from the Tate domain. The regularized specializations
at the finitely many exceptional first-variable points remain regular on
that domain. Proposition~\ref{thm:automatic-orbit-growth} gives
\[
 \|A_n\|_{\mathbb{M}_A}\leqslant |e_n(\psi)|C^{q^n}\quad(n\gg0),\qquad
 |e_n(\psi)|^{1/q^n}\longrightarrow0.
\]
Thus this series converges in $\mathbb{M}_A$, uniformly for $w$ in each bounded disk.

By Proposition~\ref{prop:GJ-Hpsi}, the function
\begin{equation}\label{eq:HJ-analytic-parameter-family}
 B(Z):=h_J^{-1}\mathcal H_\psi(\log_\psi Z)
      =\mathcal G_J(h_J(\Xi)^{-1}\log_\psi Z)
\end{equation}
is analytic on $\mathbb D_\psi$ with values in $h_J^{-1}\mathbb{M}_A$.
Indeed, on a closed subdisk where $|\log_\psi Z|\leqslant L$, the tail is bounded
by $|e_n(\psi)|(CL)^{q^n}$, which tends to zero. Its Taylor expansion is
$B(Z)=\sum_{i\geqslant0}Q_iZ^{q^i}$; each $Q_i$ is a rational differential,
since each coefficient of the formal composition is a finite sum.

{Coefficient Frobenius is continuous on $\mathbb T_A$
and $\mathbb{M}_A$, and maps $h_J^{-1}\mathbb{M}_A$ continuously to $(h_J^{(1)})^{-1}\mathbb{M}_A$.
It therefore acts termwise in this expression.}
Theorem~\ref{thm:frob} and $\exp_\psi(\log_\psi Z)=Z$ give
\begin{equation}\label{eq:HJ-parameter-defect}
 {B(Z)^{(1)}-f_\varphi B(Z)=-\widetilde\delta_J Z.}
\end{equation}
Choose a nonzero $b\in\Cinf\otimes_{\Fq}A$ such that
$b/h_J^{(1)},\,bf_\varphi/h_J\in\Cinf\otimes_{\Fq}A$ and
$b\widetilde\delta_J\in\Cinf\otimes_{\Fq}\Omega_A$.
Multiplication by $b$ makes both sides $\mathbb{M}_A$-valued, so their Taylor
coefficients agree.  {Lemma~\ref{lem:one-step-coefficient-recursion}
(with $f=f_\varphi$ and $c=\tilde\delta_J$) now gives the coefficient
recursion, and cancelling the nonzero scalar $b$ in the torsion-free
module $\mathbb M_A$ yields
\[
 Q_i=\frac{\tilde\delta_J^{(i)}}{F_\varphi^{[i+1]}}
 =\frac{\widetilde\lambda_\psi^{(i+1)}}
            {h_J^{(i+1)}F_\varphi^{[i+1]}}.
\]}
Since coefficient Frobenius fixes $\nu\in{\Omega_A^{\rat}}$, these are exactly the
coefficients of $\mathcal{P}_{\alpha/h_J,f_\varphi}(Z)\nu$.
That series is consequently the Taylor expansion of the analytic function
\eqref{eq:HJ-analytic-parameter-family}, proving its normal convergence
and \eqref{eq:HJ-logarithm-branch}.

\emph{\ref{item:bridge-exponential-branch}.}
Substitute $Z=\exp_\psi(h_J(\Xi)z)$ in \eqref{eq:HJ-logarithm-branch}.
The defining identity on $\mathbb D_0$ makes the argument of $\mathcal G_J$ equal
to $z$, proving \eqref{eq:HJ-local-large-log-equality}.
\end{proof}



\section{\texorpdfstring{The unit and its intrinsic presentation}{The unit and its intrinsic presentation}}\label{sec:unit-special}

\subsection{The intrinsic presentation: strict normalization and residue coordinates}\label{subsec:explicit-embedded-lattice-formulas}

The unit conjecture of Gazda and Maurischat \cite{GM21} predicts
a Tate-algebra unit in $\operatorname{sf}(\varphi)$
{whenever $\operatorname{sf}(\varphi)$ is free of rank one}; by
Theorem~\ref{thm:direct-shtuka-function-freeness-test} below, its content
reduces to the existence of a unit-normalizing scalar
(Definition~\ref{def:unit-normalizing-scalar}).

\begin{definition}\label{def:unit-normalizing-scalar}
For $\beta\in\Cinf^\times$ write $f_{\varphi,\beta}:=\beta^{1-q}f_\varphi$.
The scalar $\beta$ is \emph{unit-normalizing} if $f_{\varphi,\beta}\in\TT_A^\times$
and $\|f_{\varphi,\beta}-1\|_{\mathrm G}<1$.
\end{definition}

\begin{remark}\label{rem:unit-normalizing-coefficients}
Equivalently, writing $f_\varphi=\sum_i\varsigma_i\otimes a_i$ in
a fixed $\Fq$-basis $(a_i)$ with $a_0=1$, a scalar $\beta$ with
$\beta^{q-1}=\varsigma_0$ is unit-normalizing exactly when $f_\varphi\in\TT_A$,
$\varsigma_0\ne0$, and $\sup_{i>0}|\varsigma_i/\varsigma_0|<1$.  Indeed, if
$\beta$ is unit-normalizing, expand $f_{\varphi,\beta}=1+\sum_i d_i\otimes a_i$
with $\sup_i|d_i|<1$; then $f_\varphi=\beta^{q-1}f_{\varphi,\beta}$ gives
$\varsigma_0=\beta^{q-1}(1+d_0)\ne0$ and
$\varsigma_i=\beta^{q-1}d_i$ for $i>0$, with
$|\varsigma_i/\varsigma_0|=|d_i|<1$ since $|1+d_0|=1$.  Conversely, the
coefficient conditions make $f_{\varphi,\beta}=\varsigma_0^{-1}f_\varphi$
a strict $1$-unit.
\end{remark}

\begin{proposition}
\label{prop:strict-shtuka-normalization-lattice}
Let $f_\varphi$ be the shtuka function of a rank-one Drinfeld module $\varphi$.
Suppose that a unit-normalizing scalar exists
(Definition~\ref{def:unit-normalizing-scalar}) and fix one, $\beta$:
\begin{equation}\label{eq:strict-shtuka-normalization}
  f_{\varphi,\beta} 
 :=\beta^{1-q}f_\varphi\in\mathbb T_A^\times,
 \qquad
 \bigl\| f_{\varphi,\beta} -1\bigr\|_{\mathrm G}<1.
\end{equation}
\begin{enumerate}[label=\textup{(\roman*)},leftmargin=2.2em]
\item\label{item:strict-normalized-product} The product
\begin{equation}\label{eq:strict-normalized-shtuka-product}
 v_{\varphi,\beta}
 :=\prod_{n=0}^{\infty}
   \left((f_{\varphi,\beta})^{(n)}\right)^{-1}
 \end{equation}
converges in $1+\{x\in\mathbb T_A:\|x\|_{\mathrm G}<1\}$.
\item\label{item:strict-normalized-special} The function $\omega_{\varphi,\beta}:=\beta v_{\varphi,\beta}
 \in\mathbb T_A^\times $ verifies
\[
\omega_{\varphi,\beta}^{(1)}=f_\varphi\omega_{\varphi,\beta}.\]
\end{enumerate}
\end{proposition}

\begin{proof}
\emph{\ref{item:strict-normalized-product}.}
Put $\rho=\| f_{\varphi,\beta} -1\|_{\mathrm G}<1$.  Then
\[
 \bigl\|(f_{\varphi,\beta})^{(n)}-1\bigr\|_{\mathrm G}=\rho^{q^n},
\]
so \eqref{eq:strict-normalized-shtuka-product} converges to a strict
$1$-unit.     

\emph{\ref{item:strict-normalized-special}.}
Shifting the factors gives
\[
 v_{\varphi,\beta}^{(1)}
 =\prod_{n\geqslant1}\left((f_{\varphi,\beta})^{(n)}\right)^{-1}
 = f_{\varphi,\beta}\cdot v_{\varphi,\beta}.
\]
Multiplication by $\beta$ gives
$\omega_{\varphi,\beta}^{(1)}
=\beta^q (f_{\varphi,\beta})v_{\varphi,\beta}
=f_\varphi\omega_{\varphi,\beta}$.
The scalar shtuka criterion proves that it is special.  
\end{proof}

\begin{example}\label{ex:elliptic-shtuka-unit}
This illustrates Proposition~\ref{prop:strict-shtuka-normalization-lattice} in the elliptic case.  Let
$A=\mathbb{F}_q[t,y]$ for an elliptic curve $E/\mathbb{F}_q$, and let
\begin{equation}\label{eq:elliptic-shtuka-GP}
 f=\frac{y-\eta-m(t-\theta)}{t-\alpha}
\end{equation}
be the shtuka function of \cite{GP18}: then
$V_\varphi=(\alpha,\beta)$ with $|\alpha|=q^2>1$ (in particular
$\Xi,\,q_\varphi^{(-1)}\notin\operatorname{supp}(V_\varphi)$ and
$h^0(\mathcal{O}(V_\varphi))=1$), and
\[
 m=\frac{\eta-\beta^q}{\theta-\alpha^q}
 \qquad\text{satisfies}\qquad |m|=q^q .
\]
Since $|\alpha|>1$, the expansion
\[
 \frac{1}{t-\alpha}=-\sum_{n\geqslant0}\alpha^{-n-1}t^n
\]
converges in the Gauss norm and places $f$ in $\mathbb{T}_A$.  In the
basis $\{t^i,t^jy\}$ of $A$, the constant coefficient is
\[
 \varsigma_0=-\frac{m\theta-\eta}{\alpha}=:\xi ,
\]
with no cancellation ($|m\theta|=q^{q+2}>q^3=|\eta|$), so
$|\varsigma_0|=q^q$.  Every other coefficient has absolute value
$\leqslant q^{-1}$: the coefficient of $t^n$ ($n\geqslant1$) is
$-(\beta+a_1\alpha+a_3)/\alpha^{n+1}$, and that of $t^jy$ is
$-\alpha^{-j-1}$.  Hence
\[
 \sup_{i>0}|\varsigma_i/\varsigma_0|\leqslant q^{-q-1}<1,
\]
and Remark~\ref{rem:unit-normalizing-coefficients} yields the unit-normalizing scalar
$\beta_*:=\xi^{1/(q-1)}$: one has
\[
 f_{\varphi,\beta_*}=f/\xi\in\mathbb{T}_A^\times,
 \qquad
 \|f/\xi-1\|_{\mathrm G}\leqslant q^{-q-1},
\]
and
\begin{equation}\label{eq:elliptic-omega-AT}
 \omega_{\varphi,\beta_*}
 =\xi^{1/(q-1)}\prod_{n\geqslant0}\frac{\xi^{q^n}}{f^{(n)}}
\end{equation}
is the Anderson--Thakur function $\omega_\rho$ of
\cite[Eq.~(51)]{GP18} (for the sign-normalized $f$).
\end{example}

\begin{proposition}\label{prop:unit-normalizing-scalars}
All unit-normalizing scalars define the same special function up to
$\Fq^\times$: for any two unit-normalizing $\beta$ and $\beta'$ there is
$\zeta\in\Fq^\times$ with
$\omega_{\varphi,\beta}=\zeta\,\omega_{\varphi,\beta'}$.
\end{proposition}
\begin{proof}
Fix a unit-normalizing scalar $\beta_0$, write
$\mathfrak m_\infty:=\{z\in\Cinf:|z|<1\}$ for the maximal ideal of {the valuation ring of $\Cinf$},
and let $\zeta\in\Fq^\times$ be the Teichm\"uller lift of the reduction of
$\beta/\beta_0$.  Every unit-normalizing $\beta$ has a unique decomposition
\[
 \beta=\zeta\,u\,\beta_0,\qquad \zeta\in\Fq^\times,\quad u\in 1+\mathfrak m_\infty,
\]
so the unit-normalizing scalars form $\beta_0\cdot\Fq^\times\cdot(1+\mathfrak m_\infty)$.
{Under this decomposition,
\[
 v_{\varphi,\beta}=u^{-1}\,v_{\varphi,\beta_0},\qquad
 \omega_{\varphi,\beta}=\zeta\,\omega_{\varphi,\beta_0}.
\]
These identities show that
$\beta\mapsto\omega_{\varphi,\beta}$ induces a bijection}
$\{\textup{unit-normalizing }\beta\}/(1+\mathfrak m_\infty)\xrightarrow{\sim}
\Fq^\times\omega_{\varphi,\beta_0}$; in particular
$\omega_{\varphi,\beta}=\omega_{\varphi,\beta'}$ exactly when
$\beta/\beta'\in 1+\mathfrak m_\infty$.
\end{proof}

\begin{theorem}
 \label{thm:direct-shtuka-function-freeness-test}
Fix an $\Fq$-basis $(a_i)_{i\geqslant0}$ of $A$ with $a_0=1$. The following are
equivalent:
\begin{enumerate}[label=\textup{(\roman*)},leftmargin=2.2em]
\item\label{item:coefficient-dominance}
A unit-normalizing scalar exists (Definition~\ref{def:unit-normalizing-scalar}).
\item\label{item:special-line-free}
$\operatorname{sf}(\varphi)$ is free of rank one over $A$.
\end{enumerate}
\end{theorem}

\begin{proof}
\emph{\ref{item:coefficient-dominance}$\Rightarrow$\ref{item:special-line-free}.}
Choose a unit-normalizing scalar $\beta_*$.
Proposition~\ref{prop:strict-shtuka-normalization-lattice} constructs the
unit
\begin{equation}\label{eq:direct-shtuka-function-generator}
\omega_*:=\omega_{\varphi,\beta_*}
=\beta_* v_{\varphi,\beta_*}
=\beta_*\prod_{n\geqslant0}
  \bigl(f_{\varphi,\beta_*}^{(n)}\bigr)^{-1}.
\end{equation}
For any $\xi\in\operatorname{sf}(\varphi)$, the scalar shtuka equations imply
$(\xi/\omega_*)^{(1)}=\xi/\omega_*$. The quotient belongs
to $\mathbb T_A$, so Lemma~\ref{lem:general-Tate-Frobenius-constants} places
it in $A$. Thus $\operatorname{sf}(\varphi)=A\omega_*$.

The reverse implication is proved after
Theorem~\ref{thm:unit-normalized-product}, where the strict unit form of the
unit theorem is available.

\end{proof}

\revsubsection{Norm conventions}{Norm conventions}

Recall that
$A^\vee_{\mathrm{res}}=\operatorname{Hom}_{\Fq}(A,\Fq)$, with the product
topology, is the residue--Pontryagin dual of the discrete $\Fq$-space $A$.
Residue duality
\cite[Theorem~8, applied with $I=A$]{Poo96} gives
\begin{equation}\label{eq:residue-duality-A-public}
 A^\vee_{\mathrm{res}}\simeq\overline\Omega,
\end{equation}
where the pairing is
\[
 (x,\eta)\longmapsto
 \operatorname{Tr}_{\mathbb F_\infty/\Fq}
 \operatorname{res}_\infty(x\eta).
\]
 For completeness, the proof of
Lemma~\ref{lem:residue-orthogonal-basis-public} below constructs the required
topological basis, verifies surjectivity, and identifies the quotient norm
directly.  

\begin{definition}
Let $
K_\infty^\vee:=\operatorname{Hom} _{\mathbb F_q}^{\mathrm{cont}}(K_\infty,\mathbb F_q)$.
We consider the operator norm  
\begin{align*}
\|\cdot\|_{\mathrm{op}}:
K_\infty^\vee & \to \mathbb{R} \\
\|\ell\|_{\mathrm{op}} = \sup_{0 \neq x\in K_\infty }\frac{|\ell(x)|}{|x|_\infty} &= (\inf \{ |x|_\infty : x \in K_\infty,\ \ell(x)\neq 0 \} )^{-1}.
\end{align*}
\end{definition}
By residue duality, there is a topological isomorphism
\[
\check{\mathscr{F}}:\widehat\Omega_A:=K_\infty\otimes_A\Omega_A
\xrightarrow{\ \sim\ }
K_\infty^\vee ,
\]
where
\[
\check{\mathscr{F}}(\eta)(x)
=
\operatorname{Tr}_{\mathbb F_\infty/\mathbb F_q}
\operatorname{res}_\infty(x\eta).
\]
The operator norm on \(K_\infty^\vee\) is transported to a \(K_\infty\)-linear norm
\(\|\cdot\|_{\widehat\Omega_A}\) on \(\widehat\Omega_A\). Since \(\widehat\Omega_A\) is one-dimensional over \(K_\infty\), choosing such a norm amounts to choosing a nonzero element
\(\eta_0\in\widehat\Omega_A\) and setting
\[
\|x\eta_0\|_{\widehat\Omega_A}
=
|x|_\infty\|\eta_0\|_{\widehat\Omega_A}
\qquad (x\in K_\infty).
\]
{This identity expresses the fixed transported norm in any
choice of basis $\eta_0$; changing the basis does not change the norm.}

We next define the quotient norm on \(A_{\mathrm{res}}^\vee\) by applying the operator norm on \(K_\infty^\vee\) (resp. the transported norm on \(\widehat\Omega_A\)).  Residue duality gives the exact sequence
\begin{equation}\label{eq:residue-duality-exact-public}
0\longrightarrow\Omega_A\longrightarrow\widehat\Omega_A
\xrightarrow{\ \operatorname{Tr}\operatorname{res}_\infty\ } A_{\mathrm{res}}^\vee
\longrightarrow0.
\end{equation}
The restriction to $ A $ gives a morphism $
K_\infty^\vee \to A_{\res}^\vee$. 

\begin{definition}
For \(a^*\in A_{\mathrm{res}}^\vee\), define
\[
|a^*|
:=
\inf\bigl\{
\|\ell\|_{\mathrm{op}}
:\ell\in K_\infty^\vee,\ \ell|_A=a^*
\bigr\}.
\]
\end{definition}
Since the restriction map $
K_\infty^\vee\to A_{\mathrm{res}}^\vee $
is surjective, this infimum exists and defines a norm on \(A_{\mathrm{res}}^\vee\).

\begin{proposition}
The norm on \(A_{\mathrm{res}}^\vee\) equals the quotient norm:
\[
|a^*|
=
\|\bar\eta\|_{\mathrm{quot}}
:=
\inf_{\omega\in\Omega_A}
\|\eta+\omega\|_{\widehat\Omega_A},
\]
where \(\bar\eta\in\overline{\Omega}\) corresponds to \(a^*\).
\end{proposition}
\begin{proof}
The map \(\check{\mathscr{F}}:\widehat\Omega_A\to K_\infty^\vee\)
transports the operator norm on \(K_\infty^\vee\) to
\(\widehat\Omega_A\) by definition; hence it is an isometric
isomorphism.  A differential \(\eta\in\widehat\Omega_A\) annihilates
\(A\) under the residue pairing precisely when \(\eta\in\Omega_A\).
Thus \(\check{\mathscr{F}}\) identifies the subspace
\(\Omega_A\subset\widehat\Omega_A\) with the annihilator
\(\operatorname{Ann}(A)\subset K_\infty^\vee\), and induces an isometric
isomorphism
\[
\overline{\Omega}
\xrightarrow{\;\sim\;}
K_\infty^\vee/\operatorname{Ann}(A).
\]

For \(a^*\in A_{\mathrm{res}}^\vee\), a lift \(\ell\in K_\infty^\vee\)
with \(\ell|_A=a^*\) is exactly a representative of the corresponding
coset in \(K_\infty^\vee/\operatorname{Ann}(A)\).  Therefore
\[
\inf\{\|\ell\|_{\mathrm{op}}:\ell|_A=a^*\}
=
\|\bar\ell\|_{\mathrm{quot}},
\qquad
\text{the quotient norm on }
K_\infty^\vee/\operatorname{Ann}(A).
\]
Transporting back through this isometry gives the asserted formula.
\end{proof}


\subsection{Residue-orthogonal basis}

\begin{definition}\label{def:ai}
   The residue-orthogonal basis means the ordered $\Fq$-basis $ \{ a_i\}_{i \geqslant 0 } $ with $ a_0 = 1 $, with its dual $ \{ a_i^* \}_{\geqslant 0 }$ such that
   \begin{equation}\label{eq:dual-basis-norm-public}
 |a_i^*|=|a_i|_\infty^{-1}.
\end{equation}
\end{definition}
Since $a_0=1$, and
all $a_i$ with $i>0$ have a pole at $\infty$, the Riemann--Roch theorem implies that
\[
|a_i| > |a_0| = 1, \qquad |a_i| \longrightarrow \infty.  
\]

Consequently,
\begin{equation}\label{eq:unique-largest-dual-public}
 |a_0^*|>|a_i^*|\quad(i>0),
 \qquad |a_i^*|\longrightarrow0.
\end{equation}

\begin{lemma}
\label{lem:residue-orthogonal-basis-public}
There is a residue-orthogonal basis.
\end{lemma}

\begin{proof}
Choose a uniformizer $\pi$ of $K_\infty$ and set
\[
 A_{\leqslant n}:=A\cap\pi^{-n}\mathcal O_\infty
 =H^0\bigl(X,\mathcal O_X(n\infty)\bigr),\qquad n\geqslant0.
\]
Choose bases compatibly with this filtration and order the new vectors by
increasing pole order.  At the $n$-th step, choose the new vectors so that
their leading coefficients form an $\Fq$-basis of the image of
$A_{\leqslant n}/A_{\leqslant n-1}$ in $\mathbb F_\infty$.  The resulting basis is
orthogonal: the norm of a finite linear combination is the maximum norm of its
nonzero summands.  Since $A_{\leqslant0}=\Fq$, take $a_0=1$.  Every other basis vector
has a pole at $\infty$, hence has norm greater than one.

We spell out the extension to $K_\infty$, since several basis vectors may
have the same pole order when $\Ninf>1$.  For every integer $m$, the
graded piece
\[
 \pi^m\mathcal O_\infty/\pi^{m+1}\mathcal O_\infty
 \simeq\mathbb F_\infty
\]
is an $\Fq$-space of dimension $\Ninf$.  \revblue{At each filtration level $m=-n$ represented in $A$, extend the leading
coefficients of the chosen $a_i$ to an $\Fq$-basis of the graded piece.  Choose
any $\Fq$-basis at every remaining level.}  Lift these graded bases to elements
$b_{m,1},\ldots,b_{m,\Ninf}$ of valuation $m$, retaining the selected
$a_i$ themselves whenever they occur.  Successive subtraction of leading
terms gives, for every $x\in K_\infty$, a unique expansion
\[
 x=\sum_{m\geqslant v_\infty(x)}\sum_{\ell=1}^{\Ninf}
   \epsilon_{m,\ell}b_{m,\ell},
 \qquad \epsilon_{m,\ell}\in\Fq,
\]
convergent in increasing valuation.  At each valuation only finitely many
summands occur.  Thus the resulting family $(b_j)_{j\in J}$ is an orthogonal
topological $\Fq$-basis of $K_\infty$ containing $(a_i)_{i\geqslant0}$.

The local trace--residue pairing identifies
$K_\infty\otimes_A\Omega_A$ topologically with
$\operatorname{Hom}_{\Fq,\mathrm{cont}}(K_\infty,\Fq)$.  Hence every
coordinate functional of the topological basis $(b_j)$ is represented by a
unique differential; denote these representing differentials by $(b_j^*)$.
Orthogonality gives
\[
 \|b_j^*\|_{\mathrm{op}}=|b_j|_\infty^{-1}:
\]
the functional $b_j^*$ vanishes on every basis vector except $b_j$, while any
functional taking the value $1$ on $b_j$ has operator norm at least
$|b_j|_\infty^{-1}$.

{To prove that restriction onto $A^\vee_{\mathrm{res}}$ is
surjective, let $\ell\in A^\vee_{\mathrm{res}}$ and put}
$\epsilon_i=\ell(a_i)\in\Fq$.  For the indices of the $b$-family occupied by
the $a_i$, the series
\[
 \widetilde\ell:=\sum_{i\geqslant0}\epsilon_i b_i^*
\]
converges in the continuous dual of $K_\infty$: indeed
$\|b_i^*\|_{\mathrm{op}}=|a_i|_\infty^{-1}\to0$.  Its restriction to $A$
is $\ell$.  Thus the restriction map from the continuous residue dual of
$K_\infty$ onto $A^\vee_{\mathrm{res}}$ is surjective, and its kernel is the
annihilator
of $A$.  The global residue theorem gives $\Omega_A$ inside this
annihilator.  For the reverse inclusion, put
\[
 \Omega_{\geqslant n\infty}
 :=\{\eta\in K_\infty\otimes_A\Omega_A:
       \operatorname{ord}_\infty(\eta)\geqslant n\}.
\]
The finite-level principal-parts criterion used in the proof of
\cite[Theorem~8]{Poo96} gives the exact annihilator identity
\begin{equation}\label{eq:finite-level-residue-annihilator-public}
 \operatorname{Ann}(A_{\leqslant n})
 =\Omega_A+\Omega_{\geqslant n\infty}
 \qquad(n\geqslant0).
\end{equation}
Here the residue at the closed point means
$\operatorname{Tr}_{\mathbb F_\infty/\Fq}\operatorname{res}_\infty$, as
in \eqref{eq:residue-duality-exact-public}.  Since
$A=\bigcup_{n\geqslant0}A_{\leqslant n}$, the annihilator of $A$ is therefore
\[
 \bigcap_{n\geqslant0}
  \bigl(\Omega_A+\Omega_{\geqslant n\infty}\bigr)
 =\Omega_A.
\]
\revblue{The intersection equals $\Omega_A$} because $\Omega_A$ is discrete in $\widehat\Omega_A$ and the $\Omega_{\geqslant n\infty}$ form a decreasing
neighborhood basis of zero.  Hence the annihilator is exactly $\Omega_A$,
which proves the
exactness of \eqref{eq:residue-duality-exact-public} and realizes
\eqref{eq:residue-duality-A-public} with the claimed quotient norm.

In particular, restricting $b_i^*$ to $A$ gives the dual coordinate
$a_i^*$.  \revblue{Every lift of $a_i^*$ takes the value $1$ on $a_i$, so its norm is at
least $|a_i|_\infty^{-1}$.  Thus $b_i^*$ realizes the infimum defining the
quotient norm.}  This proves
\eqref{eq:dual-basis-norm-public}.  
\end{proof}
The strict comparison of the dual radii does not depend on the normalization.
\revblue{Multiply every norm on the one-dimensional differential space
$\widehat{\Omega}_A$ by a fixed $\gamma>0$.  Then
\eqref{eq:dual-basis-norm-public} becomes
$|a_i^*|=\gamma|a_i|_\infty^{-1}$, with the same $\gamma$ for every $i$.
This rescaling does not change the two conclusions in
\eqref{eq:unique-largest-dual-public}.}  
With the operator-norm normalization
fixed above, $\gamma=1$ and \eqref{eq:dual-basis-norm-public} is exact.

\subsection{Radial growth of the exponential}

Let $\Lambda\subset\Cinf$ be a discrete rank-one $A$-lattice and let
\[
 \exp_\Lambda(z)=z\prod_{0\ne\lambda\in\Lambda}
 \left(1-\frac z\lambda\right).
\]

\begin{lemma}
\label{lem:lattice-quotient-minimum-public}
Every class in $K_\infty\Lambda/\Lambda$ has a representative of minimum
norm.  A minimum representative $c$ of a nonzero class satisfies\revblue{:}
\begin{equation}\label{eq:minimum-representative-public}
 |c-\lambda|\geqslant |c|\qquad(\lambda\in\Lambda).
\end{equation}
\end{lemma}

\begin{proof}
Choose $0\ne\lambda_0\in\Lambda$.  \revblue{The lattice $\Lambda$ is projective of rank one over $A$.  Its generic
fiber is therefore one-dimensional over $K$.  Choose a fractional $A$-ideal
$I\subset K$ such that:}
\begin{equation}\label{eq:rank-one-lattice-fractional-ideal-public}
 \Lambda=\iota(I)\lambda_0.
\end{equation}
For every $C>0$, the elements of $I$ of norm at most $C$ have a uniformly
bounded pole order at $\infty$.  After clearing one fixed denominator of
$I$, they lie in a Riemann--Roch space
$H^0(X,\mathcal O_X(n\infty))$ for some $n$ depending on $C$.  This is a
finite-dimensional $\Fq$-space and hence a finite set.  Therefore
\begin{equation}\label{eq:rank-one-lattice-bounded-finite-public}
 \#\{\lambda\in\Lambda:|\lambda|\leqslant C\}<\infty.
\end{equation}

Start with any representative $c_0$.  If $|\lambda|>|c_0|$, then
$|c_0-\lambda|=|\lambda|>|c_0|$.  Hence a representative of norm at most
$|c_0|$ can only be obtained by subtracting a lattice point in the bounded
ball $\{\lambda\in\Lambda:|\lambda|\leqslant |c_0|\}$.  This set is finite by
\eqref{eq:rank-one-lattice-bounded-finite-public}.  The infimum of the norms
in the coset is therefore a minimum.  Inequality
\eqref{eq:minimum-representative-public} is exactly its minimality.  If the
class is nonzero, its minimum representative is nonzero.
\end{proof}

\begin{notation}\label{not:ci}
    Let
\(
\jmath:\Omega_A\xrightarrow{\sim}\Lambda_\varphi
\)
be an $A$-module isomorphism.  \revblue{Let $(a_i)_{i\geqslant0}$ be the residue-orthogonal basis from
Lemma~\ref{lem:residue-orthogonal-basis-public}, with residue-dual classes
$a_i^*$.  For each $i$, choose $c_i=c_i(\jmath)\in K_\infty\Lambda_\varphi$ of
minimum norm in the image of $a_i^*$ under:}
\[
 \overline\Omega
 \xrightarrow{\overline{\jmath} }
 K_\infty\Lambda_\varphi/\Lambda_\varphi.
\]
\end{notation}
\begin{lemma}
\label{lem:residue-dual-radius-transport-public}
With the notation in \ref{not:ci}, there is one constant $\gamma_\jmath>0$, independent of $i$, such that
\begin{equation}\label{eq:residue-dual-radius-transport-public}
 |c_i|=\gamma_\jmath |a_i|_\infty^{-1}
 \qquad(i\geqslant0).
\end{equation}
In particular, every $c_i$ represents a nonzero quotient class and
\begin{equation}\label{eq:residue-dual-radius-dominance-public}
 |c_0|>|c_i|\quad(i>0),
 \qquad |c_i|\longrightarrow0.
\end{equation}
\end{lemma}

\begin{proof}
Give $K_\infty\otimes_A\Omega_A$ the norm transported from the continuous
residue dual of $K_\infty$, as in
Lemma~\ref{lem:residue-orthogonal-basis-public}, and give
$K_\infty\Lambda_\varphi\subset\Cinf$ the inherited absolute value.  Choose
$0\ne\eta_0\in K_\infty\otimes_A\Omega_A$ and put
\[
 \gamma_\jmath:=
 \frac{|\jmath_\infty(\eta_0)|}{\|\eta_0\|_{\widehat\Omega_A}}>0.
\]
Both spaces are one-dimensional over $K_\infty$.  Hence, writing
$\eta=x\eta_0$ with $x\in K_\infty$, one obtains the exact scaling identity
\begin{equation}\label{eq:residue-dual-line-scaling-public}
 |\jmath_\infty(\eta)|
 =\gamma_\jmath\|\eta\|_{\widehat\Omega_A}.
\end{equation}
Taking infima over the corresponding cosets modulo $\Omega_A$ and
$\Lambda_\varphi$ shows that the induced quotient isomorphism
$\overline{\jmath}_\infty$ scales every quotient norm by the same constant
$\gamma_\jmath$.  Lemma~\ref{lem:residue-orthogonal-basis-public} gives
$\|a_i^*\|_{\mathrm{quot}}=|a_i|_\infty^{-1}$, while the minimum choice of
$c_i$ and Lemma~\ref{lem:lattice-quotient-minimum-public} give
$|c_i|=\|\overline{\jmath}_\infty(a_i^*)\|_{\mathrm{quot}}$.  This proves
\eqref{eq:residue-dual-radius-transport-public}.  The classes $a_i^*$ are
nonzero, so their images and the minimum representatives $c_i$ are nonzero.
Finally $a_0=1$, every $a_i$ with $i>0$ has norm greater than one, and
$|a_i|_\infty\to\infty$; these facts give
\eqref{eq:residue-dual-radius-dominance-public}.
\end{proof}

\begin{example}
\label{ex:degree-two-infinity-residue-radii-public}
Let $X=\mathbb P^1_{\Fq}$, let $P(t)\in\Fq[t]$ be monic irreducible of
degree two, and choose the closed point $\infty=[P]$.  Then
\[
 A=\bigcup_{n\geqslant0}
 \left\{\frac{g(t)}{P(t)^n}:g\in\Fq[t],\ \deg g\leqslant2n\right\},
\]
and an $\Fq$-basis adapted to the pole filtration is
\begin{equation}\label{eq:degree-two-infinity-orthogonal-basis-public}
 a_0=1,
 \qquad
 a_{n,0}=P(t)^{-n},
 \qquad
 a_{n,1}=tP(t)^{-n}
 \quad(n\geqslant1).
\end{equation}
Indeed, Euclidean division of a numerator by $P$ proves spanning and
independence inductively.  At pole order $n$, the leading coefficients of
$a_{n,0}$ and $a_{n,1}$ are $1$ and $\bar t$ in
$\Fq[t]/P\simeq\mathbb F_{q^2}$; they are $\Fq$-independent.  Thus the basis
is orthogonal.  If $Q_\infty:=|P(t)^{-1}|_\infty>1$, its residue-dual
coordinates satisfy
\begin{equation}\label{eq:degree-two-infinity-dual-radii-public}
 |a_0^*|=1,
 \qquad
 |a_{n,0}^*|=|a_{n,1}^*|=Q_\infty^{-n}.
\end{equation}
{For a unit anchor $\jmath$, the corresponding minimum
representatives have radius $\gamma_\jmath Q_\infty^{-n}$.
Thus equal radii occur among the nonconstant coordinates, while the
constant coordinate is strictly dominant, as required by the unit proof.}
\end{example}

\begin{lemma}
\label{lem:radial-exp-norm-public}
If $c$ is a minimum-norm representative of a nonzero class in
$K_\infty\Lambda/\Lambda$, then
\begin{equation}\label{eq:radial-exp-formula-public}
 |\exp_\Lambda(c)|
 =R_\Lambda(|c|),
 \qquad
 R_\Lambda(r):=
 r\prod_{\substack{0\ne\lambda\in\Lambda\\|\lambda|<r}}
 \frac r{|\lambda|}.
\end{equation}
The product on the right is finite, and $R_\Lambda$ is strictly increasing on
the positive radii which occur as minimum norms of nonzero quotient classes.
\end{lemma}

\begin{proof}
If $|\lambda|>|c|$, then $|1-c/\lambda|=1$; if
$|\lambda|<|c|$, then $|1-c/\lambda|=|c|/|\lambda|$.  If
$|\lambda|=|c|$, Lemma~\ref{lem:lattice-quotient-minimum-public} gives
$|c-\lambda|\geqslant|c|$, while the ultrametric
inequality gives the reverse inequality.  The equal-radius factor therefore
also has norm one.  This proves \eqref{eq:radial-exp-formula-public}.  The
product is finite by \eqref{eq:rank-one-lattice-bounded-finite-public}.  If
$0<r_1<r_2$, then the factors already present at $r_1$, together with the
initial factor, contribute a power of $r_2/r_1$ strictly greater than one to
$R_\Lambda(r_2)/R_\Lambda(r_1)$; every new factor, indexed by
$r_1\leqslant|\lambda|<r_2$, is also strictly greater than one.  Thus
$R_\Lambda(r_1)<R_\Lambda(r_2)$.  In particular this holds on the radii of
minimum representatives.
\end{proof}

\subsection{The arbitrary-degree unit theorem}

{In the residue-dual expansion
\eqref{eq:generating-function-residue-dual-expansion}, the pole norms
$|a_i|_\infty$ determine the radii of the minimum representatives
$c_i$.  The Gauss norm measures the unweighted coefficients
$\exp_\varphi(c_i)$, since every basis vector $a_i$ has Gauss norm one.
The radial norm formula transfers the strict inequality
$|c_0|>|c_i|$ for $i>0$ to these coefficients, giving invertibility
by a Neumann series.}

{For rational infinity, this dual-norm argument appears in
\cite[\S4]{Fer25}; the residue-orthogonal construction above
allows repeated pole orders.}

\begin{theorem}
\label{thm:unit-special-arbitrary-degree-public}
Let $\jmath:\Omega_A\xrightarrow{\sim}\Lambda_\varphi$ be a {unit anchor}, and choose
the minimum representatives $c_i$ of Notation~\ref{not:ci}. Put
$\beta_0=\exp_\varphi(c_0)$. Then $\beta_0\ne0$, and:
\begin{enumerate}[label=\textup{(\roman*)},leftmargin=2.2em]
\item\label{item:unit-dominance} The residue-dual coordinate has the form
\begin{equation}\label{eq:quantitative-unit-factorization-public}
\omega_\jmath=\beta_0(1+\varepsilon_\jmath),
\qquad \|\varepsilon_\jmath\|_{\mathrm G}<1.
\end{equation}
\item\label{item:unit-inverse} It is a unit of $\mathbb T_A$, with
\[
\|\omega_\jmath\|_{\mathrm G}=|\beta_0|,
\qquad
\|\omega_\jmath^{-1}\|_{\mathrm G}=|\beta_0|^{-1}.
\]
\end{enumerate}
\end{theorem}

\begin{proof}
\emph{\ref{item:unit-dominance}.}
Lemma~\ref{lem:residue-dual-radius-transport-public} gives
\begin{equation}\label{eq:quantitative-residue-coordinate-radii-public}
|c_i|=\gamma_\jmath|a_i|_\infty^{-1}.
\end{equation}
In particular, the radius dominance \eqref{eq:residue-dual-radius-dominance-public} of Lemma~\ref{lem:residue-dual-radius-transport-public} holds.
Each $c_i$ represents a nonzero class modulo $\Lambda_\varphi$, so
$\exp_\varphi(c_i)\ne0$. The radial formula gives
\begin{equation}\label{eq:quantitative-exponential-radii-public}
|\exp_\varphi(c_i)|=R_{\Lambda_\varphi}(\gamma_\jmath|a_i|_\infty^{-1}).
\end{equation}
Strict monotonicity therefore implies
\begin{equation}\label{eq:dominant-exp-public}
|\exp_\varphi(c_i)|<|\beta_0|\qquad(i>0).
\end{equation}
By \eqref{eq:generating-function-residue-dual-expansion},
\begin{equation}\label{eq:generating-unit-expansion-public}
\omega_\jmath=\sum_{i\geqslant0}\exp_\varphi(c_i)\otimes a_i.
\end{equation}
Define
\begin{equation}\label{eq:vjmath}
\varepsilon_\jmath=\sum_{i>0}\frac{\exp_\varphi(c_i)}{\beta_0}\otimes a_i,
\qquad v_\jmath=1+\varepsilon_\jmath.
\end{equation}
The coefficients tend to zero because $c_i\to0$. Each has norm less than
one, so their supremum is also less than one: a tail is uniformly small
and only finitely many coefficients remain. The coefficient description
of the Gauss norm proves \eqref{eq:quantitative-unit-factorization-public}.

\emph{\ref{item:unit-inverse}.}
The Neumann series $\sum_{n\geqslant0}(-\varepsilon_\jmath)^n$ converges to
$v_\jmath^{-1}$ and has norm at most one. Since $\|v_\jmath\|_{\mathrm G}=1$,
submultiplicativity in $v_\jmath v_\jmath^{-1}=1$ gives
$\|v_\jmath^{-1}\|_{\mathrm G}=1$. Multiplication by $\beta_0$ gives both
asserted norms.
\end{proof}



\subsection{\texorpdfstring{Proof of the coefficient criterion, reverse implication}{Proof of the coefficient criterion, reverse implication}}

\begin{corollary}
 \label{thm:unit-normalized-product} 
 Let $ \jmath$ be a {unit anchor}, and choose $c_0 = c_0(\jmath)\in K_\infty\Lambda_\varphi$
of minimum norm in the image of $a_0^*$ as in Notation~\ref{not:ci}; put $\beta_0=\exp_\varphi(c_0)$.
Then $\beta_0$ is a unit-normalizing scalar, and
\[
\omega_\jmath
=\beta_0\prod_{n\geqslant0}
  \left((f_{\varphi,\beta_0})^{(n)}\right)^{-1}
\qquad\text{in }\mathbb T_A.
\]
\end{corollary}
\begin{proof}
The unit theorem (Theorem~\ref{thm:unit-special-arbitrary-degree-public})
writes $\omega_\jmath=\beta_0(1+\varepsilon_\jmath)$ with
$\|\varepsilon_\jmath\|_{\mathrm G}<1$; put
$v_\jmath=\beta_0^{-1}\omega_\jmath=1+\varepsilon_\jmath$.
The scalar shtuka equation gives
\[
f_{\varphi,\beta_0}
=\frac{v_\jmath^{(1)}}{v_\jmath},
\qquad
f_{\varphi,\beta_0}-1
=\frac{\varepsilon_\jmath^{(1)}-\varepsilon_\jmath}{1+\varepsilon_\jmath},
\]
of Gauss norm less than one: $\beta_0$ is a unit-normalizing scalar.
The finite products telescope:
\[
\prod_{n=0}^{M-1}
 \left((f_{\varphi,\beta_0})^{(n)}\right)^{-1}
=\prod_{n=0}^{M-1}\frac{v_\jmath^{(n)}}{v_\jmath^{(n+1)}}
=\frac{v_\jmath}{v_\jmath^{(M)}}.
\]
Since
$\|\varepsilon_\jmath^{(M)}\|_{\mathrm G}
=\|\varepsilon_\jmath\|_{\mathrm G}^{q^M}\to0$,
the denominators and their inverses tend to one. Therefore
\begin{equation}\label{eq:unit-normalized-product-public}
v_\jmath=\prod_{n\geqslant0}
 \left((f_{\varphi,\beta_0})^{(n)}\right)^{-1}.
\end{equation}
Multiplication by $\beta_0$ proves the assertion.
\end{proof}

It remains to prove the reverse implication of Theorem~\ref{thm:direct-shtuka-function-freeness-test}.
\begin{proof}
\emph{\ref{item:special-line-free}$\Rightarrow$\ref{item:coefficient-dominance}.}
An $A$-basis $\omega$ corresponds, by
Theorem~\ref{thm:direct-logarithmic-special-line}, to a homomorphism
$\jmath:\Omega_A\to\Lambda_\varphi$ with $\omega=\omega_\jmath$;
by part (iii) of the same theorem, $\jmath$ is an isomorphism.
Corollary~\ref{thm:unit-normalized-product} makes
$\beta_0=\exp_\varphi(c_0)$ a unit-normalizing scalar.

\end{proof}

\subsection{\texorpdfstring{{Period extraction}}{Period extraction}}
Explicit formulas in this direction go back to \cite{EP13}, and
periods of $t$-motives to \cite{Sin97}.
{For the shtuka-product and characteristic-residue formula
on elliptic curves, see \cite[Theorem~4.6]{GP18}.}
For the explicit lattice formulas, we normalize the scalar
coordinate of the unit.

{To extract a characteristic residue from the product
expression for $\omega_\jmath$ in
Corollary~\ref{thm:unit-normalized-product}, we extend it
meromorphically from the Tate domain to $U_{\Cinf}^{\mathrm{an}}$.
Lemma~\ref{lem:strict-product-affine-convergence} shows that the
product of the late Frobenius factors converges to an analytic unit
on every affinoid.  The identity principle then identifies this
expression with the continuation of $\omega_\jmath$ from
Theorem~\ref{thm:rational-anderson-continuation}.}

\begin{lemma}
\label{lem:strict-product-affine-convergence}
Let $g\in\Cinf(X)\cap\mathbb T_A^\times$ satisfy
$\|g-1\|_{\mathrm G}<1$.
\begin{enumerate}[label=\textup{(\roman*)},leftmargin=2.2em]
\item\label{item:affine-unit-tail} On every affinoid
$W\subset U_{\Cinf}^{\mathrm{an}}$, the functions $g^{(n)}$ are analytic
units for all sufficiently large $n$, satisfy $\|g^{(n)}-1\|_W\to0$, and
the product of the factors from any sufficiently large index on is an
analytic unit on $W$.
\item\label{item:product-meromorphic-convergence} For every
$n_0\geqslant0$, the product $\prod_{n\geqslant n_0}(g^{(n)})^{-1}$ converges locally
meromorphically on $U_{\Cinf}^{\mathrm{an}}$.
\end{enumerate}
\end{lemma}

\begin{proof}
Choose a nonconstant $t\in A$ and a free $\Fq[t]$-basis $e_1,\ldots,e_s$ of
$A$. Write $g-1=\sum_{j,m}c_{j,m}e_jt^m$. The coefficient functions in this
finite free presentation are rational in $t$ and analytic on $|t|\leqslant1$.
Their finitely many poles lie outside that disc, so for some $R_0>1$ one has
$|c_{j,m}|R_0^m\to0$. Also $\sup_{j,m}|c_{j,m}|<1$. By first bounding a
tail and then the finitely many remaining coefficients, choose
$1<R<R_0$ with $\rho:=\sup_{j,m}|c_{j,m}|R^m<1$.
 Concretely, bound the tail via $|c_{j,m}|R_0^m\leqslant B$ and pick $M$ with $B(R/R_0)^{M+1}<1$; then take $R>1$ close enough to $1$ that the finitely many remaining coefficients also contribute less than $1$.
On $U_S=\{|t|\leqslant S\}$, where $S\geqslant1$, put
$C_S=\max_j\|e_j\|_{U_S}$. For $S\leqslant R^{q^n}$,
\[
\|g^{(n)}-1\|_{U_S}\leqslant C_S\rho^{q^n}.
\]
Every $W$ is contained in some $U_S$.

\emph{\ref{item:affine-unit-tail}.}
\revblue{The bound tends to zero.  Hence the Neumann series supplies an inverse
for every sufficiently large $n$, and the late product converges to an analytic
unit.}

\emph{\ref{item:product-meromorphic-convergence}.}
The finitely many earlier
rational factors give the meromorphic limit. The limits agree on overlaps
and restrict to the Tate product.
\end{proof}

{\revblue{More generally, retain the isogeny datum of
Notation~\ref{not:primitive-G}, including $\Omega_A=\ii_J\nu_J$ and
$\jmath(\nu_J)=z_J=\tilde\pi/h_J(\Xi)$ after localization.  Do not impose
$\Xi^{(-1)}\notin\operatorname{supp}(V_\varphi)$.  Then the period has the
unified residue form:}
\begin{equation}\label{eq:unified-residue-period}
 \tilde{\pi}
 =h_J(\Xi)\,\beta_0\,
  \operatorname{res}_\Xi\Bigl(\nu_J
   \prod_{n\geqslant0}\bigl(f_{\varphi,\beta_0}^{(n)}\bigr)^{-1}\Bigr),
\end{equation}
\revblue{Meromorphic continuation gives the product meaning
(Lemma~\ref{lem:strict-product-affine-convergence}).  On the Tate domain it
equals $\beta_0^{-1}\omega_\jmath$ by
Theorem~\ref{thm:unit-normalized-product}.  The primitive factorization gives
$z_J=\operatorname{res}_\Xi(\nu_J\omega_\jmath)$.  Multiplication by
$h_J(\Xi)$ proves the formula.}}

\begin{corollary}
\label{cor:simplified-period-extraction}
Use the \revblue{isogeny datum} of Notation~\ref{not:primitive-G}, and suppose
$\Omega_A=\ii_J\nu_J$ and
$\jmath(\nu_J)=z_J= \tilde{\pi} /h_J(\Xi)$ after localization.
Assume $\Xi^{(-1)}\notin\operatorname{supp}(V_\varphi)$, so that the normalized strict unit
$ f_{\varphi,\beta_0}$ has a simple zero at $\Xi$.
Put $\beta_0=\exp_\varphi(c_0)$.
Then 
\[
 \tilde{\pi}
= h_J(\Xi)\,\beta_0\,{\Upsilon_{\beta_0}(\Xi)}\operatorname{res}_\Xi(\nu_J/f_{\varphi,\beta_0}),
\qquad{\Upsilon_{\beta_0}:=\prod_{n\geqslant1}\bigl(f_{\varphi,\beta_0}^{(n)}\bigr)^{-1}},
\]
{where $\Upsilon_{\beta_0}(\Xi)$ is the value of the meromorphic product
after cancellation; individual early factors need not be regular at
$\Xi$.  This is the simple-zero evaluation of the unified residue
formula \eqref{eq:unified-residue-period}.}
\end{corollary}

\begin{proof}
 Put  ${\Upsilon_{\beta_0}}=\prod_{n\geqslant1}((f_{\varphi,\beta_0})^{(n)})^{-1}$. We shall show that ${\Upsilon_{\beta_0}}$ is regular at $\Xi$.
Summing the divisors of the inverse twists of $ f_{\varphi,\beta_0}$ gives
\[
\operatorname{div}\prod_{n=1}^{M}((f_{\varphi,\beta_0})^{(n)})^{-1}
=V_\varphi^{(1)}-V_\varphi^{(M+1)}
-\sum_{n=1}^{M}[\Xi^{(n)}]+\sum_{n=1}^{M}[q_\varphi^{(n)}].
\]
For large $M$ the order at $\Xi$ is $\operatorname{mult}_\Xi(V_\varphi^{(1)})=0$
(because $\Xi^{(-1)}\notin\operatorname{supp}(V_\varphi)$ implies
$\operatorname{mult}_\Xi(V_\varphi^{(1)})=0$);
the remaining tail is an analytic unit near $\Xi$
(by Lemma~\ref{lem:strict-product-affine-convergence}
\ref{item:affine-unit-tail}, applied to $f_{\varphi,\beta_0}$
on an affinoid neighborhood of $\Xi$).
Hence ${\Upsilon_{\beta_0}}$ is regular at $\Xi$.

On the Tate domain $\omega_\jmath=\beta_0 {\Upsilon_{\beta_0}}/ f_{\varphi,\beta_0}$
(by Theorem~\ref{thm:unit-normalized-product}, separating the
$n=0$ factor), and this identity extends
meromorphically (by Lemma~\ref{lem:strict-product-affine-convergence}
\ref{item:product-meromorphic-convergence}).  The primitive factorization gives
\[
z_J=\operatorname{res}_\Xi(\nu_J\omega_\jmath)
=\beta_0\,\operatorname{res}_\Xi\!\left(\frac{\nu_J\,{\Upsilon_{\beta_0}}}{ f_{\varphi,\beta_0} }\right).
\]
Because ${\Upsilon_{\beta_0}}$ is regular at $\Xi$ and $ f_{\varphi,\beta_0}$ has a simple zero there,
\[
\operatorname{res}_\Xi\!\left(\frac{\nu_J\,{\Upsilon_{\beta_0}}}{ f_{\varphi,\beta_0} }\right)
={\Upsilon_{\beta_0}}(\Xi)\operatorname{res}_\Xi(\nu_J/ f_{\varphi,\beta_0}).
\]
Multiplying by $h_J(\Xi)$ and using $z_J= \tilde{\pi} /h_J(\Xi)$
yields the formula.
\end{proof}

\begin{remark}
The $\tilde{\pi}$ of Corollary~\ref{cor:simplified-period-extraction} is normalized by both $\jmath$ and $\nu_J$: once the pair $(\jmath,\nu_J)$ is chosen, it is unique (Definition~\ref{def:normalized-period}).  In general, unit-normalizing scalars enter only through the $\Fq^\times$-ambiguity: for any unit-normalizing $\beta=\zeta u\beta_0$ as in Proposition~\ref{prop:unit-normalizing-scalars},
\[
 \tilde{\pi}
 =h_J(\Xi)\,\zeta^{-1}\beta\,{\Upsilon_\beta(\Xi)}
  \operatorname{res}_\Xi(\nu_J/f_{\varphi,\beta}),
 \qquad{\Upsilon_\beta:=\prod_{n\geqslant1}\bigl(f_{\varphi,\beta}^{(n)}\bigr)^{-1},}
\]
{with $\Upsilon_\beta(\Xi)$ taken after cancellation in the
meromorphic product (individual early factors need not be regular at
$\Xi$).}

\end{remark}

\section{Finite torsion traces on a curve with one closed infinite point}
\label{sec:higher-genus-single-infinity}
The normalized shtuka differential of Definition~\ref{def:shtuka-diff}
has a finite arithmetic expression.  The identity
$\widetilde\lambda_\psi^{(1)}=\mathcal R_\psi^{[1]}$, proved in
Theorem~\ref{thm:hg-shtuka-trace-identification}, identifies its first
coefficient twist with the first torsion pairing of a principal-period
module.  It holds at every closed infinity degree.

We construct the pairing by finite torsion traces and prove the identity
by determining its zeros, its possible poles on the infinity orbit, and
its characteristic residue.  The uniqueness theorem of Section~\ref{sec:shtuka-differential} then identifies it with the geometric
differential.  Contraction gives the scalar special values, and
free-target isogenies give the corresponding formulas for arbitrary
source lattices.  The curve zeta identity at rational infinity is the
final application.
\subsection{The curve, the lattice, and the normalization}

We retain \(X,A,K,\infty\) and the fixed characteristic embedding
\(\iota\).  Write \(\Ninf=\deg(\infty)\) and \(g=g(X)\), so
\[
 U=X\setminus\{\infty\},\qquad A=\Gamma(U,\mathcal O_X),\qquad K=\Fq(X).
\]
The completions remain \(K_\infty\) and \(\Cinf\).  In the arithmetic
formulas, \(\chi(a)=1\otimes a=a(\two)\) is the geometric copy,
in the sense of \cite{APTR16}.  With the normalization
of \S\ref{sec:basedata}, \(|\iota(a)|_\infty=q^{-\Ninf v_\infty(a)}\); the geometric
variable is independent of \(\iota\).  In the same Tate algebra, set
\begin{equation}
\label{eq:hg-all-element-value}
 \mathbb T_A=\Cinf\widehat\otimes_{\Fq}A,
 \qquad
 \mathcal L_A=\sum_{0\ne a\in A}\iota(a)^{-1}\chi(a).
\end{equation}
The Gauss norm is the unweighted coefficient norm in any \(\Fq\)-basis of
\(A\).

When $\Ninf=1$ and a sign function is fixed, write
\begin{equation}\label{eq:curve-zeta-monic-convention}
 A_+=\{a\in A:\operatorname{sgn}(a)=1\},
 \qquad
 \mathcal L_{A_+}=\sum_{a\in A_+}\frac{\chi(a)}{\iota(a)}.
\end{equation}
Each nonzero element of $A$ is uniquely a constant multiple of a monic
element, and the summand is unchanged by that multiple; hence
\begin{equation}\label{eq:curve-zeta-all-element-sign}
 \mathcal L_A=(q-1)\mathcal L_{A_+}=-\mathcal L_{A_+}.
\end{equation}

\begin{proposition}
\label{prop:hg-one-infinity-convergence}
The sum \(\mathcal L_A\) converges unconditionally in \(\mathbb T_A\), and
\[
 \|\mathcal L_A+1\|_{\mathrm G}\leqslant q^{-\Ninf}<1.
\]
In particular, \(\mathcal L_A\) is a unit.  Every positive-integer-weight
sum \(\sum_{a\ne0}\iota(a)^{-s}\chi(a)\), and each of its exact-sign
subfamilies, also converges in \(\mathbb T_A\).
\end{proposition}

\begin{proof}
The spaces \(H^0(X,\mathcal O_X(n\infty))\) are finite sets and exhaust \(A\).
A nonconstant element of \(A\) has a pole at \(\infty\), since otherwise it is
regular on the projective curve.  Thus only finitely many terms of the sum
have Gauss norm above any fixed positive bound.  This proves unconditional
convergence.

The nonzero constants contribute \(q-1=-1\), and every
remaining term has norm at most \(q^{-\Ninf}\).  The Neumann series proves
invertibility.  The same finiteness argument applies to every stated weight
and subfamily.
\end{proof}

Fix a field \(\iota(K)\subset H\subset\Cinf\).  The Drinfeld module
used in this section is an \(\Fq\)-algebra homomorphism
$\psi:A\longrightarrow H\{\tau\}$ with $\tau c=c^q\tau$,
 differential $\partial\psi_a=\iota(a)$, and
 $\deg_\tau\psi_a=-\Ninf v_\infty(a)$ for $0\ne a\in A$.
Throughout this section \(\psi\) denotes this fixed module on \(X\),
of generic characteristic given by \(\iota\) and rank one.
Its derivative-one exponential is the entire \(\Fq\)-linear
function satisfying
\[
 \exp_\psi(z)=z+\sum_{i\geqslant1}e_i z^{q^i},\qquad
 \exp_\psi(\iota(a) z)=\psi_a(\exp_\psi(z)).
\]
Drinfeld modules over \(\Cinf\) are uniformizable; write
\(\Lambda_\psi=\ker\exp_\psi\) for the embedded projective rank-one
\(A\)-lattice.  We first construct its special functions for an arbitrary
lattice, and then impose the principal-period normalization used in the
finite value formulas.


\subsection{Uniform small-ball vanishing lemma}

Choose \(a\in A\setminus\Fq\), put \(n=-v_\infty(a)\), and write
\[
 \psi_a(Z)=\iota(a) Z+\sum_{i=1}^{nN}\alpha_i Z^{q^i}.
\]
Set
\begin{equation}
\label{eq:hg-finite-differential-space}
 \mathcal V_a=H^0\bigl(X,\Omega_X^1((n+1)\infty)\bigr),
 \qquad d_a=\dim_{\Fq}\mathcal V_a=g-1+(n+1)\Ninf.
\end{equation}
Choose a basis \(\nu_1,\ldots,\nu_{d_a}\) and rational functions
\(r_1,\ldots,r_{d_a}\in K\) satisfying
\begin{equation}
\label{eq:hg-finite-residue-dual-basis}
 \langle r_j,\nu_k\rangle_\infty=\delta_{jk}.
\end{equation}
They exist by residue duality and the density of \(K\) in \(K_\infty\).
Choose one nonconstant \(b\in A\) clearing all their denominators, and put
\(c_j=br_j\in A\).

Define the finite \(H\)-algebra
\begin{equation}
\label{eq:hg-nonzero-torsion-algebra}
 \mathscr T_b=H[Z]/\bigl(\psi_b(Z)/Z\bigr).
\end{equation}
Multiplying the defining polynomial by its leading coefficient inverse
is understood when a monic polynomial is desired.  This algebra is finite
\'{e}tale of dimension \(q^{\deg(bA)}-1\): its geometric points are the
nonzero \(b\)-torsion points of \(\psi\).  The class of \(Z\) is a unit
because the defining polynomial has constant coefficient \(\iota(b)\ne0\).

We need two lattice calculations.  The first is the negative-index
vanishing of \cite[Proposition~5.20]{Fer25}, proved here with a
uniform estimate valid at every closed-point degree.  The second
expresses the residue-dual pairing as a finite torsion trace.

\begin{lemma}
\label{lem:hg-independent-lattice-vanishing}
Let \(\Gamma\subset\Cinf\) be a Drinfeld \(A\)-lattice of rank
\(r\geqslant1\), with derivative-one exponential \(E_\Gamma\).
\begin{enumerate}[label=\textup{(\roman*)},leftmargin=2.2em]
\item\label{item:small-ball-vanishing} For every
integer \(\ell\geqslant1\) and \(c\in K_\infty\) with
\(|c|_\infty\leqslant q^{-(\ell+1)/r}\),
\begin{equation}
\label{eq:hg-independent-negative-vanishing}
 \sum_{0\ne\lambda\in\Gamma}
       \frac{E_\Gamma(\iota(c)\lambda)}{\lambda^{q^{-\ell}}}=0.
\end{equation}
The sum converges unconditionally.
\item\label{item:small-ball-estimate} Choose an \(\Fq\)-basis of
\(\Gamma\) successively with least norm outside the preceding span,
and let \(\Gamma_m\) be the span of its first \(m\) vectors.
There is a constant \(C>0\), independent of \(m\) and of \(c\) in the
stated ball, such that, for all sufficiently large \(m\),
\begin{equation}
\label{eq:hg-uniform-truncation-bound}
 \left|\sum_{0\ne\lambda\in\Gamma_m}
       \frac{E_\Gamma(\iota(c)\lambda)}{\lambda^{q^{-\ell}}}\right|_\infty
 \leqslant Cq^{-(q-1)q^{m-Nr-\ell}/r}.
\end{equation}
Here negative Frobenius powers denote the unique roots in the perfect
field \(\Cinf\).
\end{enumerate}
\end{lemma}

\begin{proof}
\emph{\ref{item:small-ball-vanishing}: convergence.}
We suppress \(\iota\) in the proof. Riemann--Roch gives a compact set
\(F\subset K_\infty\) representing \(K_\infty/A\): sufficiently high principal
parts can be removed by elements of \(A\), leaving a fixed compact ball.
Choose finite \(A\)-generators \(\gamma_1,\ldots,\gamma_t\) of
\(\Gamma\).  For \(\lambda=\sum_j a_j\gamma_j\), reduce each
\(ca_j\) modulo \(A\) to an element of \(F\).  Additivity and $\Gamma$-periodicity of the exponential show that all \(E_\Gamma(c\lambda)\), for all
\(c\in K_\infty\) and \(\lambda\in\Gamma\), belong to the compact set
\(\{\sum_j E_\Gamma(f_j\gamma_j):f_j\in F\}\).
Thus the numerators are uniformly bounded.  Discreteness of the lattice
proves unconditional convergence, uniformly in \(c\).

\emph{\ref{item:small-ball-estimate}.}
Write \(s_i\) for the norms of the chosen basis of \(\Gamma\).
For a lattice
subspace of finite codimension \(d\), its norm sequence is eventually
\(s_{i+d}\): above the norms of fixed quotient representatives, the
dimensions of the two norm filtrations differ by \(d\).
Apply this to \(a\Gamma\subset\Gamma\).  Since
\(\dim_{\Fq}\Gamma/a\Gamma=r\deg(a)\), one obtains
\(s_{i+r\deg(a)}=|a|_\infty s_i\) for large \(i\).
Riemann--Roch supplies two functions whose \(\infty\)-pole orders differ by
one, say \(n\) and \(n+1\).  Put \(d=Nr\).  The two relations, with the
first applied at index \(i+d\), are
\[
 s_{i+d+dn}=q^{Nn}s_{i+d},\qquad
 s_{i+d(n+1)}=q^{\Ninf(n+1)}s_i
 \quad(i\gg0).
\]
Their left sides agree.  Dividing the right sides gives
\begin{equation}
\label{eq:hg-lattice-norm-period}
 d=Nr,\qquad s_{i+d}=q^Ns_i\quad(i\gg0).
\end{equation}

Write \(E_\Gamma(z)=\sum_{h\geqslant0}e_hz^{q^h}\) and set
\(P_m(j)=\sum_{\lambda\in\Gamma_m}\lambda^{q^j-1}\) for \(j\geqslant1\).
The lattice product and finite-field power sums give
\begin{align}
 |e_h|_\infty&\leqslant\prod_{i=1}^h s_i^{q^{i-1}-q^i},
 \label{eq:hg-lattice-coefficient-bound}\\
 P_m(j)&=0\quad(j<m),\qquad
 |P_m(j)|_\infty\leqslant s_m^{q^j-q^m}
       \prod_{i=1}^m s_i^{q^i-q^{i-1}}\quad(j\geqslant m).
 \label{eq:hg-finite-power-bound}
\end{align}
Here are direct verifications of these estimates.  The ordered basis is
\(\Fq\)-orthogonal by its minimality, so the \(q^i-q^{i-1}\)
elements of \(\Gamma_i\setminus\Gamma_{i-1}\) have norm \(s_i\).
Selecting the \(q^h-1\) smallest nonzero lattice elements in the
exponential product proves \eqref{eq:hg-lattice-coefficient-bound}.
In the expansion of a finite-field power sum, a nonzero monomial with
exponents \(b_1,\ldots,b_m\) has each \(b_i>0\) divisible by \(q-1\).
Its multinomial coefficient is nonzero only if addition of the exponents
has no carries in the characteristic, and hence none in base \(q\).
Consequently \(\sum_{i=1}^u b_i\geqslant q^u-1\) for every \(u\leqslant m\):
the base-\(q\) digit sum of this partial sum is at least \(u(q-1)\).
This proves the asserted vanishing.  For total degree \(q^j-1\),
monotonicity of \(s_i\) and the identity
\[
 \prod_{i=1}^m s_i^{b_i}
 =s_m^{q^j-1}\prod_{u=1}^{m-1}
       (s_u/s_{u+1})^{\sum_{i=1}^u b_i}
\]
give the remaining bound in \eqref{eq:hg-finite-power-bound}.

For \(m>\ell\), expansion of the exponential and the vanishing just
proved express the truncated sum as
\[
 S_m(c)=\sum_{j\geqslant m}e_{j-\ell}c^{q^{j-\ell}}
                                      P_m(j)^{q^{-\ell}}.
\]
Put
\[
 \rho=q^{-\Ninf\lceil(\ell+1)/d\rceil},\qquad
 A_i=\rho\frac{s_i}{s_{i-\ell}},
\]
and
\[ 
 D_i=q^{i-\ell}-q^{i-\ell-1}\quad(i\geqslant\ell+1).
\]
The value group gives \(|c|_\infty\leqslant\rho\leqslant q^{-(\ell+1)/r}\).
Using \eqref{eq:hg-lattice-coefficient-bound} and
\eqref{eq:hg-finite-power-bound}, and then \(s_i\geqslant s_m\) for
\(i>m\), yields
\begin{align}
 |S_m(c)|_\infty
 &\leqslant\sup_{j\geqslant m}\left\{
 \rho^{q^{j-\ell}}
 \prod_{u=1}^{j-\ell}s_u^{q^{u-1}-q^u}
 \left(\prod_{i=1}^j s_i^{q^i-q^{i-1}}\right)^{q^{-\ell}}
 \right\}\notag\\
 &=C_0\sup_{j\geqslant m}\prod_{i=\ell+1}^j A_i^{D_i},
 \label{eq:hg-retained-prefix-bound}
\end{align}
where 
\[ 
 C_0=\rho^{q^{-\ell}}
       \prod_{i=1}^{\ell}(\rho s_i)^{q^{i-\ell}-q^{i-\ell-1}}.
\]
In particular, \(C_0\) is independent of \(m,j,c\).

By \eqref{eq:hg-lattice-norm-period} and monotonicity, eventually
\[
 A_i\leqslant\rho q^{\Ninf\lceil\ell/d\rceil}\leqslant1,
 \qquad A_{i+d}=A_i.
\]
For any sufficiently late block of \(d\) consecutive indices, the same
period relation gives
\[
 \prod_{i=s+1}^{s+d}\frac{s_i}{s_{i-\ell}}
 =\prod_{h=0}^{\ell-1}\frac{s_{s+d-h}}{s_{s-h}}
 =q^{\Ninf\ell},
 \qquad
 \prod_{i=s+1}^{s+d}A_i=\rho^dq^{\Ninf\ell}\leqslant q^{-\Ninf}.
\]
Thus every such block contains an \(A_i\leqslant q^{-1/r}\).
Choose a fixed sufficiently large \(I\) and absorb only the factors
with \(i<I\) into a constant \(C\).  For \(m\geqslant I\),
\[
 |S_m(c)|_\infty\leqslant C\sup_{j\geqslant m}\prod_{i=I}^j A_i^{D_i}
             =C\prod_{i=I}^m A_i^{D_i}.
\]
For large \(m\), a strict contraction occurs among the last \(d\)
indices, where \(D_i\geqslant(q-1)q^{m-d-\ell}\).  All other factors
are at most one, so the cumulative-product bound proves
\eqref{eq:hg-uniform-truncation-bound}.  Passing to the unconditional
limit gives the vanishing in \ref{item:small-ball-vanishing}.
\end{proof}

\subsection{Differential pairing}
Assume that $\Lambda_\psi = \iota(A) \tilde{\pi} $ is a principal period lattice.
 Write $\mathcal{H}_\psi(\tilde{\pi})$ for the special-function value at the period $\tilde{\pi}$. 
\begin{definition}\label{def:completed-differential-pairing}
   For Drinfeld module $ \psi $, we define the completed differential pairing directly by
\[
 \mathcal R_\psi^{[k]} :=- \tilde{\pi} ^{-1}\mathcal L_A  \mathcal{H}_\psi(\tilde{\pi})^{(k)}
 \quad(k\geqslant0).
\]
We abbreviate $\mathcal R_\psi:=\mathcal R_\psi^{[0]}$.
\end{definition}
The two identities below follow from the scalar shtuka equation and the Drinfeld action, respectively.
\begin{lemma}\label{lem:special-R}
Let $\psi$ be the Drinfeld module as above. Then
\[
        \mathcal R_\psi^{[k]} = F_\psi^{[k]} \mathcal R_\psi
    \]  
    with $ F_\psi^{[k]} = f_\psi^{(0)} \cdots f_\psi^{(k-1)} $.
    Moreover, 
    \begin{equation}
\label{eq:hg-cleared-pairing}
 \chi(a) \mathcal R_\psi
       =\iota(a) \mathcal R_\psi + \sum_{i=1}^{nN}\alpha_i\mathcal R_\psi^{[i]} .
\end{equation}
\end{lemma}
\begin{proof}
 The definition of $\mathcal{H}_\psi(\tilde{\pi})$  yields 
\[
 \mathcal{H}_\psi(\tilde{\pi}) ^{(1)}
 =f_\psi  \mathcal{H}_\psi(\tilde{\pi}) .
\]
Inductively, twisting $ \mathcal{H}_\psi(\tilde{\pi}) ^{(k)}=F_\psi^{[k]} \mathcal{H}_\psi(\tilde{\pi}) $ once and using this
identity gives
\[
 \mathcal{H}_\psi(\tilde{\pi}) ^{(k+1)}=(F_\psi^{[k]})^{(1)}f_\psi \mathcal{H}_\psi(\tilde{\pi}) =F_\psi^{[k+1]} \mathcal{H}_\psi(\tilde{\pi}) .
\]
Substitution in the definition of each pairing yields
\begin{equation}\label{eq:hg-pairing-frobenius-products}
 \mathcal R_\psi^{[k]}=F_\psi^{[k]}\mathcal R_\psi,
 \qquad k\geqslant0.
\end{equation}
Here the factor $- \tilde{\pi}^{-1}\mathcal L_A$ in every pairing is unchanged.
For the second assertion, each contraction of $ \mathcal{H}_\psi(\tilde{\pi})$ is a special
function, so $\psi_a( \mathcal{H}_\psi(\tilde{\pi}))=\chi(a) \mathcal{H}_\psi(\tilde{\pi})$; a finite dual frame gives
the same equality for the differential itself. Expand $\psi_a$ and
multiply by $- \tilde{\pi}^{-1}\mathcal L_A$ to obtain
\eqref{eq:hg-cleared-pairing}.
\end{proof}

 \begin{lemma} 
\label{lem:hg-torsion-evaluation}
Assume that \(\Lambda_\psi=\iota(A) \tilde{\pi}\) is a principal period lattice, and
fix \(k\geqslant0\).
\begin{enumerate}[label=\textup{(\roman*)},leftmargin=2.2em]
\item\label{item:lattice-series-expression} For every \(u\in K_\infty\), the
residue pairing satisfies\revblue{:}
\begin{equation}\label{eq:uR}
\left\langle u,\mathcal R_\psi^{[k]}\right\rangle_{\infty}
=
-\sum_{0\neq \lambda\in \Lambda_\psi}
\frac{\exp_\psi(\iota(u)\lambda)^{q^k}}{\lambda}.
\end{equation}
\item\label{item:finite-torsion-evaluation} For \(b\in A\setminus\Fq\) and
\(c\in A\),
\begin{equation}\label{eq:hg-finite-pairing-at-rational-class}
 \langle c/b,  \mathcal R_\psi^{[k]} \rangle_{\infty} 
 =-\iota(b)^{-1}\operatorname{Tr}_{\mathscr T_b/H}
       \left(\frac{\psi_c(Z)^{q^k}}{Z}\right).
\end{equation}
\end{enumerate}
{For rational classes $u=c/b$, equation
\eqref{eq:hg-finite-pairing-at-rational-class} expresses the pairing
as a finite trace over the field of definition $H$, without adjoining
individual torsion points.}
\end{lemma}

\begin{proof}
\emph{\ref{item:lattice-series-expression}.}
By Definition~\ref{def:completed-differential-pairing},
\[
\mathcal R_\psi^{[k]}
=
- \tilde{\pi}^{-1}\mathcal L_A  \mathcal{H}_\psi(\tilde{\pi})^{(k)}.
\]
Under residue duality in Proposition~\ref{prop:dualelement}, \( \mathcal{H}_\psi(\tilde{\pi})\)
corresponds to the continuous function
\[
 F_{\tilde{\pi}}\colon K_\infty/A\longrightarrow  \Cinf,\qquad
 F_{\tilde{\pi}}(c+A)=\exp_\psi( \tilde{\pi}\, \iota(c)).
\]
Therefore, for \(u\in K_\infty\),
\[
\left\langle u, \mathcal{H}_\psi(\tilde{\pi})\right\rangle_{\infty}
=
 F_{\tilde{\pi}}(u+A)
=
\exp_\psi\bigl( \tilde{\pi}\,\iota(u)\bigr).
\]
Multiplication by \(\chi(a)=1\otimes a\) under residue duality corresponds to
scaling the argument by \(a\), i.e.
\[
\left\langle u,\chi(a) \mathcal{H}_\psi(\tilde{\pi})\right\rangle_{\infty}
=
 F_{\tilde{\pi}}(a u+A)
=
\exp_\psi\bigl( \tilde{\pi}\, \iota(a)\iota(u)\bigr).
\]
The coefficient-Frobenius map \( \mathcal{H}_\psi(\tilde{\pi})\mapsto  \mathcal{H}_\psi(\tilde{\pi})^{(k)}\)
raises function values to the \(q^k\)-th power, hence
\[
\left\langle u,\chi(a) \mathcal{H}_\psi(\tilde{\pi})^{(k)}\right\rangle_{\infty}
=
\exp_\psi\bigl( \tilde{\pi}\, \iota(a)\iota(u)\bigr)^{q^k}.
\]
Recall the definition
\[
\mathcal L_A=\sum_{0\neq a\in A}\frac{\chi(a)}{\iota(a)}.
\]
It follows that
\begin{equation}\label{eq:residue-dual-pairing}
\left\langle u,\mathcal L_A  \mathcal{H}_\psi(\tilde{\pi})^{(k)}\right\rangle_{\infty}
=
\sum_{0\neq a\in A}
\frac{1}{\iota(a)}
\exp_\psi\bigl( \tilde{\pi} \, \iota(a)\iota(u)\bigr)^{q^k}.
\end{equation}
 \revblue{The lattice estimate in Lemma~\ref{lem:hg-independent-lattice-vanishing}
proves unconditional convergence.  We use the rank-one case and the convergence
step of its proof.  The same estimate also proves convergence of \eqref{eq:uR}
below.}
Multiplying by the factor \(- \tilde{\pi}^{-1}\) from
Definition~\ref{def:completed-differential-pairing} yields
\eqref{eq:residue-dual-pairing} with the sign attached, i.e.
\[
\left\langle u,\mathcal R_\psi^{[k]}\right\rangle_{\infty}
=
- \tilde{\pi}^{-1}
\sum_{0\neq a\in A}
\frac{1}{\iota(a)}
\exp_\psi\bigl( \tilde{\pi}\, \iota(a)\iota(u)\bigr)^{q^k}.
\]
Now set \(\lambda:= \tilde{\pi}\,\iota(a)\). As \(a\) runs over \(A\setminus\{0\}\),
\(\lambda\) runs exactly over \(\Lambda_\psi\setminus\{0\}\). Therefore the
preceding display becomes \eqref{eq:uR}.

\emph{\ref{item:finite-torsion-evaluation}.}
Apply \ref{item:lattice-series-expression} with \(u=c/b\):
\[
\left\langle \frac{c}{b},\mathcal R_\psi^{[k]}\right\rangle_{\infty}
=
-\sum_{0\neq \lambda\in \Lambda_\psi}
\frac{\exp_\psi\bigl(\iota(c)\lambda/\iota(b)\bigr)^{q^k}}{\lambda}.
\]
The numerators take finitely many values, so the series converges and may be
grouped modulo \(\iota(b)\Lambda_\psi\).  The zero coset contributes zero.
For a nonzero coset represented by \(\lambda_0\), put
\(x=\exp_\psi(\lambda_0/\iota(b))\).  The lattice partial-fraction identity,
obtained by logarithmically differentiating the derivative-one exponential
product, gives
\[
 \sum_{\mu\in \iota(b)\Lambda_\psi}
      \frac1{\lambda_0+\mu}
 =\frac1{\exp_{\iota(b)\Lambda_\psi}(\lambda_0)}
 =\frac1{\iota(b) x}.
\]
For the first equality, the derivative of the normalized lattice
exponential is \(1\); its logarithmic derivative is the sum over its
simple zeros.  Changing \(\mu\) to \(-\mu\) leaves the lattice unchanged.
For the second, the functions
\(\exp_{\iota(b)\Lambda_\psi}(z)\) and
\(\iota(b)\exp_\psi(z/\iota(b))\) have the same normalized lattice product.
Thus their values at \(\lambda_0\) agree.
The numerator on that coset is \(\psi_c(x)^{q^k}\).  Since the exponential
identifies \(\Lambda_\psi/\iota(b)\Lambda_\psi\) with \(\psi[b]\), summing
over nonzero cosets is exactly the finite \'{e}tale trace in
\eqref{eq:hg-finite-pairing-at-rational-class}.
\end{proof}

\begin{remark}
The  formula \eqref{eq:uR} agrees with the universal pairing
\(\zeta_\psi\mathbin{\cdot}\boldsymbol\omega_\psi^{(k)}\) in
\cite[Lemma~7.2]{Fer25}, with
\(\zeta_\psi=-\sum_{\lambda\ne0}\lambda^{-1}\otimes\lambda\).
\end{remark}

\begin{notation}\label{not:Tj}
For \(1\leqslant j\leqslant d_a\), put
\begin{equation}
\label{eq:hg-explicit-torsion-trace-coefficients}
 T_j^\psi =\operatorname{Tr}_{\mathscr T_b/H}
 \left(
 \frac{\psi_{ac_j}(Z)-\iota(a)\psi_{c_j}(Z)}{Z}
 \right)\in H.
\end{equation}
The quotient inside the trace is a polynomial.  All these data are obtained
from the curve, a finite residue matrix, and the Drinfeld action polynomials.
\end{notation}

\begin{theorem} 
\label{thm:hg-finite-torsion-trace-value}
The pairings have the following finite expressions and pole bounds.
\begin{enumerate}[label=\textup{(\roman*)},leftmargin=2.2em]
    \item\label{item:positive-finite-traces} For $k >0$, we have
\begin{equation}
\label{eq:hg-positive-pairing-pole-bound}
 \mathcal R_\psi^{[k]}\in H\otimes_{\Fq}
 H^0\left(X,\Omega_X^1\left(
            \left\lceil\frac{k+1}{\Ninf}\right\rceil \infty\right)\right) .
\end{equation}
Moreover, put \(m_1=\lceil(k+1)/\Ninf\rceil\), and choose an \(\Fq\)-basis
\(\mu_1,\ldots,\mu_e\) of
\(H^0(X,\Omega_X^1(m_1\infty))\), where \(e=g-1+m_1\Ninf\).
Choose residue-dual classes \(c_j/b\), with \(c_j\in A\) and one
\(b\in A\setminus\Fq\) clearing their denominators.  Then
\begin{equation} \label{eq:Rk}
 \mathcal R_\psi^{[k]} = -\frac1{\iota(b)}\sum_j
    \operatorname{Tr}_{\mathscr T_b/H}
       \!\left(\frac{\psi_{c_j}(Z)^{q^k}}{Z}\right)\mu_j.
\end{equation}
\item\label{item:zeroth-finite-trace} With the coefficients $T_j^\psi$ of
Notation~\ref{not:Tj}, the zeroth pairing is
\begin{equation}
\label{eq:hg-explicit-rational-differential}
 \mathcal R_\psi
 =-\frac{\sum_{j=1}^{d_a}T_j^\psi \nu_j}
         {\iota(b)\bigl(\chi(a)-\iota(a)\bigr)}.
\end{equation}
\item\label{item:zeroth-pole-bound} If \(\Xi\in X(H)\) is the characteristic point, then
\begin{equation}
\label{eq:hg-logarithmic-pole-bound}
 \mathcal R_\psi\in
 H^0\bigl(X_H,\Omega^1_{X_H/H}(\Xi+\infty_H)\bigr).
\end{equation}
\end{enumerate}
\end{theorem}

\begin{proof}
\emph{\ref{item:positive-finite-traces}.} For \(k\geqslant1\), apply
Lemma~\ref{lem:hg-independent-lattice-vanishing} with
\(\Gamma=\Lambda_\psi\), \(r=1\), and \(\ell=k\), then take the
\(q^k\)-th power.  The direct lattice expression for \(\mathcal R_\psi^{[k]}\)
shows that it annihilates every \(u\in K_\infty\) with
\(|u|_\infty\leqslant q^{-(k+1)}\).  Consequently its continuous residue-dual
map factors through the finite \(\Fq\)-vector space
\[
 K_\infty/\bigl(A+\mathfrak m_\infty^{\lceil(k+1)/\Ninf\rceil}\bigr).
\]
To identify its dual, put \(m=\lceil(k+1)/\Ninf\rceil\) and choose a local
parameter \(t_\infty\).  Residue duality identifies the continuous
\(\Fq\)-valued dual of \(K_\infty/A\) with \(\Omega_A\).
Write the local expansion of \(\nu\in\Omega_A\) as
\(\nu=\sum_j v_jt_\infty^j\,\rdt_\infty\), with
\(v_j\in\mathbb F_\infty\).  For \(n\geqslant m\) and
\(c\in\mathbb F_\infty\), one has
\[
 \langle ct_\infty^n,\nu\rangle_\infty
 =\operatorname{Tr}_{\mathbb F_\infty/\Fq}(c\,v_{-n-1}).
\]
Nondegeneracy of the finite-field trace pairing shows that
\(\nu\) annihilates \(\mathfrak m_\infty^m\) exactly when
\(v_j=0\) for every \(j<-m\).
Since \(\nu\) is already regular on \(U\), this is precisely
\(\nu\in H^0(X,\Omega_X^1(m\infty))\).
The quotient is finite: \(K_\infty/A\) is compact and the image of
\(\mathfrak m_\infty^m\) is open.
For a finite \(\Fq\)-vector space \(Q\),
\(\operatorname{Hom}_{\Fq}(Q,\Cinf)
=\Cinf\otimes_{\Fq}\operatorname{Hom}_{\Fq}(Q,\Fq)\).
Applying this to the quotient proves
\begin{equation} \label{eq:R_kinCinf}
 \mathcal R_\psi^{[k]}\in\Cinf\otimes_{\Fq}
 H^0\left(X,\Omega_X^1\left(
            \left\lceil\frac{k+1}{\Ninf}\right\rceil \infty\right)\right).
\end{equation}
The local calculation uses the full residue-field trace, so it also
applies when \(\Ninf>1\) or the characteristic divides \(\Ninf\).

From \eqref{eq:R_kinCinf}, we write 
\[
    \mathcal R_\psi^{[k]} = \sum_{j=1}^e \beta_j \mu_j . 
\]
Then Lemma~\ref{lem:hg-torsion-evaluation} \revblue{gives}
\[
\beta_j = \langle c_j/b, \mathcal R_\psi^{[k]} \rangle_{\infty} =  -\frac1{\iota(b)} 
    \operatorname{Tr}_{\mathscr T_b/H}
       \!\left(\frac{\psi_{c_j}(Z)^{q^k}}{Z}\right)\revblue{,}
\]
which confirms \eqref{eq:Rk}.
These coefficients lie in $H$, so the inclusion \eqref{eq:hg-positive-pairing-pole-bound} holds.

\emph{\ref{item:zeroth-finite-trace}.} Lemma~\ref{lem:special-R} gives
\begin{equation}
\label{eq:hg-cleared-pairing-finite-space}
 \mathcal{B}_a:=(\chi(a)-\iota(a))\mathcal R_\psi
       =\sum_{i=1}^{nN}\alpha_i\mathcal R_\psi^{[i]}
       \in\Cinf\otimes_{\Fq}\mathcal V_a.
\end{equation}
Expand \(\mathcal{B}_a=\sum_j\beta_j\nu_j\).  Evaluating on the dual classes
\(r_j=c_j/b\) gives
\begin{align*}
 \beta_j&= \langle r_j, \mathcal{B}_a \rangle_{\infty}
 =\langle ac_j/b, \mathcal R_\psi\rangle_{\infty} -\iota(a) \langle c_j/b, \mathcal R_\psi \rangle_{\infty}\\
 &=-\iota(b)^{-1}T_j^\psi
\end{align*}
by Lemma~\ref{lem:hg-torsion-evaluation}.  This proves
\[
(\chi(a)-\iota(a))\mathcal R_\psi = - \sum_j  \iota(b)^{-1}T_j^\psi \nu_j ,
\]
Since $a$ is nonconstant, $|\iota(a)|>1$ while
$\|\chi(a)\|_{\mathrm G}=1$. Thus
\[
(\chi(a)-\iota(a))^{-1}
=-\iota(a)^{-1}\sum_{m\geqslant0}\bigl(\chi(a)/\iota(a)\bigr)^m
\]
converges in $\mathbb T_A$. Division therefore recovers the original
completed pairing, proves \eqref{eq:hg-explicit-rational-differential},
and establishes its descent to $H$.
The residue-dual evaluations above are taken on the finite
algebraic differential $\mathcal{B}_a$.

\emph{\ref{item:zeroth-pole-bound}.} Repeat \eqref{eq:hg-cleared-pairing} with each
\(a\in A\).  The functions \(\chi(a)-\iota(a)\) generate the maximal
ideal of the smooth point \(\Xi\) on \(U_H\).  Consequently \(\mathcal R_\psi\)
has no finite pole away from \(\Xi\), and has at most a simple pole there.
At every geometric point above \(\infty\), the chosen denominator has pole
order \(n\), whereas \(\mathcal{B}_a\) has pole order at most \(n+1\).
Thus \(\mathcal R_\psi\) has at most a simple pole at every such point.
This proves \eqref{eq:hg-logarithmic-pole-bound}.
\end{proof}

\begin{remark}
\label{rem:hg-normalization-and-descent}
The field $ H $ in Theorem  {\ref{thm:hg-finite-torsion-trace-value}} is the field over which
the fixed Drinfeld action is defined.  In a Hayes realization over a finite
extension of \(\iota(K)\), the formula descends to that same extension.  
\end{remark}

\subsection{Identification with the normalized shtuka differential}
\label{subsec:hg-shtuka-identification}

Fix the geometric realization $(V_\psi,q_\psi,f_\psi)$ used in
Definition~\ref{def:shtuka-diff}, with its identity motive frame.
Recall from Definition~\ref{def:completed-differential-pairing} that
$\mathcal R_\psi=\mathcal R_\psi^{[0]}$.
The coarse pole bounds distinguish the first positive pairing from
the zeroth one.  To identify the former, we must also locate its
possible poles within the infinity orbit.

\begin{lemma} \label{lem:zeroth-pairing-shtuka-line}
With the notation above, the following hold.
\begin{enumerate}[label=\textup{(\roman*)},leftmargin=2.2em]
\item\label{item:zeroth-pairing-membership} 
The pairing
$\mathcal R_\psi$ belongs to $\mathscr S_\psi$; its only poles are
simple poles at $\Xi$ and $q_\psi^{(-1)}$;
\item\label{item:zeroth-pairing-residues}  
Its residues are
\begin{equation}\label{eq:curve-zeta-two-residues}
 \operatorname{res}_\Xi(\mathcal R_\psi)=1,\qquad
 \operatorname{res}_{q_\psi^{(-1)}}(\mathcal R_\psi)=-1.
\end{equation}
\end{enumerate}
\end{lemma}

\begin{proof}
All divisor calculations below are on $X_{\Cinf}$.  Put $V=V_\psi$,
$p_j=q_\psi^{(j)}$ for $0\leqslant j<\Ninf$.
By Lemma~\ref{lem:special-R},
\[
\mathcal R_\psi^{[k]} = F_\psi^{[k]} \mathcal R_\psi.
\]

\emph{\ref{item:zeroth-pairing-membership}.}
The additive-polynomial motive has $\Cinf$-basis
$(\tau^n s_\psi)_{n\geqslant0}$.  Under the geometric realization these
elements have coordinates $F_\psi^{[n]}$, by
\eqref{eq:geometric-motive-coordinate}.  Hence
\[
 \Gamma(U_{\Cinf},\OO(V))=\sum_{n\geqslant0}\Cinf F_\psi^{[n]}.
\]
Coefficient Frobenius is bijective over $\Cinf$ and takes
$\OO(V)$ to $\OO(V^{(1)})$.
Twisting this equality and using
$f_\psi(F_\psi^{[n]})^{(1)}=F_\psi^{[n+1]}$ gives
\begin{equation}\label{eq:curve-zeta-positive-motive}
 \sum_{k\geqslant1}\Cinf F_\psi^{[k]}
 =f_\psi\Gamma(U_{\Cinf},\OO(V^{(1)}))
 =\Gamma(U_{\Cinf},\OO(V-\Xi)).
\end{equation}
\revblue{For the identification with $\Gamma(U_{\Cinf},\OO(V-\Xi))$}, on $U$ one has
$\operatorname{div}(f_\psi)=V^{(1)}-V+\Xi$.
Thus $g$ satisfies $\operatorname{div}(g)+V^{(1)}\geqslant0$ on $U$
\revblue{precisely when} $f_\psi g$ satisfies
$\operatorname{div}(f_\psi g)+V-\Xi\geqslant0$ there.
Multiplication by $f_\psi$, with inverse division by $f_\psi$,
therefore identifies the two modules.
By \eqref{eq:hg-positive-pairing-pole-bound}, every
$F_\psi^{[k]}\mathcal R_\psi$ with $k\geqslant1$ is regular on $U_{\Cinf}$.
Thus multiplication by $\mathcal R_\psi$ sends the invertible module
on the right of \eqref{eq:curve-zeta-positive-motive} into regular
differentials.  This is a homomorphism of modules over
$\Gamma(U_{\Cinf},\OO)$.
On an affine scheme, localization of the module of sections of a
line bundle is its stalk.  At a geometric point $x\in U_{\Cinf}$,
choose a uniformizer $z$ and put
$n_x=\operatorname{mult}_x(V)-\mathbf1_{x=\Xi}$.
The stalk of $\OO(V-\Xi)$ is $z^{-n_x}\OO_{X_{\Cinf},x}$.
Writing $\mathcal R_\psi=b(z)\,dz$, the localized inclusion says that
$z^{-n_x}b(z)\,dz$ is regular.  It follows that
$\operatorname{ord}_x(\mathcal R_\psi)\geqslant n_x$.
For all $x$ together this gives
\[
 \operatorname{div}(\mathcal R_\psi)|_U\geqslant V|_U-[\Xi].
\]

Theorem~\ref{thm:hg-finite-torsion-trace-value} gives at most a simple
pole of $\mathcal R_\psi$ at each $p_j$.
If $\Ninf=1$, the only such point is $q_\psi^{(-1)}$, its multiplicity
in $V$ is zero, and the desired bound there is exactly
$\operatorname{ord}_{q_\psi^{(-1)}}(\mathcal R_\psi)\geqslant-1$.
Suppose $\Ninf\geqslant2$ and fix $0\leqslant j\leqslant \Ninf-2$.
Choose $k=j+1$, so that $\lceil(k+1)/\Ninf\rceil=1$.
The positive-pairing bound gives
$\operatorname{ord}_{p_j}(\mathcal R_\psi^{[k]})\geqslant-1$.
On the other hand, telescoping \eqref{eq:divf} gives
\[
 \operatorname{div}(F_\psi^{[k]})
 =V^{(k)}-V+\sum_{i=0}^{k-1}[\Xi^{(i)}]
                  -\sum_{i=0}^{k-1}[p_i].
\]
At $p_j=q_\psi^{(k-1)}$, the multiplicity of $V^{(k)}$ equals that
of $V$ at $q_\psi^{(-1)}$, which is zero.  The last sum contains
$p_j$ exactly once, and the characteristic orbit is disjoint from
the infinity orbit.  Hence
\[
 \operatorname{ord}_{p_j}(F_\psi^{[k]})
 =-\operatorname{mult}_{p_j}(V)-1,\qquad
 \operatorname{ord}_{p_j}(\mathcal R_\psi)
 \geqslant\operatorname{mult}_{p_j}(V).
\]
At the remaining point $p_{\Ninf-1}=q_\psi^{(-1)}$, the multiplicity
of $V$ is zero and the original simple-pole bound suffices.
Together with the affine calculation, this proves
$\mathcal R_\psi\in\mathscr S_\psi$ for every $\Ninf$.
No avoidance condition on the other points of the infinity orbit
has been used.

\emph{\ref{item:zeroth-pairing-residues}.}
Choose $a\in A\setminus\Fq$.  Residue-dual evaluation in
Lemma~\ref{lem:hg-torsion-evaluation}, first at $c=1,b=a$ and then
at $c=a,b=a$, gives
\begin{equation}\label{eq:hg-finite-trace-residue-normalization}
 \left\langle 1/a,\mathcal R_\psi\right\rangle_{\infty}
 =-\frac{\operatorname{Tr}_{\mathscr T_a/H}(1)}{\iota(a)}
 =\frac1{\iota(a)},\qquad
 \langle 1,\mathcal R_\psi\rangle_{\infty}=0.
\end{equation}
For the first evaluation, $\psi_1(Z)/Z=1$; multiplication by this
element is the identity on the
$(q^{\deg(aA)}-1)$-dimensional algebra $\mathscr T_a$.
Its trace is therefore $q^{\deg(aA)}-1$, whose image in $H$ is $-1$.
For the second evaluation, $\psi_a(Z)/Z$ is zero in $\mathscr T_a$
by its defining relation.  These evaluations are of residue classes,
not values of differentials at geometric points.

The cleared pairing
$\mathcal{B}_a=(\chi(a)-\iota(a))\mathcal R_\psi$ in
\eqref{eq:hg-cleared-pairing} is a finite global differential.
Write $\mathcal{B}_a=\sum_j\beta_j\nu_j$ in the finite $\Fq$-basis used in
Theorem~\ref{thm:hg-finite-torsion-trace-value}.
Its evaluation at $1/a$ is the finite sum
$\sum_j\beta_j\langle1/a,\nu_j\rangle_\infty$.
By \eqref{eq:residue-dual-coordinate-line}, each summand is an algebraic
residue paired with the residue-field trace.
Under the splitting
$\mathbb F_\infty\otimes_{\Fq}\Cinf\simeq\prod_{p\mid\infty}\Cinf$,
this trace is the sum of the components; the scalars $\beta_j$
remain fixed in the coefficient factor.
Consequently the finite sum is
$\sum_{p\mid\infty}\operatorname{res}_p(\mathcal{B}_a/\chi(a))$.
Multiplication by $\chi(a)$ on a residue-dual coordinate is
precomposition by multiplication by $a$, so
\begin{align*}
 -1
 &=\langle 1,\mathcal R_\psi\rangle_\infty-\iota(a)\langle 1/a,\mathcal R_\psi\rangle_\infty\\
 &=\langle 1/a, \mathcal{B}_a \rangle_\infty
 =\sum_{p\mid\infty}
    \operatorname{res}_p\!\left(\frac{\mathcal{B}_a}{\chi(a)}\right)
 =\sum_{p\mid\infty}\operatorname{res}_p(\mathcal R_\psi).
\end{align*}
\revblue{For the identification of the residue sum}, put $n=-v_\infty(a)\geqslant1$.
At every $p\mid\infty$, one has
$\operatorname{ord}_p(\chi(a))=-n$ and
$\operatorname{ord}_p(\mathcal R_\psi)\geqslant-1$.
Thus $\mathcal R_\psi/\chi(a)$ has order at least $n-1\geqslant0$.
It is regular there and contributes no residue to
$\mathcal{B}_a/\chi(a)=\mathcal R_\psi-\iota(a)\mathcal R_\psi/\chi(a)$.
The trace is the sum already displayed, without a factor $\Ninf$.
The affine pole bound leaves no pole away from $\Xi$ and infinity,
so the global residue theorem gives
$\operatorname{res}_\Xi(\mathcal R_\psi)=1$.
The preceding pole calculation then gives the second residue in
\eqref{eq:curve-zeta-two-residues}.

\end{proof}
\begin{theorem}
\label{thm:hg-shtuka-trace-identification} 
For every $\Ninf\geqslant1$, with
$\Lambda_\psi=\iota(A) \tilde{\pi}$ and the preceding normalizations, the following hold.
\begin{enumerate}[label=\textup{(\roman*)},leftmargin=2.2em]
\item\label{item:shtuka-trace-equality} The first torsion pairing is
\begin{equation}\label{eq:hg-shtuka-trace-identification}
 \widetilde\lambda_\psi^{(1)}=\mathcal R_\psi^{[1]}. 
\end{equation}
The corresponding zeroth pairing is
\begin{equation}\label{eq:hg-zeroth-shtuka-coordinate}
 \mathcal R_\psi
 =K_\psi(\Xi,\cdot)
 =\frac{\widetilde\lambda_\psi^{(1)}}{f_\psi}.
\end{equation}
\item\label{item:shtuka-trace-divisor} The first pairing satisfies\revblue{:}
\begin{equation}\label{eq:hg-first-pairing-geometric-bound}
 \operatorname{div}(\mathcal R_\psi^{[1]})
 \geqslant V_\psi^{(1)}-[q_\psi^{(-1)}]-[q_\psi].
\end{equation}
\end{enumerate}
\end{theorem}

\begin{proof}
\emph{\ref{item:shtuka-trace-equality}.}
 \revblue{Lemma~\ref{lem:zeroth-pairing-shtuka-line} gives
$\mathcal R_\psi\in\mathscr S_\psi$ with residue one.  Corollary~\ref{cor:shtuka-residue-uniqueness}
identifies this residue-one element as
$K_\psi(\Xi,\cdot)=\widetilde\lambda_\psi^{(1)}/f_\psi$.  Therefore
$\mathcal R_\psi=K_\psi(\Xi,\cdot)$, proving
\eqref{eq:hg-zeroth-shtuka-coordinate}.  Finally,
\eqref{eq:hg-pairing-frobenius-products} with $k=1$ gives
\eqref{eq:hg-shtuka-trace-identification}.}

\emph{\ref{item:shtuka-trace-divisor}.} Since
$\mathcal R_\psi\in\mathscr S_\psi$ by
Lemma~\ref{lem:zeroth-pairing-shtuka-line}, its divisor satisfies
$\operatorname{div}(\mathcal R_\psi)\geqslant V_\psi -[\Xi]-[q_\psi^{(-1)}]$.
Multiplication by $f_\psi$ gives
\[
 \operatorname{div}(\mathcal R_\psi^{[1]})
 \geqslant (V_\psi-[\Xi]-[q_\psi^{(-1)}])
       +(V_\psi^{(1)}-V_\psi+[\Xi]-[q_\psi])
 =V_{\psi}^{(1)}-[q_\psi^{(-1)}]-[q_\psi],
\]
which is \eqref{eq:hg-first-pairing-geometric-bound}.

\end{proof}

When $\Ninf=1$, the two possible poles in
\eqref{eq:hg-first-pairing-geometric-bound} coincide; when $\Ninf=2$,
they are the two points above infinity.  For $\Ninf\geqslant3$, the other
points of the infinity orbit are excluded from the pole support,
and zeros prescribed by $V_\psi^{(1)}$ are retained. The first pairing retains the untwisted prefactor
$-\tilde{\pi}^{-1}\mathcal L_A$; thus, in general, 
\[
 \mathcal R_\psi^{(1)}
 =- \tilde{\pi} ^{-q}\mathcal L_A^{(1)} \mathcal{H}_\psi(\tilde{\pi})^{(1)}
 \ne- \tilde{\pi}^{-1}\mathcal L_A \mathcal{H}_\psi(\tilde{\pi})^{(1)}
 =\mathcal R_\psi^{[1]}.
\]

\begin{remark}\label{rem:finite-pairing-scalar-conjugation}
The normalization is compatible with scalar conjugation.  Let
$c\in\Cinf^\times$, put $\psi^c=c\psi c^{-1}$ and
$\tilde{\pi}^c=c\tilde{\pi}$.  The residue-dual definition and
\eqref{eq:scalar-conjugation-realization} give
\begin{equation}\label{eq:finite-pairing-scalar-conjugation}
 \mathcal{H}_{\psi^c}(\tilde{\pi}^c)=c\,\mathcal{H}_\psi(\tilde{\pi}).
\end{equation}
Substituting this and $\tilde{\pi}^c=c\tilde{\pi}$ into
Definition~\ref{def:completed-differential-pairing}, with
$\mathcal L_A$ unchanged, gives
\begin{equation}\label{eq:finite-pairing-R-scalar}
 \mathcal R_{\psi^c}^{[k]}=c^{q^k-1}\mathcal R_\psi^{[k]}.
\end{equation}
The geometric data have the same $V$ and $q_\psi$, and
$f_{\psi^c}=c^{q-1}f_\psi$.  Hence
$c^{q-1}\widetilde\lambda_\psi^{(1)}$ has the required divisor and,
after division by $f_{\psi^c}$, characteristic residue one.
The uniqueness assertion of Theorem~\ref{prop:exist-shtuka} gives
\begin{equation}\label{eq:finite-pairing-shtuka-scalar}
 \widetilde\lambda_{\psi^c}^{(1)}
 =c^{q-1}\widetilde\lambda_\psi^{(1)}.
\end{equation}
Thus \eqref{eq:hg-shtuka-trace-identification} is preserved under
arbitrary scalar conjugation.
\end{remark}

Theorem~\ref{thm:hg-shtuka-trace-identification} turns the torsion
pairing into a computation of $\widetilde\lambda_\psi^{(1)}$.
All coefficients come from a finite algebra over the field of
definition of the Drinfeld action.

The traces can be computed without finding any torsion point.  For example,
write \(\psi_b(Z)=\sum_{i=0}^{D}B_iZ^{q^i}\), with \(B_0=\iota(b)\).
If \(s_i=\operatorname{Tr}_{\mathscr T_b/H}(Z^{q^i-1})\), then
\begin{equation}
\label{eq:hg-torsion-moment-recursion}
 s_0=-1,\qquad s_i=0\ (0<i<D),\qquad
 s_i=-\sum_{j=0}^{D-1}
       \left(\frac{B_j}{B_D}\right)^{q^{i-D}}s_{i-D+j}\quad(i\geqslant D).
\end{equation}
For the initial vanishing, over a splitting field put
\[
 R(T)=\prod_{0\ne x\in\psi[b]}(1-xT)
     =1+\sum_{j=0}^{D-1}\frac{B_j}{B_D}T^{q^D-q^j}.
\]
Logarithmic differentiation gives
\[
 \sum_{n\geqslant1}\left(\sum_{0\ne x\in\psi[b]}x^n\right)T^n
 =-\frac{TR'(T)}{R(T)}
 =\frac{(B_0/B_D)T^{q^D-1}}{R(T)}.
\]
Here the derivatives of the terms with \(j\geqslant1\) vanish in the
characteristic.  All power sums of positive degree less than \(q^D-1\)
therefore vanish, giving \(s_i=0\) for \(0<i<D\); also
\(s_0=q^D-1=-1\).  The recursion follows by
raising \(\psi_b(x)=0\) to \(q^{i-D}\), dividing by \(x\), and summing
over nonzero torsion points.  Thus a trace of \(\psi_c(Z)/Z\) is simply
\(\sum_i(\psi_c)_i s_i\).


\subsection{Pellarin's period identity}

\begin{theorem}[Pellarin's period identity]\label{thm:pellarin-period-identity}
This identity belongs to the family of functional identities of
\cite{AP14}.
For $\hbar\in\operatorname{Hom}_A(\Omega_A,A)$, put
$\jmath_\hbar= \tilde{\pi}\,\iota\circ\hbar$. Then:
\begin{enumerate}[label=\textup{(\roman*)},leftmargin=2.2em]
\item\label{item:pellarin-first}
The first-twist identity is
\begin{equation}\label{eq:hg-first-twist-period-identity}
\mathcal L_A\omega_{\jmath_\hbar}^{(1)}
=- \tilde{\pi}\,\hbar(\widetilde\lambda_\psi^{(1)})= \tilde{\pi}\Delta_\hbar,
\qquad \Delta_\hbar\in H\otimes_{\Fq}A.
\end{equation}
For the basis in Theorem~ {\ref{thm:hg-finite-torsion-trace-value}} with $k=1$,
\begin{equation}\label{eq:hg-first-twist-finite-factor}
\Delta_\hbar=\frac1{\iota(b)}\sum_j
\operatorname{Tr}_{\mathscr T_b/H}
 \left(\frac{\psi_{c_j}(Z)^q}{Z}\right)\chi(\hbar(\mu_j)).
\end{equation}
\item\label{item:pellarin-zeroth}
The untwisted identity is
\begin{equation}\label{eq:hg-zeroth-twist-period-identity}
\mathcal L_A\omega_{\jmath_\hbar}
=- \tilde{\pi}\,\hbar(K_\psi(\Xi,\cdot)).
\end{equation}
\end{enumerate}
\end{theorem}

\begin{proof}
\emph{\ref{item:pellarin-first}.}
Maintaining the notation of Section~\ref{sec:finite-special-values},
\[
\mathcal R_\psi^{[1]}=- \tilde{\pi}^{-1}\mathcal L_A \mathcal{H}_\psi(\tilde{\pi})^{(1)}.
\]
Contraction by $1\otimes\hbar$ commutes with coefficient Frobenius.

Using $\hbar( \mathcal{H}_\psi(\tilde{\pi}))=\omega_{\jmath_\hbar}$ in Corollary~\ref{cor:hg-shtuka-trace-identification} and
Theorem~\ref{thm:hg-shtuka-trace-identification} gives
\eqref{eq:hg-first-twist-period-identity}. The finite formula follows by
contracting \eqref{eq:Rk} at $k=1$ and changing sign. Every $\mu_j$ is
regular on $U$, so $\hbar(\mu_j)\in A$.

\emph{\ref{item:pellarin-zeroth}.}
Apply the same contraction to
\[
\mathcal R_\psi=- \tilde{\pi}^{-1}\mathcal L_A \mathcal{H}_\psi(\tilde{\pi})=K_\psi(\Xi,\cdot).
\]
This proves \eqref{eq:hg-zeroth-twist-period-identity}. Its rationality
follows from the finite formula for $\mathcal R_\psi$.
\end{proof}

\begin{example} 
\label{ex:hg-carlitz-sign-check}
Let \(A=\Fq[t]\), \(\iota(t)=\theta\), and
\(\psi_t(Z)=\theta Z+Z^q\).  Choose \(a=b=t\), \(\nu_1=\rdt\),
and \(r_1=-1/t\), so \(c_1=-1\).  The trace in
\eqref{eq:hg-explicit-torsion-trace-coefficients} equals \(-\theta\),
and \(\mathcal R_\psi=\rdt/(t-\theta)\).  With the residue convention
\eqref{eq:residue-dual-coordinate-line}, \( \mathcal{H}_\psi(\tilde{\pi})=-\omega_C\rdt\).
Taking \(\hbar(\rdt)=-1\) therefore recovers
\[
 \left(\sum_{0\ne a\in\Fq[t]}\frac{a(t)}{a(\theta)}\right)
 \frac{\omega_C}{\tilde{\pi}_C}=\frac1{t-\theta}.
\]
The sign agrees with the all-element convention rather than the
monic convention.
\end{example}

\subsection{Finite values for an arbitrary source lattice}
\label{subsec:finite-values-isogeny-transport}
To apply the identification to a possibly nonprincipal source lattice
$\Lambda_\varphi$, choose a separable
isogeny datum $J:\varphi\to\psi$ with free target lattice, and write
\[
 \Lambda_\psi=\iota(A) \tilde{\pi},
 \qquad z_J= \tilde{\pi}/h_J(\Xi).
\]
Such a target exists for every rank-one source by
Corollary~\ref{cor:isogeny-direct-period-map-exact-scope}.  Take a common
field $H\subset\Cinf$ of definition of the two actions and the rational
coordinate $h_J$.  The finite pairing remains the one attached to the
principal-period target $\psi$.

\begin{proposition}\label{prop:finite-torsion-isogeny-kernels}
For every $a\in A$, the normalized kernel value is the rational
differential over $H$ given by\revblue{:}
\begin{equation}\label{eq:finite-trace-isogeny-kernel-value}
 \frac{\mathcal L_A}{ \tilde{\pi}}\mathcal G_J(\iota(a)z_J)
 =-\frac{\chi(a)}{h_J(\two)}K_\psi(\Xi,\two).
\end{equation}
\end{proposition}

\begin{proof}
Let $ \mathcal{H}_\psi(\tilde{\pi})$ be the canonical period differential for $\psi$.
The target comparison \eqref{eq:J-kernel-target-comparison} gives
\[
 \mathcal G_J(\iota(a)z_J)
 =h_J(\two)^{-1}\mathcal H_\psi(\iota(a) \tilde{\pi})
 =\frac{\chi(a)}{h_J(\two)} \mathcal{H}_\psi(\tilde{\pi}).
\]
The second equality follows from the $A$-linearity of the
characteristic-residue isomorphism.  Theorem
\ref{thm:hg-finite-torsion-trace-value}, applied to $\psi$, identifies
$ \tilde{\pi}^{-1}\mathcal L_A \mathcal{H}_\psi(\tilde{\pi})=-\mathcal R_\psi$.
This proves the finite formula and its independence of the presentation.
\end{proof}

In particular, when $ \Omega_A=\ii_J\nu_J$ (or equivalently $ \sfm(\varphi )  $ is free), the unit anchor $\jmath_*$ of
Proposition~\ref{prop:isogeny-direct-period-map} turns this into a scalar
formula for the source special function.  Taking $a=1$ and applying
\eqref{eq:primitive-G-identity} gives the following result.

\begin{theorem}\label{thm:finite-value-isogeny-coordinate}
Let $\omega_{\jmath_*}=\mathcal G_J(z_J)/\nu_J$ be the unit special
function attached to the unit anchor $\jmath_*$. Then
\begin{enumerate}[label=\textup{(\roman*)},leftmargin=2.2em]
\item\label{item:source-finite-value}
\begin{equation}\label{eq:finite-trace-source-special-value}
 \mathcal L_A\frac{\omega_{\jmath_*}}{ \tilde{\pi}}
 =-\frac{K_\psi(\Xi, \two)}{h_J\nu_J}.
\end{equation}
\item\label{item:source-first-finite-value}
\begin{equation}\label{eq:finite-trace-source-special-value2}
 \mathcal L_A\frac{\omega_{\jmath_*}^{(1)}}{ \tilde{\pi}}
 = - \delta_J = -  \frac{\tilde{\lambda}_\psi^{(1)}}{h_J^{(1)}    \nu_J }\in H(X) .
\end{equation}
\end{enumerate}
\end{theorem}
\begin{proof}
\emph{\ref{item:source-finite-value}.}
Set $a=1$ in \eqref{eq:finite-trace-isogeny-kernel-value} and divide by
$\nu_J$, using $\mathcal G_J(z_J)=\nu_J\omega_{\jmath_*}$.

\emph{\ref{item:source-first-finite-value}.}
Multiply \ref{item:source-finite-value} by $f_\varphi$ and use
$\omega_{\jmath_*}^{(1)}=f_\varphi\omega_{\jmath_*}$.
The intertwiner equation
$f_\varphi/h_J=f_\psi/h_J^{(1)}$ and
$f_\psi K_\psi(\Xi,\cdot)=\widetilde\lambda_\psi^{(1)}$ give the result.
Neither $\mathcal L_A$ nor $ \tilde{\pi}$ is twisted in this argument.
\end{proof}

\begin{remark}
 \label{rem:curve-zeta-comparison}
Assume $\Ninf=1$ with a fixed sign function; $\mathcal L_{A_+}$ of \eqref{eq:curve-zeta-monic-convention} is $Z_{X,A}$ in the convention of \cite[\S5.2]{Pel26}.
For a unit-source isogeny datum the finite-trace identity becomes
$\mathcal L_{A_+}\omega_{\jmath_*}^{(1)}/\tilde{\pi}=\delta_J$.
 Comparison with \cite[Theorem~5.3]{Pel26} requires matching
the modules, shtuka data, and period normalizations.  In particular,
if the source $\varphi$ is fixed and only the target is replaced by
$\psi^c=c\psi c^{-1}$ for $c\in\Cinf^\times$, the corresponding isogeny is
$\eta_{J^c}=c\eta_J$, with
\[
 h_{J^c}=c h_J,\qquad \tilde\pi^c=c\tilde\pi,\qquad
 z_{J^c}=z_J,\qquad
 \omega_{\jmath_*}^c=\omega_{\jmath_*}.
\]
Here the kernel ideal and the choice of $\nu_J$ are unchanged.
Equations~\eqref{eq:scalar-conjugation-realization}
and~\eqref{eq:szego-shtuka} give
$\widetilde\lambda_{\psi^c}^{(1)}
=c^{q-1}\widetilde\lambda_\psi^{(1)}$, hence
$\delta_{J^c}=c^{-1}\delta_J$, as required by the displayed
finite-trace identity.  Together with Corollary~\ref{cor:N=1},
these transformation laws specify the divisor and scalar data for
comparison with \cite[Theorem~5.3 and Lemma~6.4]{Pel26}. 
\end{remark}

\begin{example}
 \label{ex:hg-first-twist-green-papanikolas}
For an elliptic curve with a rational infinite point,
Green--Papanikolas \cite[Theorem~7.1 and equations~(5), (18)]{GP18}
prove, in their normalization (their $L_+(A;1)$ is our $\mathcal L_{A_+}$),
$\omega_\rho^{(1)} \mathcal L_{A_+}/\tilde{\pi}_\rho =- \delta^{(1)}$
with $\omega_\rho^{(1)}=f_\rho \omega_\rho$ and
$\delta=t-\alpha$,
where \((\alpha,\beta)\) is the Drinfeld point and \(f_\rho\) is the
rational shtuka function.   
\end{example}

\begin{example}[The HHY datum: explicit finite values]
\label{ex:HHY-finite-values}
Retain the HHY isogeny datum
$J=[\eta_J;h_J;\mathfrak i_\infty^2,1]:\Psi\to\Phi$ of
Example~\ref{ex:primitive-G-HHY}; in particular $\nu_J=\rd t$,
$z_J=(\theta-\eta)^2\tilde{\pi}_\Phi$, $h_J(\Xi)=(\theta-\eta)^{-2}$, and
\[
 h_J(t)=\frac{\theta-\eta^q}{(\theta-\eta)(t-\eta)(t-\eta^q)}.
\]
The intertwiner equation~\eqref{eq:intertwiner-difference} gives the
target shtuka function
\[
 f_\Phi(t)
 =\frac{(\theta^q-\eta^{q^2})(\theta-t)}
        {(\theta^q-\eta^q)(\theta-\eta^q)(t-\eta^{q^2})},
\]
so that $V_\Phi=0$ and $q_\Phi=\eta^{q^2}$ in the $t$-coordinate.  The
differential
\[
 \frac{(\theta-\eta^q)\,\rd t_{\two}}
      {(t_{\two}-\eta^q)(t_{\two}-\theta)}
\]
has a simple pole of residue $1$ at $t_{\two}=\theta$, a simple pole of
residue $-1$ at $t_{\two}=\eta^q=q_\Phi^{(-1)}$, and no other poles;
by the uniqueness in Theorem~\ref{prop:exist-shtuka} it equals
$K_\Phi(\theta,\two)$.  Therefore
\[
 \widetilde\lambda_\Phi^{(1)}
 =f_\Phi\,K_\Phi(\theta,\cdot)
 =-\frac{\theta^q-\eta^{q^2}}{\theta^q-\eta^q}\,
 \frac{\rd t}{(t-\eta^q)(t-\eta^{q^2})}
 =-h_J^{(1)}\,\rd t,
\]
and the scalar forcing factor is
$\delta_J=\widetilde\lambda_\Phi^{(1)}/(h_J^{(1)}\,\rd t)=-1$.

The lattice $\Lambda_\Phi=\iota(A)\tilde{\pi}_\Phi$ is principal, and
$\Omega_A=\mathfrak i_\infty^2\,\rd t=\mathfrak i_J\nu_J$ with
$\nu_J=\rd t$: the datum satisfies every hypothesis of
Theorem~\ref{thm:finite-value-isogeny-coordinate}.  With $\psi=\Phi$,
$\tilde{\pi}=\tilde{\pi}_\Phi$ and
$\omega_{\jmath_*}=\mathcal G_J(z_J)/\rd t$,
part~\ref{item:source-finite-value} gives the explicit identity
\begin{equation}\label{eq:HHY-finite-value}
 \mathcal L_A\,\frac{\omega_{\jmath_*}}{\tilde{\pi}_\Phi}
 =-\frac{K_\Phi(\theta,\cdot)}{h_J\,\rd t}
 =\frac{(\theta-\eta)(t-\eta)}{\theta-t},
\end{equation}
 This is a rational function on $X$.  Its only pole is a simple pole at
$\Xi$, and its zero at $t=\eta$ lies at a removed point. Moreover, Part~\ref{item:source-first-finite-value} then gives: 
\begin{equation}\label{eq:HHY-first-finite-value}
 \mathcal L_A\,\frac{\omega_{\jmath_*}^{(1)}}{\tilde{\pi}_\Phi}
 =-\delta_J=1.
\end{equation}
 Example~\ref{ex:hg-carlitz-sign-check} is the $\Ninf=1$
counterpart of this computation.
\end{example}

\subsection*{Acknowledgments}
ChatGPT and OpenAI Codex (OpenAI) were used during manuscript preparation for structural organization, English-language refinement, and LaTeX formatting. These tools were also employed to support preliminary discussions of proof strategies, logical consistency checks, and alternative argument formulations. All content, revisions, and mathematical reasoning presented in the manuscript have been fully reviewed, verified, and validated by the authors. The authors take exclusive and full responsibility for all mathematical statements, proofs, analytical arguments, references, and every part of the final manuscript, and ensure full adherence to all publication ethics standards.

\appendix
\section{Notation}\label{app:notation}

The following notation is maintained throughout the paper.
\begin{itemize}[leftmargin=2.2em]
\item$\Fq$: the finite field with $q$ elements, of characteristic $p$.
\item$X$: a smooth, projective, geometrically connected curve over
$\Fq$ with a fixed closed point $\infty$; $\Ninf=\deg\infty$.
\item$U=X\setminus\{\infty\}$ is the affine curve,
$A=\Gamma(U,\mathcal O_X)$ its ring of regular functions, and
$K=\Fq(X)$ the function field.
\item$K_\infty$: the completion of $K$ at $\infty$;
$\Cinf=\widehat{\overline{K_\infty}}$: the completed algebraic closure;
$\iota\colon K_\infty\hookrightarrow\Cinf$ is a fixed isometric
$\Fq$-embedding.
\item$\Xi\in X(\Cinf)$: the characteristic point, defined by
$c(\Xi)=\iota(c)$ for $c\in K$.
\item$\one$ and $\two$: placeholders for the first and second
variable of a bivariate kernel or function.
\item$x^{(i)}$: the $i$-th Frobenius twist of $x$, raising the
$\Cinf$-coefficients to the $q^i$-th power and fixing the $A$-factor.
\item$\varphi$ and $\psi$: rank-one Drinfeld $A$-modules in
fixed geometric realizations $A\to\Cinf\{\tau\}$.
\item$\mathcal M_\varphi$: the motive module of $\varphi$, an
$A$-module of $\Fq$-linear polynomials.
\item$\exp_\psi$ and $\log_\psi$: the exponential and the
logarithm of $\psi$.
\item$\Lambda_\varphi$ and $\Lambda_\psi$: the period
lattices; $\tilde\pi$ denotes a chosen generator when
$\Lambda_\psi=\iota(A)\tilde\pi$.
\item$f_\varphi$: the shtuka function of $\varphi$;
$F_\varphi^{[n]}=f_\varphi f_\varphi^{(1)}\cdots f_\varphi^{(n-1)}$.
\item$\mathbb T_A=\Cinf\widehat\otimes_{\Fq}A$: the Tate
algebra, with its Gauss norm.
\item$\chi(a)=1\otimes a=a(\two)$: the geometric copy of
$a\in A$.
\item$\Omega_A$: the $A$-module of regular differentials on
$U$;  $\Omega_A^{\rat}=K\otimes_A\Omega_A$ and
$\widehat\Omega_A=K_\infty\otimes_A\Omega_A$.
\item$\overline\Omega=\widehat\Omega_A/\Omega_A$: the
residue-class quotient.
\item $\sfm(\varphi)$: the
$A$-module of special functions of $\varphi$.
\item$\omega_\jmath$: the residue-dual special function
attached to an anchor $\jmath\colon\Omega_A\to\Lambda_\varphi$,
normalized by $\operatorname{res}_\Xi(\nu\omega_\jmath)=\jmath(\nu)$.
\item$\mathcal H_\psi$: the 
 Anderson generating differential of $\psi$.
\item$\mathcal G_J$: the Anderson generating differential of
the isogeny datum $J$.
\item$K_\varphi(\one,\two)$ and $K_J(\one,\two)$: the
canonical and isogeny-adapted Cauchy coordinates.
\item $\GM_{u;\widehat{\mathfrak b}}$: the lift-dependent
Gazda--Maurischat transform; $\GM_u^\Omega$: its canonical
differential assembly on periods.
\item $\operatorname{Log}_{\alpha,f_\psi}$: the deformed
logarithm; $\mathcal{P}_{\alpha,f_\psi}(\two;Z)$: the Papanikolas-type
generating series.
\item$\hbar\in\operatorname{Hom}_A(\Omega_A,A)$:  a
contraction, acting as $1\otimes\hbar$ on
$\Cinf\widehat\otimes_{\Fq}\Omega_A$; its $K_\infty$-linear extension
induces $\overline\Omega\to K_\infty/A$.
\item$\langle c,\nu\rangle_\infty=
\operatorname{Tr}_{\mathbb F_\infty/\Fq}\operatorname{res}_\infty(c\nu)$
: the residue pairing at $\infty$.
\item$\operatorname{res}_\Xi$: the characteristic residue,
which identifies $\sfm(\psi)\otimes_A\Omega_A$ with $\Lambda_\psi$ on
special-function tensors.
\item$\mathcal L_A=\sum_{0\ne a\in A}\chi(a)/\iota(a)$: the
all-element special-value series; $\mathcal L_{A_+}$ is its monic
analogue when $\Ninf=1$ with a fixed sign function.
\item$\mathcal R_\psi^{[k]}=-\tilde\pi^{-1}\mathcal L_A
\mathcal H_\psi(\tilde\pi)^{(k)}$: the $k$-th completed differential
pairing; $\mathcal R_\psi=\mathcal R_\psi^{[0]}$.
\item$\widetilde\lambda_\psi$: 
 the normalized shtuka
differential of $\psi$ (Definition~\ref{def:shtuka-diff}), written $\widetilde\lambda_\psi=\alpha\nu$ in the deformed-logarithm formulas; $\widetilde\lambda_\psi^{(1)}$ is its
first twist.
\item $J=[\eta_J;h_J;\mathfrak i_J,\mu_J]\colon\varphi\to\psi$:
an isogeny datum, with rational motive multiplier $h_J$.
When $\Lambda_\psi=\iota(A)\tilde\pi$, set $z_J=\tilde\pi/h_J(\Xi)$;
if moreover $[\Lambda_\varphi]=[\Omega_A]$, choose
$\nu_J\in\Omega_A^{\rat}$ with $\Omega_A=\mathfrak i_J\nu_J$.
\item$\delta_J=\widetilde\lambda_\psi^{(1)}/
(h_J^{(1)}\nu_J)$: the scalar forcing factor;
$\widetilde\delta_J=\widetilde\lambda_\psi^{(1)}/h_J^{(1)}$: the
differential forcing term.
\item $e_{\LL}$: the tautological rational frame of an acyclic divisor
line $\LL$.
\item $S_{\LL}=K_{\LL}(\one,\two)e_{\LL}(\one)\otimes e_{\LL}(\two)^{-1}$:
the normalized Szeg\H o kernel, where $K_{\LL}$ denotes its tautological Cauchy coordinate.
\end{itemize}

\begingroup
\renewcommand{\bibliofont}{\footnotesize\baselineskip=10pt\selectfont}
\bibliographystyle{amsplain} 
\bibliography{paper} 
\endgroup

\end{document}